\documentclass{article}[12pt]
\usepackage{fullpage}
\usepackage{graphics}
 \usepackage{graphicx}	
\usepackage{amssymb}
 \usepackage{amsmath}
 \usepackage{amsthm} 
 \usepackage{lmodern} 
 \usepackage{hyperref}
\usepackage{pdfsync}
\usepackage{yhmath}
\usepackage{amssymb,mathtools}
\usepackage{stmaryrd}
\usepackage{palatino,tipa,graphicx}

\usepackage{soul}
\usepackage{color}

  \usepackage{xcolor,lmodern}

  \usepackage{calc}
\newsavebox\CBox
\newcommand\hcancel[2][0.5pt]{%
  \ifmmode\sbox\CBox{$#2$}\else\sbox\CBox{#2}\fi%
  \makebox[0pt][l]{\usebox\CBox}%
  \rule[0.5\ht\CBox-#1/2]{\wd\CBox}{#1}}

 \newcommand{\dt}{\delta} 
 \renewcommand{\th}{\theta}
 \newcommand{\pd}{\partial} 
 \renewcommand{\H}{\mathcal{H}} 
 
 \newcommand{\Lra}{\Longrightarrow}
 
 \newcommand{\ggm}{\Gamma}
 \newcommand{\R}{{\mathbb{R}}} 
 \newcommand{\8}{\infty} 
 \renewcommand{\~}{\tilde}

\renewcommand{\H}{\mathcal{H}} 

\newcommand{\PP}{\mathcal{P}}

\newcommand{\lra}{\longrightarrow} 
\newcommand{\wra}{\rightharpoonup}

\renewcommand{\~}{\widetilde}

\newcommand{\rhu}{\rightharpoonup}
\newcommand{\wto}{\rhu}

\newcommand{\al}{\alpha}
 \newcommand{\bt}{\beta}
\newcommand{\sm}{\Sigma}

\renewcommand{\vartheta}{\Theta}

\newcommand{\RR}{{\R^2}}
\newcommand{\EE}{\mathcal{G}}
\newcommand{\GG}{\EE_\sigma}
\newcommand{\Om}{\Omega}
\newcommand{\TT}{{\mathbb{T}^2}}

\newcommand{\NN}{{\mathbb{N}}}

 \newtheorem{theorem}{Theorem}[section]
\newtheorem{lemma}[theorem]{Lemma}
\newtheorem{proposition}[theorem]{Proposition}
\newtheorem{corollary}[theorem]{Corollary}

\newtheorem{counter}{Counter}[section]
\newtheorem{lem}[counter]{Lemma}
\newtheorem{defn}[counter]{Definition}
\newtheorem{thm}[counter]{Theorem}
\newtheorem{remark}[counter]{Remark}

\newtheorem{cor}[counter]{Corollary}

\DeclareMathOperator{\diam}{diam}
	
\newcommand{\vep}{\varepsilon} 

\usepackage{stackengine}
\stackMath
\usepackage{mathtools}
\usepackage{esint}

\title{ 
Core shells and double bubbles \\ in a weighted nonlocal isoperimetric problem
 } 

\author{Stanley Alama
\thanks{Department of Mathematics and Statistics, McMaster University. E-mail: alama@mcmaster.ca} 
 \qquad Lia Bronsard 
 \thanks{Department of Mathematics and Statistics, McMaster University. E-mail: bronsard@mcmaster.ca} 
 \qquad Xinyang Lu
 \thanks{Department of Mathematical Sciences, Lakehead University. Email: xlu8@lakeheadu.ca}
 \qquad Chong Wang
 \thanks{Department of Mathematics, Washington and Lee University. Email: cwang@wlu.edu}
}

\begin{document}

\date{}
\maketitle

\begin{abstract}

We consider a sharp-interface model of $ABC$ triblock copolymers, for which the surface tension $\sigma_{ij}$ across the interface separating phase $i$ from phase $j$ may depend on the components.  We study global minimizers of the associated ternary  local isoperimetric problem in $\R^2$, and show how the geometry of minimizers changes with the surface tensions $\sigma_{ij}$, varying from symmetric double-bubbles for equal surface tensions, through asymmetric double bubbles, to core shells as the values of $\sigma_{ij}$ become more disparate.
Then we consider the effect of nonlocal interactions in a droplet scaling regime, in which vanishingly small particles of two phases are distributed in a sea of the third phase.  
We are particularly interested in a degenerate case of $\sigma_{ij}$ in which minimizers exhibit core shell geometry, as this phase configuration is expected on physical grounds in nonlocal ternary systems.

\medskip

\end{abstract} 

 \numberwithin{equation}{section}

\section{Introduction}

In this paper we continue our study of ternary systems, in which three constituents or phases interact through both short range attractive and long range repulsive forces.  A prominent example of such ternary systems are the $ABC$ triblock copolymers, linear chains of molecules consisting of three subchains, joined covalently to each other. A subchain of type $A$ monomer is connected to one of type $B$, which in turn is connected to another subchain of type $C$ monomer. Because of the repulsive forces between different types of monomers, different types of subchain tend to segregate. However, since subchains are chemically bonded in molecules, segregation can lead to a phase separation only at microscopic level, 
where $A, B$ and $C$-rich micro-domains emerge, forming morphological phases, many of which have been observed experimentally: see Figure \ref{csc}.

\begin{figure}[!htb]
\centering
 \includegraphics[width=3.8cm]{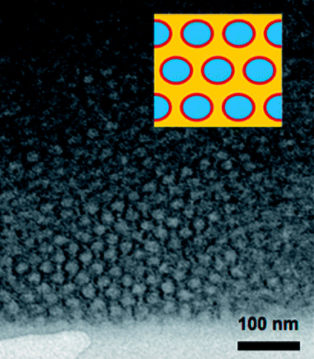} 
 \caption{Cross-sectional TEM image of PMMA-{\it b}-PVCa-{\it b}-PSt triblock copolymer thin flim. The inset image shows schematic illustration of Core-Shell Cylindrical phase (blue: Pst, red: PVCa, yellow: PMMA) \cite{ynn}.
Reproduced from the Royal Society of Chemistry 2018  \href{https://pubs.rsc.org/en/content/articlepdf/2018/ra/c8ra00630j}{(Link to this Open Access Article).
}
 }
 \label{csc}
\end{figure}

A triblock copolymer can be described as a stable critical point of an energy derived with Nakazawa and Ohta's density functional theory for triblock copolymers \cite{microphase, lameRW}.  This is a   diffuse interface model, a nonlocal version of the vector-valued Cahn-Hilliard energy.  In this paper we consider periodic configurations in two dimensions, so we pose our problem in the flat unit torus $\TT=\left[-\frac12,\frac12\right]^2$.
 The system is determined via a vector-valued density function $u = (u_1, u_2,u_0) \in L^1(\TT, \mathbb{R}^3)$, in which each scalar function  $u_i, i=0, 1, 2$, gives the density of one constituent of the mixture. The free energy of the system is

\begin{eqnarray} \label{energyuep}
\mathcal{E}^{\epsilon} (u) :=   \frac{1}{2}   \int_{\TT}  \left [ \frac{ \epsilon^2 }{  2 } |\nabla u |^2  + W(u) \right ] dx +
  \sum_{i,j = 1}^2  \frac{\epsilon \gamma_{ij}}{2} \int_{\TT} \int_{\TT} G_{}(x-y)\; u_i (x) \; u_j (y) dx dy.
\end{eqnarray}
$W$ is a triple-well potential which achieves minimum value $0$ at exactly three points: $\alpha_1 = (1,0,0), \alpha_2 = (0,1,0)$ and $\alpha_0 = (0,0,1)$, and $G$ is the Laplace Green's function on $\TT$, of mean zero.
The functional $\mathcal{E}^{\epsilon}$ is defined on
 \begin{eqnarray}
\left \{  u = (u_1,u_2, u_0)\, : \, u_i \in H^1(\TT; \mathbb{R}) , \ i=0,1,2;  \ \sum_{i=0}^2 u_i = 1; \  \frac{1}{|\TT|} \int_{\TT} u_i (x) dx = M_i, i = 1, 2 \right \}.
\end{eqnarray}
Using $\Gamma$-convergence \cite{blendCR,miniRW}, the variational structure of $\mathcal{E}^{\epsilon}$ is connected to minimization of an associated sharp-interface model,
\begin{equation} \label{energyu}
\mathcal{E} (u) :=  \sum_{0 \leq i < j \leq 2} \sigma_{ij}   \mathcal{H}^1 (\partial \Omega_i \cap \partial \Omega_j)+
  \sum_{i,j = 1}^2  \frac{ \gamma_{ij}}{2} \int_{\TT} \int_{\TT} G_{}(x-y)\; u_i (x) \; u_j (y) dx dy,
\end{equation}
where the densities $u_i$ are replaced by phase domains described by characteristic functions $u_i=\chi_{\Om_i}$ of sets $\Om_i$, $i=0,1,2$, of finite perimeter.

The constants $\sigma_{ij}$, $i\neq j$, represent {\it surface tension} along the interfaces separating the phase domains, and they are positive, material-dependent constants.  Their values are determined from the $\Gamma$-limit of the vector-valued Cahn-Hilliard part of $\mathcal{E}^{\epsilon}$ \cite{baldo,sternberg}, as a geodesic distance in a degenerate Riemannian metric induced on $\R^2$ by $\sqrt{W(u)}$, 
\begin{eqnarray} \label{sigij}
\sigma_{ij} = \inf \left \{   \sqrt{2}   \int_{0}^1  \sqrt{W (\zeta(t))} |\zeta^{\prime}(t)| dt: \zeta \in C^1 ([0,1]; \mathbb{R}^3), \zeta(0) = \alpha_i, \zeta(1) = \alpha_j \right \}.
\end{eqnarray}
Conversely, given positive constants $\sigma_{ij}=\sigma_{ji}$, $i\neq j$, one might want to engineer a potential $W$ for which $\sigma_{ij}$ give the surface tensions along each interface in $\mathcal{E}(u)$.  This may or may not be possible, as the definition \eqref{sigij} imposes a necessary condition for $\sigma_{ij}$ to arise as coefficients of the perimeter in $\mathcal{E}(u)$ in any $\Gamma$-limit of $\mathcal{E}^{\epsilon}$, in the form of a set of {\it triangle inequalities,}
\begin{equation}\label{triangle}
\sigma_{ij} \leq \sigma_{ik} + \sigma_{kj}, \quad 0 \leq i,j,k \leq 2.
\end{equation}

Our previous paper \cite{ablw} discusses the special case $\sigma_{01} = \sigma_{02} = \sigma_{12}= 1$, which arises (for instance) when $W$ is symmetric with respect to permutations of the $u_i$-axes.  In this paper we consider the case of unequal surface tensions $\sigma_{ij}$, in which the geometry of minimizers can be quite different.  Although it is not motivated by $\Gamma$-convergence, we also consider a case where the triangle inequalities \eqref{triangle} are violated,  $\sigma_{02} > \sigma_{01} + \sigma_{12}$.

When the surface tensions $\sigma_{ij}$ obey the triangle inequalities \eqref{triangle}, the variational problem \eqref{energyu} may conveniently be posed with characteristic functions lying in BV spaces,
\begin{eqnarray}
\left \{  u = (u_1,u_2, u_0) : u_i \in BV(\TT; \{0,1\}) ,  i=0,1,2;  \sum_{i=0}^2 u_i = 1;   \frac{1}{|\TT|} \int_{\TT} u_i (x) dx = M_i, i = 1, 2 \right \}.
\end{eqnarray}
Here $ BV(\TT; \{0,1\})$ is the space of functions of bounded variation that only take two values: $0$ and $1$. Thus, each $u_i = \chi_{\Omega_i}$ where $\Omega_1, \Omega_2, \Omega_0$ is a partition of $\TT$. The weighted perimeter is then re-expressed in terms of the total variations of the components $u_i$,
\begin{equation} \label{bvperim}
P_\sigma(\Om_1, \Om_2,\Om_0)=
  \sum_{0 \leq i < j \leq 2} \sigma_{ij}   \mathcal{H}^1 (\partial \Omega_i \cap \partial \Omega_j)
= \frac{1}{2} \sum_{i=0}^2 \beta_i \int_{\mathbb{T}^2} |\nabla u_i | , 
\end{equation}
with weights
\begin{equation}\label{betasigma}
\beta_1  = \sigma_{01} + \sigma_{12} - \sigma_{02},  \qquad 
\beta_2  = \sigma_{02} + \sigma_{12} - \sigma_{01},  \qquad 
\beta_0  = \sigma_{01} + \sigma_{02} - \sigma_{12}.
\end{equation}
When \eqref{triangle} holds, each $\beta_i\ge 0$.  Moreover, at most one of $\beta_i=0$, since we assume each $\sigma_{ij}>0$.  We note that since $u_0=1-u_1-u_2$, the state of the system is completely determined by the pair $(u_1,u_2)$, or indeed any pair of the triple $(u_1,u_2,u_0)$.
In particular, the weighted perimeter $P_\sigma$ is equivalent  to the standard BV norm on the cluster $(\Om_1,\Om_2,\Om_0)$, and thus is coercive and lower semicontinuous with respect to $L^1$ convergence in this case.
  The energy may then be expressed as
\begin{eqnarray} 
\mathcal{E} (u) =  \frac{1}{2} \sum_{i=0}^2 \beta_i \int_{\mathbb{T}^2} |\nabla u_i |+
  \sum_{i,j = 1}^2  \frac{ \gamma_{ij}}{2} \int_{\mathbb{T}^2} \int_{\mathbb{T}^2} G_{}(x-y)\; u_i (x) \; u_j (y) dx dy. \notag
\end{eqnarray}
Our goal in this paper is to characterize global minimizers of this functional, both in the absence of the nonlocal interaction (i.e., $\gamma_{ij}=0$) and in an appropriate ``droplet scale'' limit, in which minimizers form a very dilute lattice of particles of two minority phases in a sea of the third phase \cite{bi1,ABCT1,ablw}.

\subsection*{The local isoperimetric problem }

\begin{figure}[ht]
\centering
\includegraphics[scale=0.175]{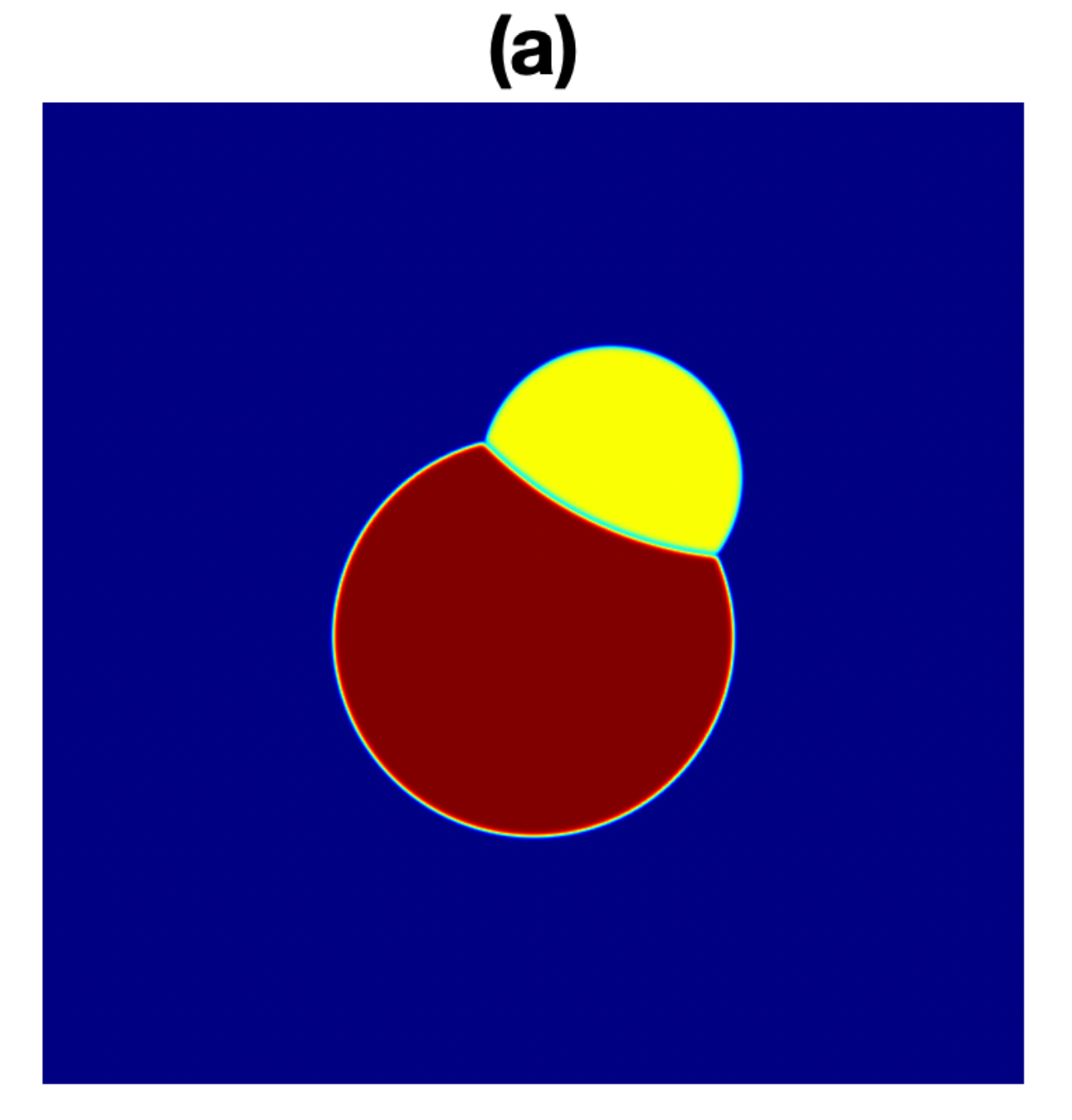}
\includegraphics[scale=0.175]{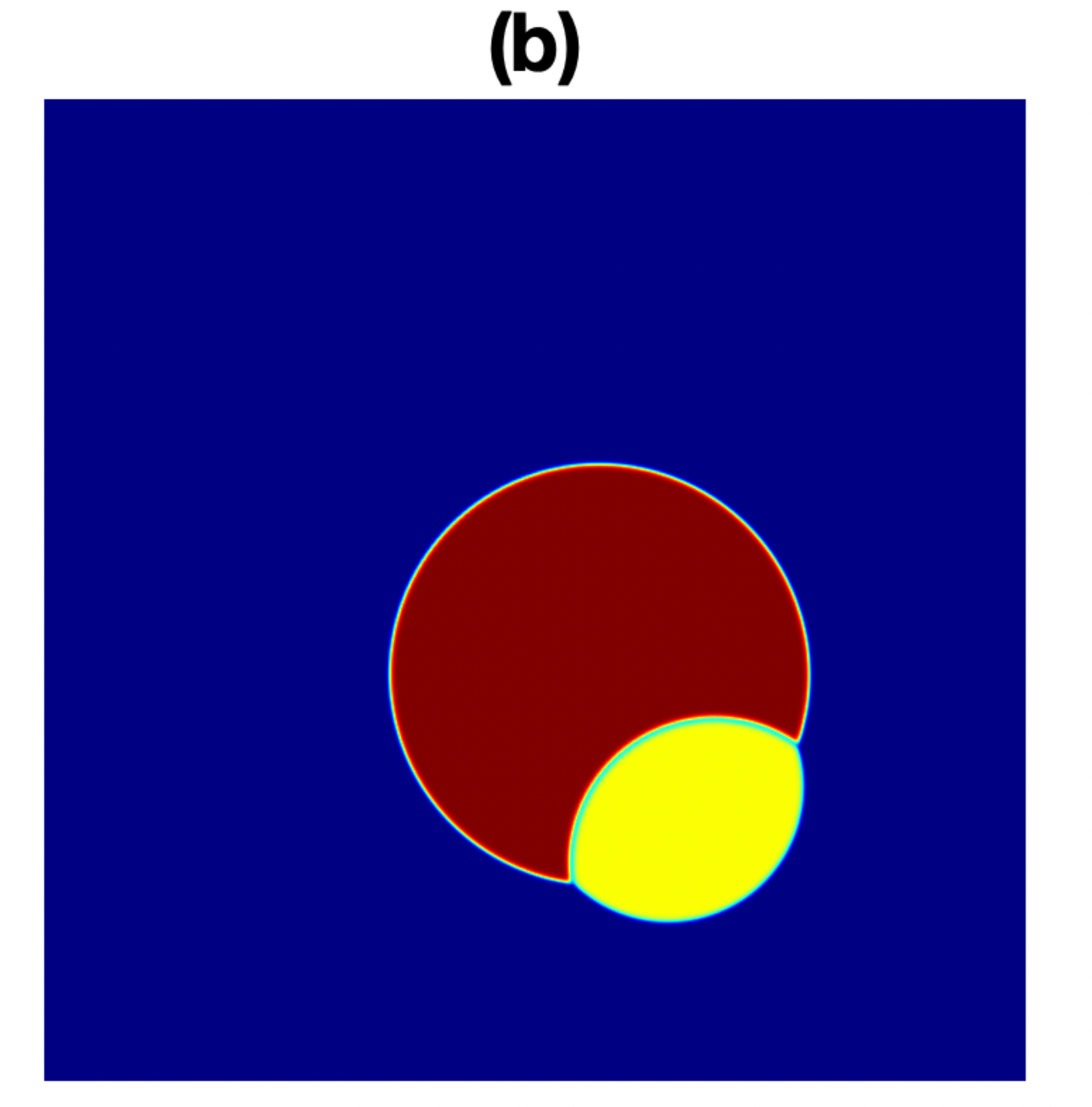}
\includegraphics[scale=0.175]{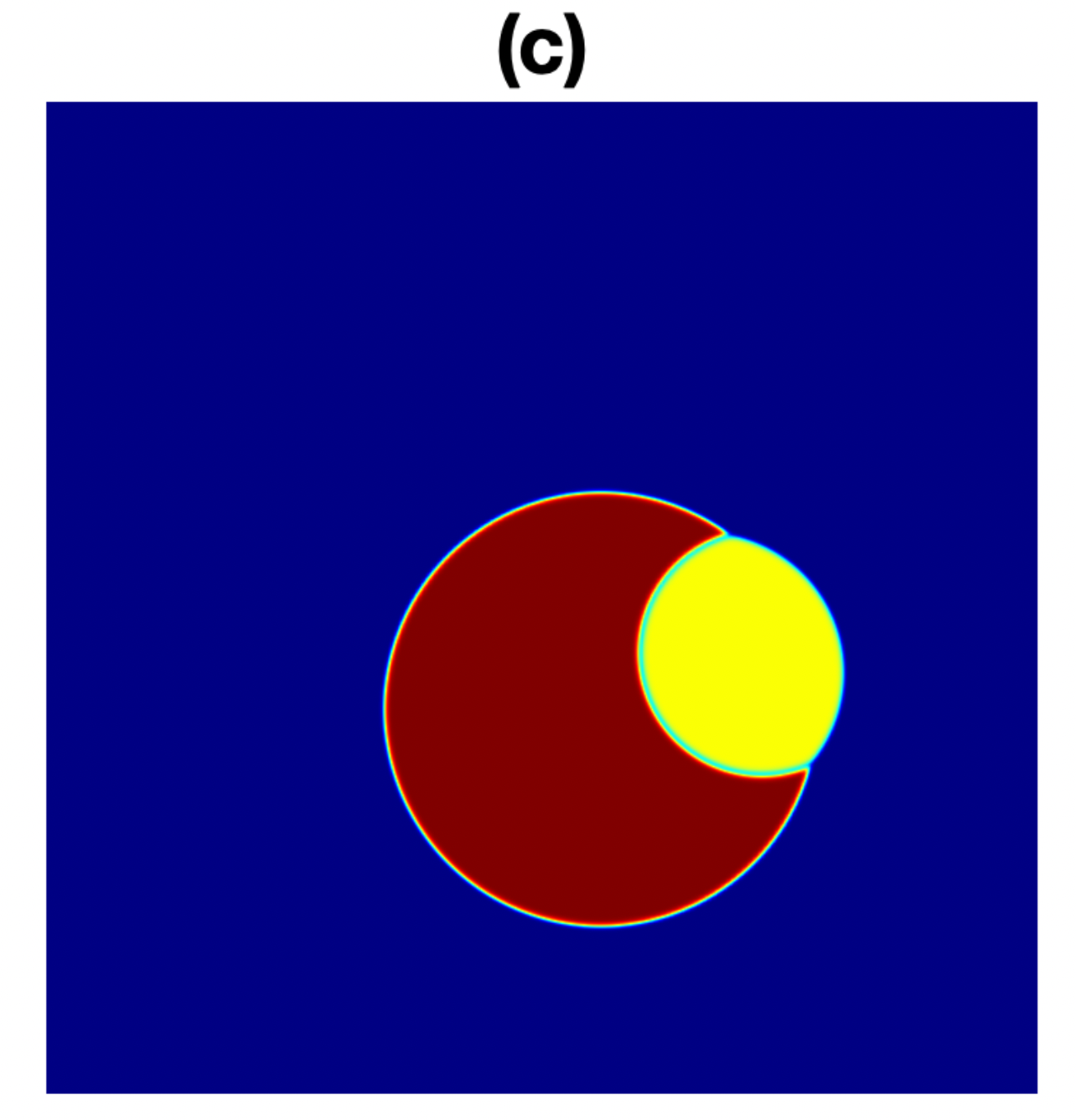}\\
\includegraphics[scale=0.18]{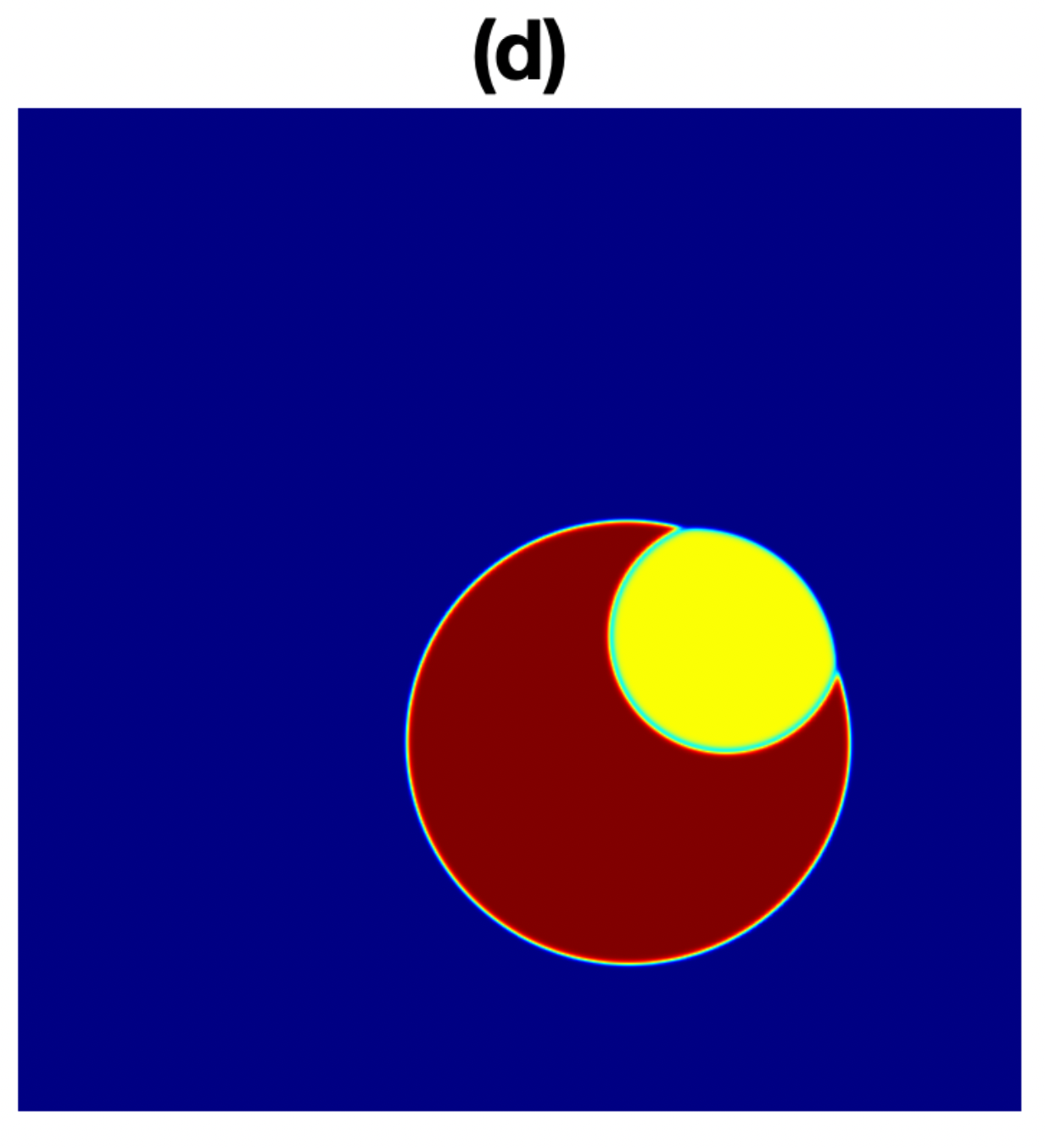}
\includegraphics[scale=0.18]{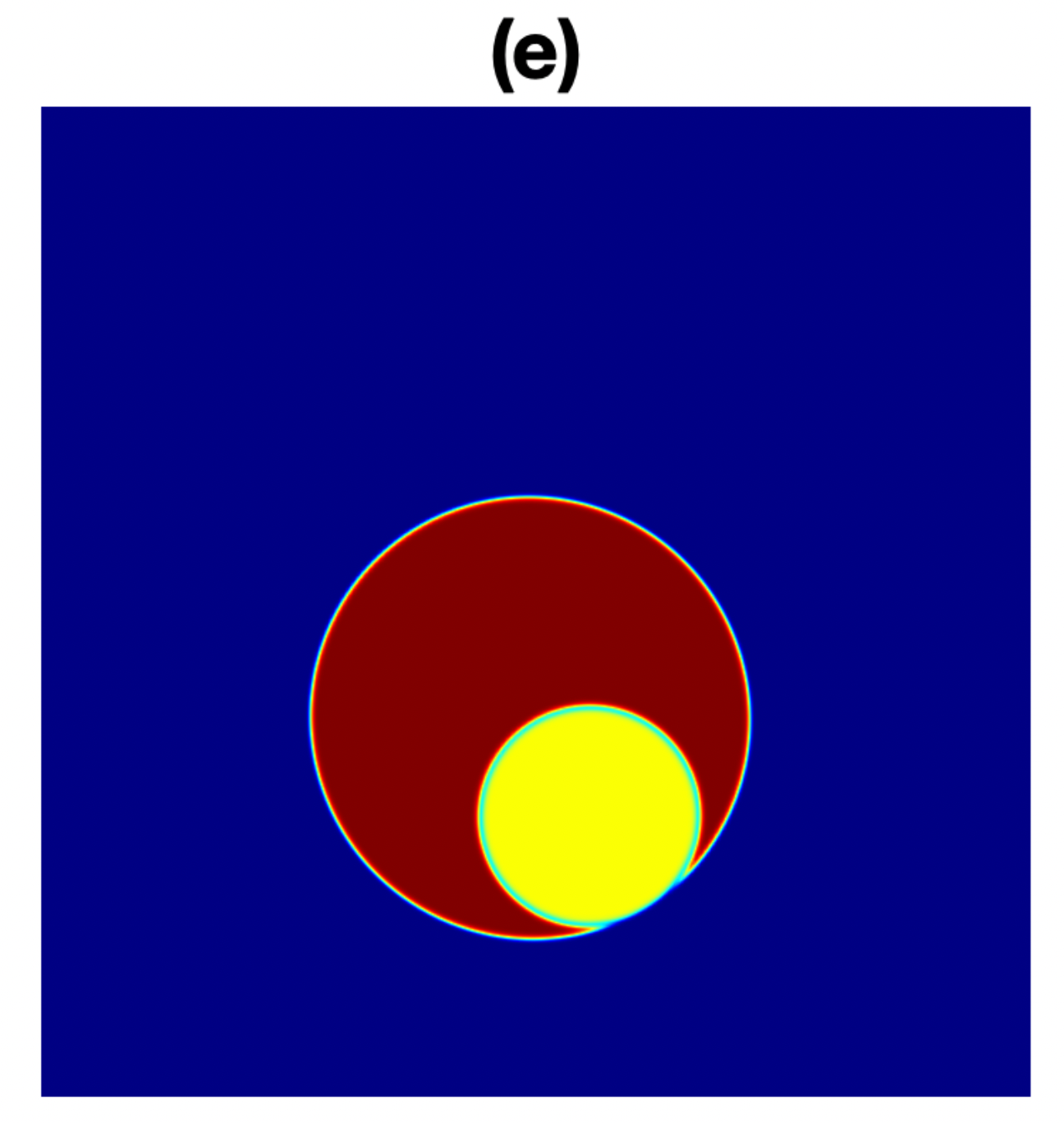}
\includegraphics[scale=0.18]{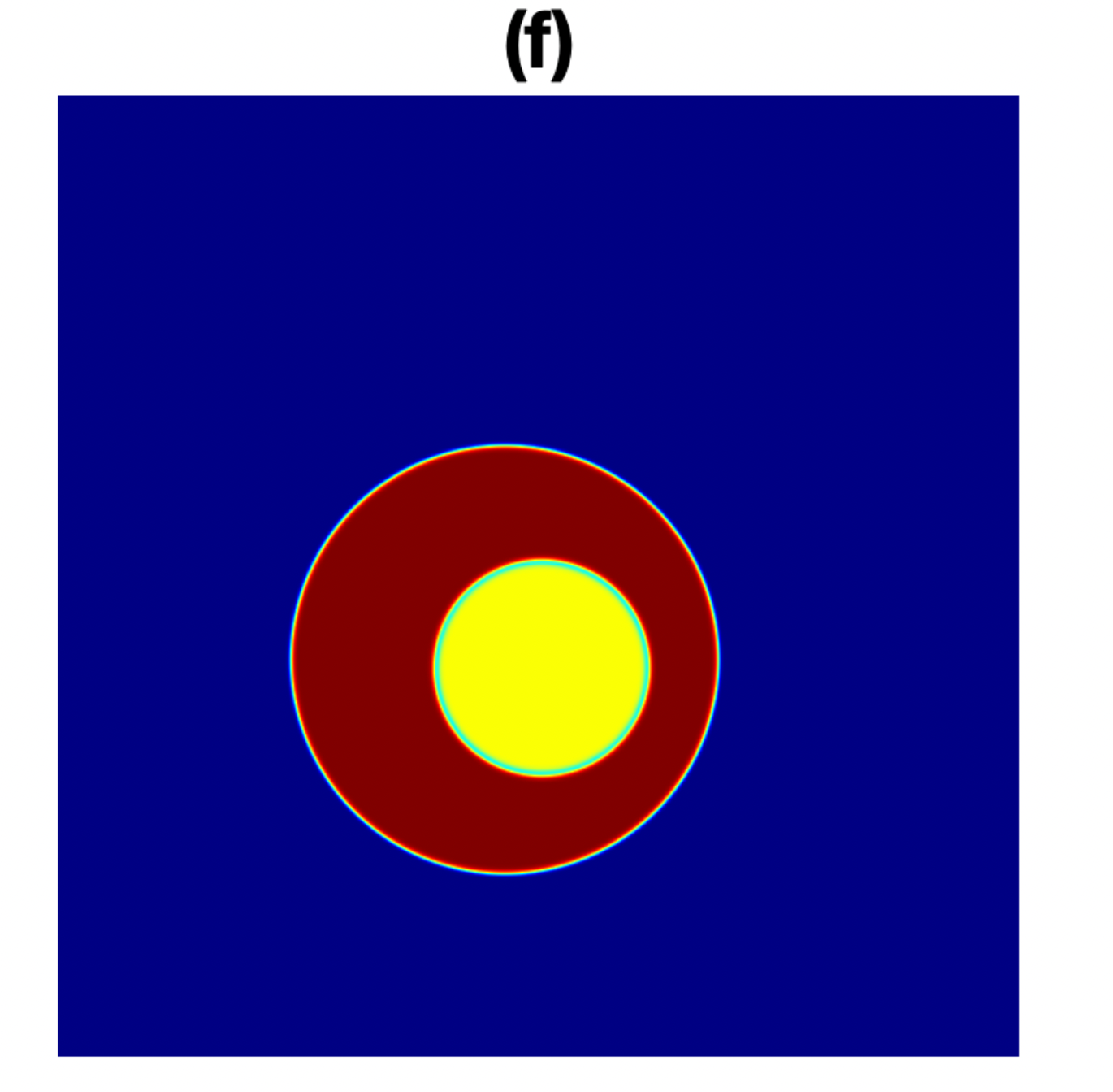}
\caption{ Numerical simulations: six stationary states in local ternary systems under different values of $\sigma_{02}$.
Each stationary state evolves from random initial data. 
(a) A standard double bubble with $\sigma_{02} = 1 $.
(b)-(d) Weighted double bubbles with $\sigma_{02} = 1.6, 1.8, 1.9 $ respectively.
(e) A tangential core shell with $\sigma_{02} = 2 $.
(f) A generalized core shell with $\sigma_{02} = 3 $.
For all six simulations, $\sigma_{01}=\sigma_{12}=1$, $M_1 = 0.12$, $M_2 = 0.04$.
} 
\label{evolves}
\end{figure}

\begin{figure}[ht]
\centering
\includegraphics[scale=0.18]{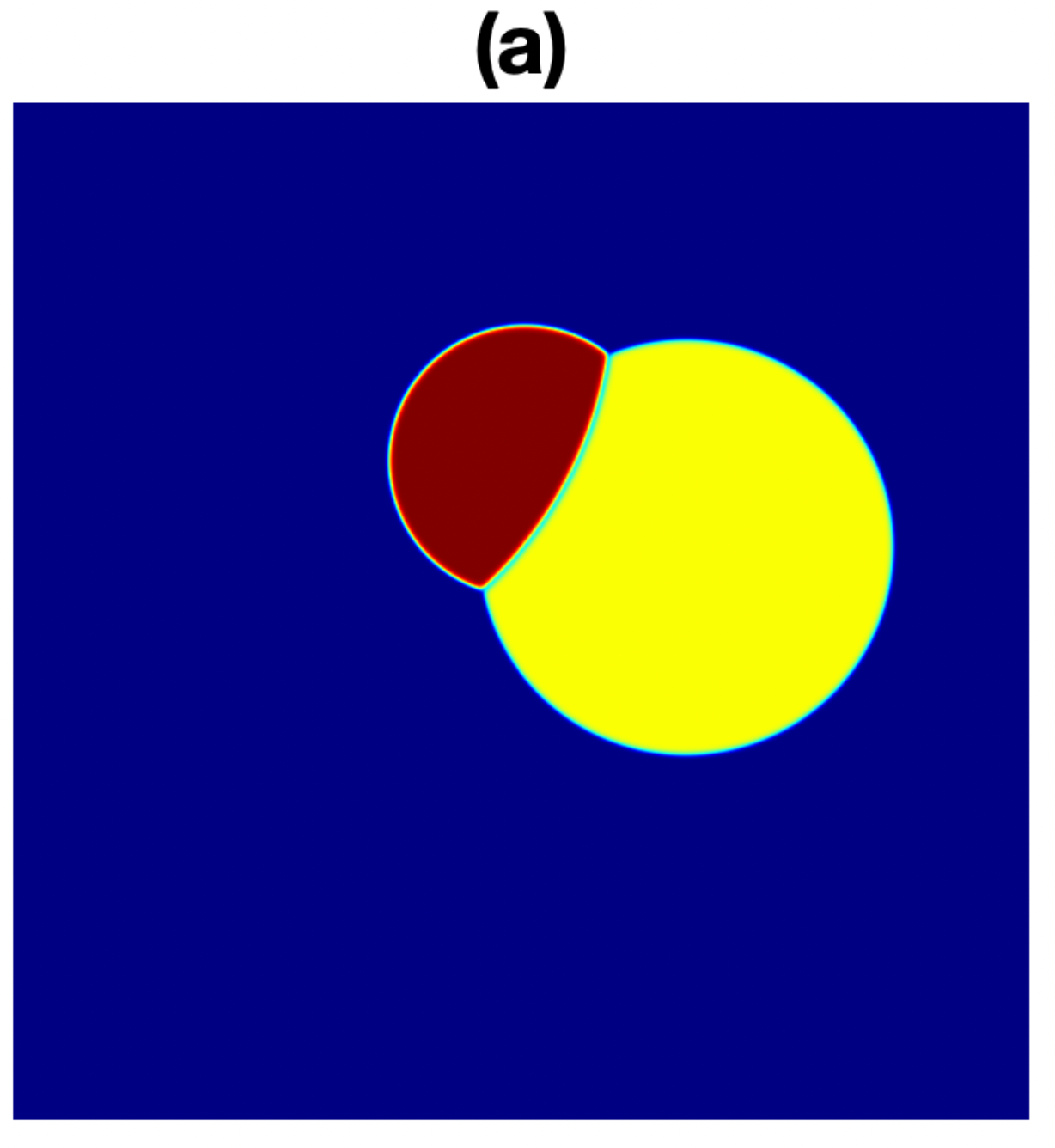}
\includegraphics[scale=0.18]{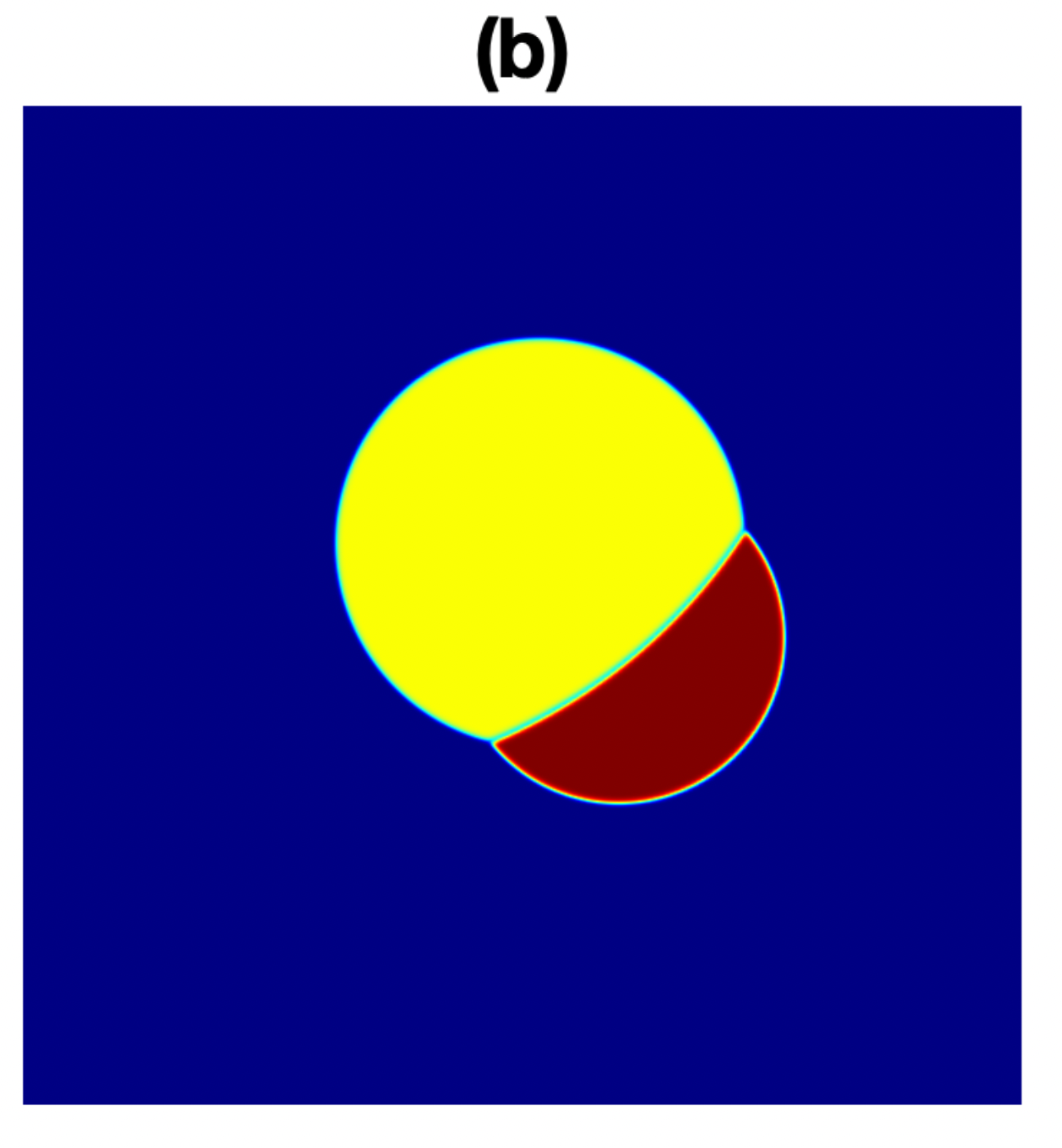}
\includegraphics[scale=0.18]{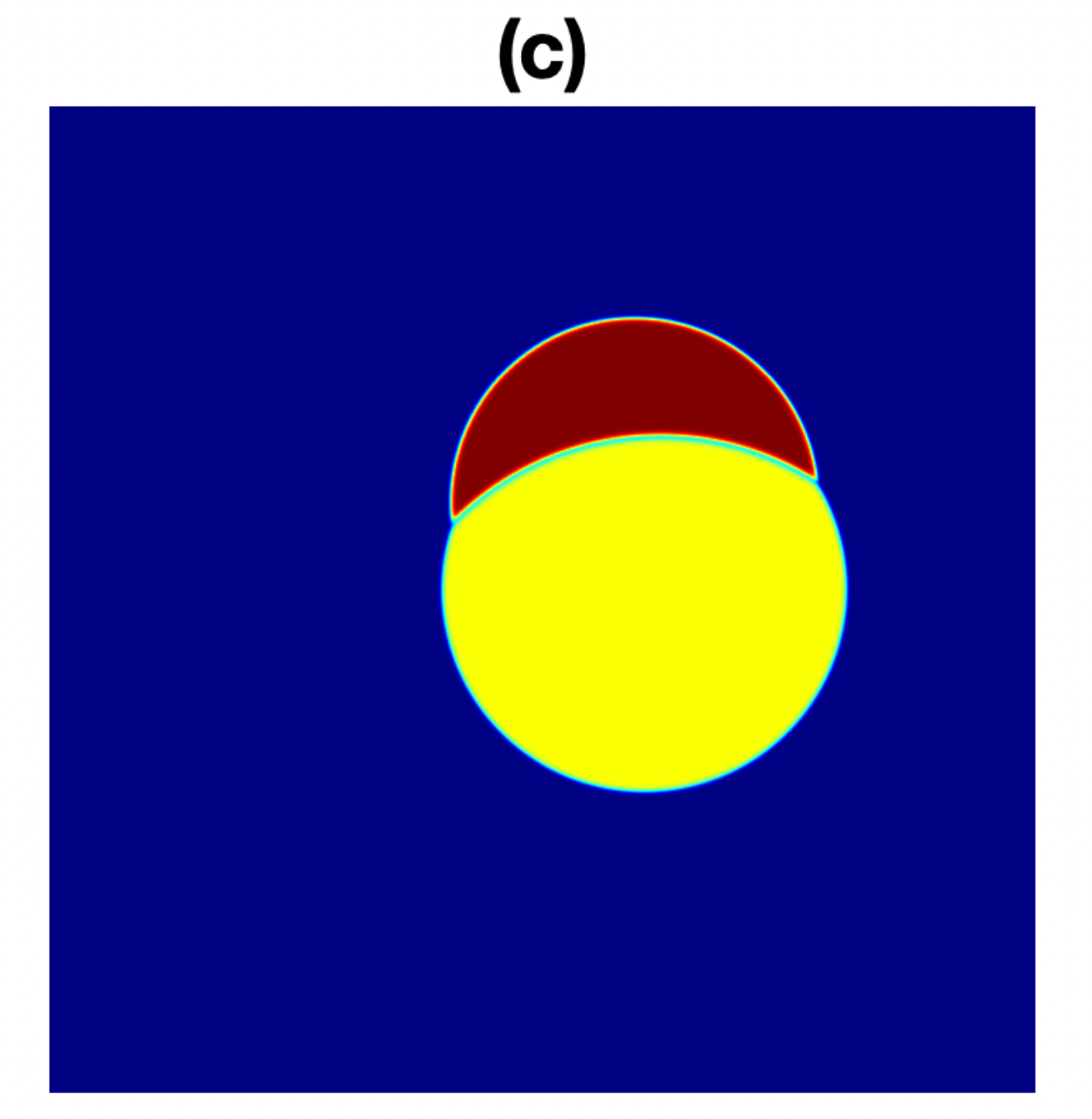}\\
\includegraphics[scale=0.18]{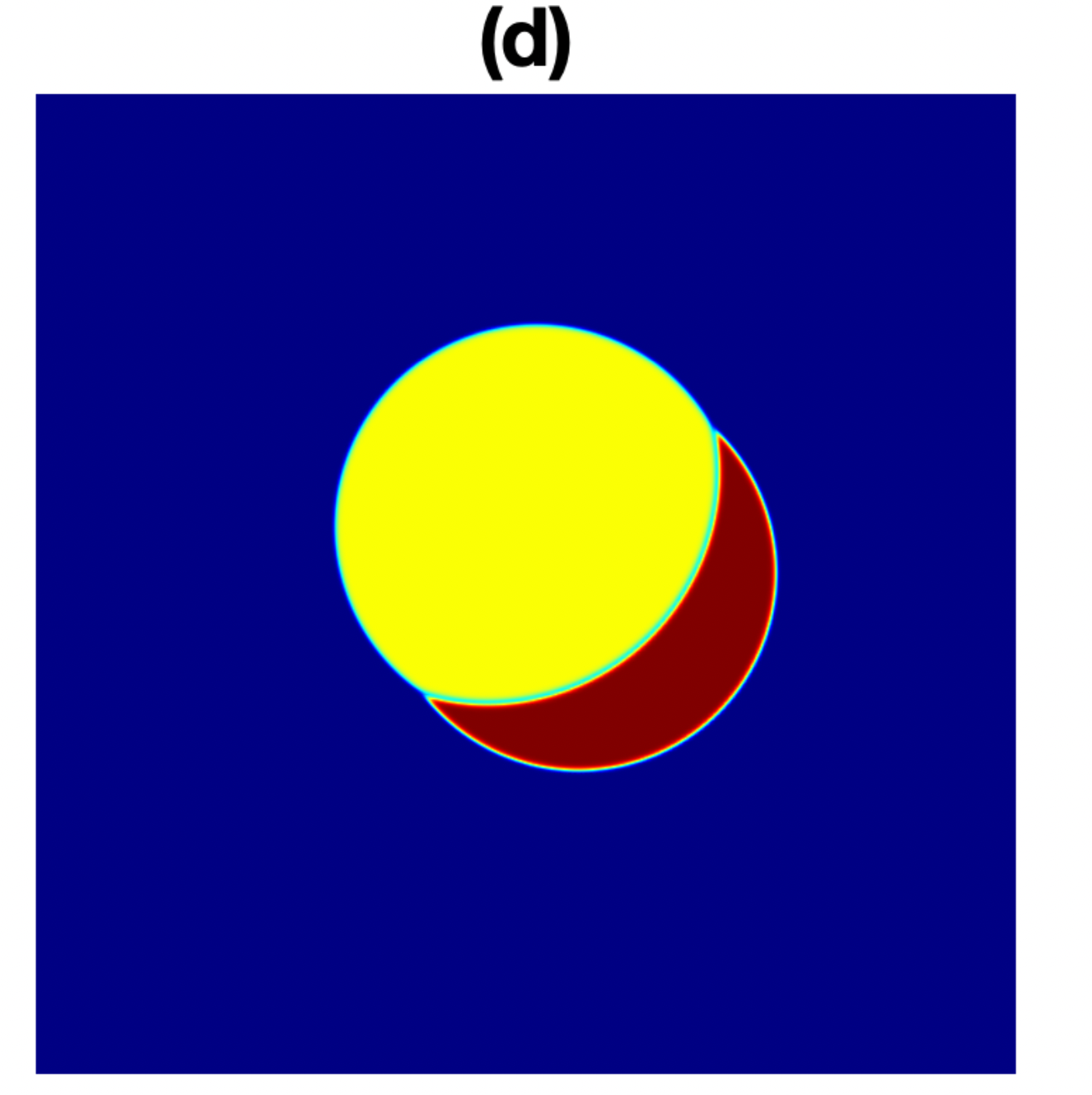}
\includegraphics[scale=0.18]{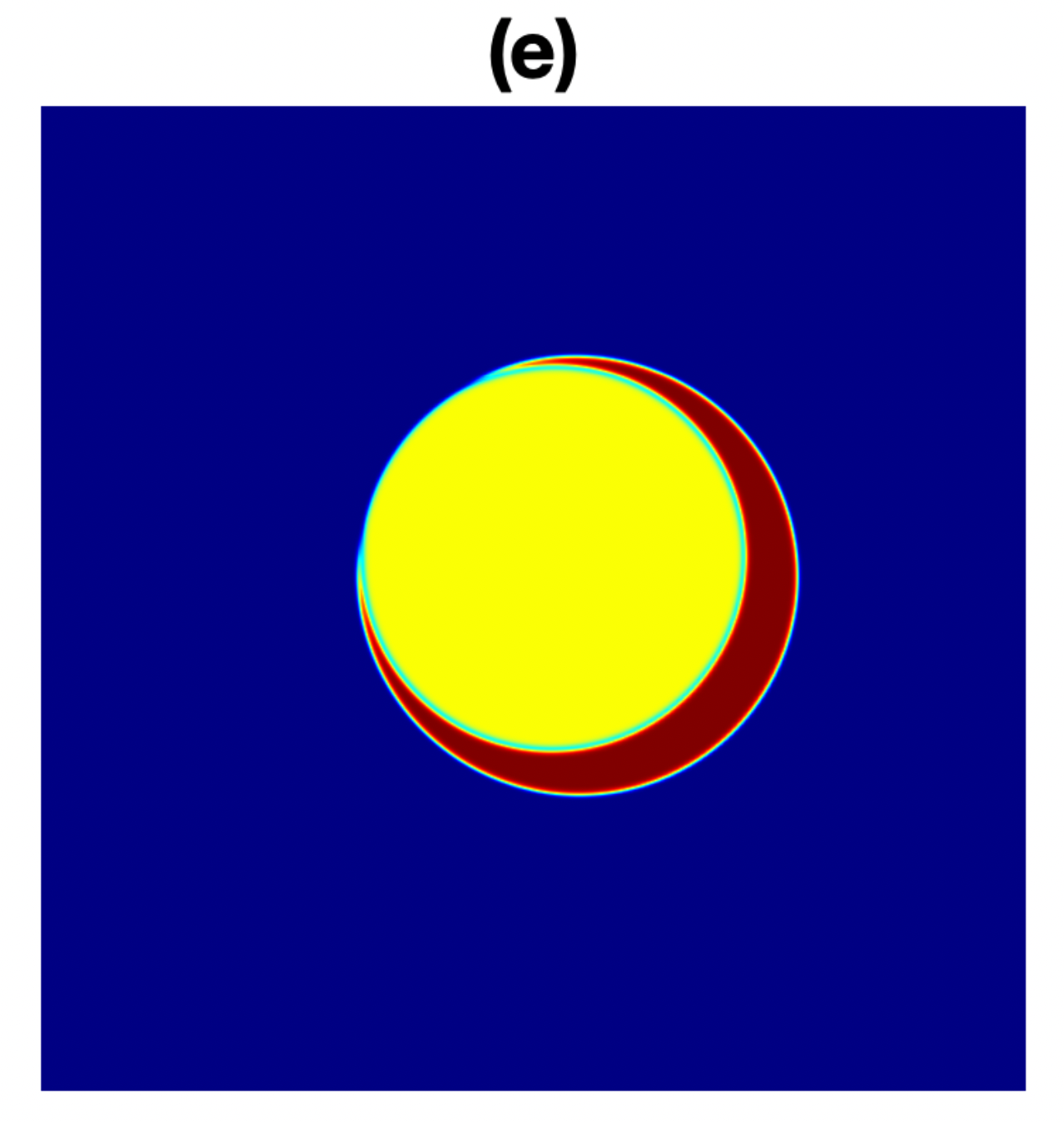}
\includegraphics[scale=0.18]{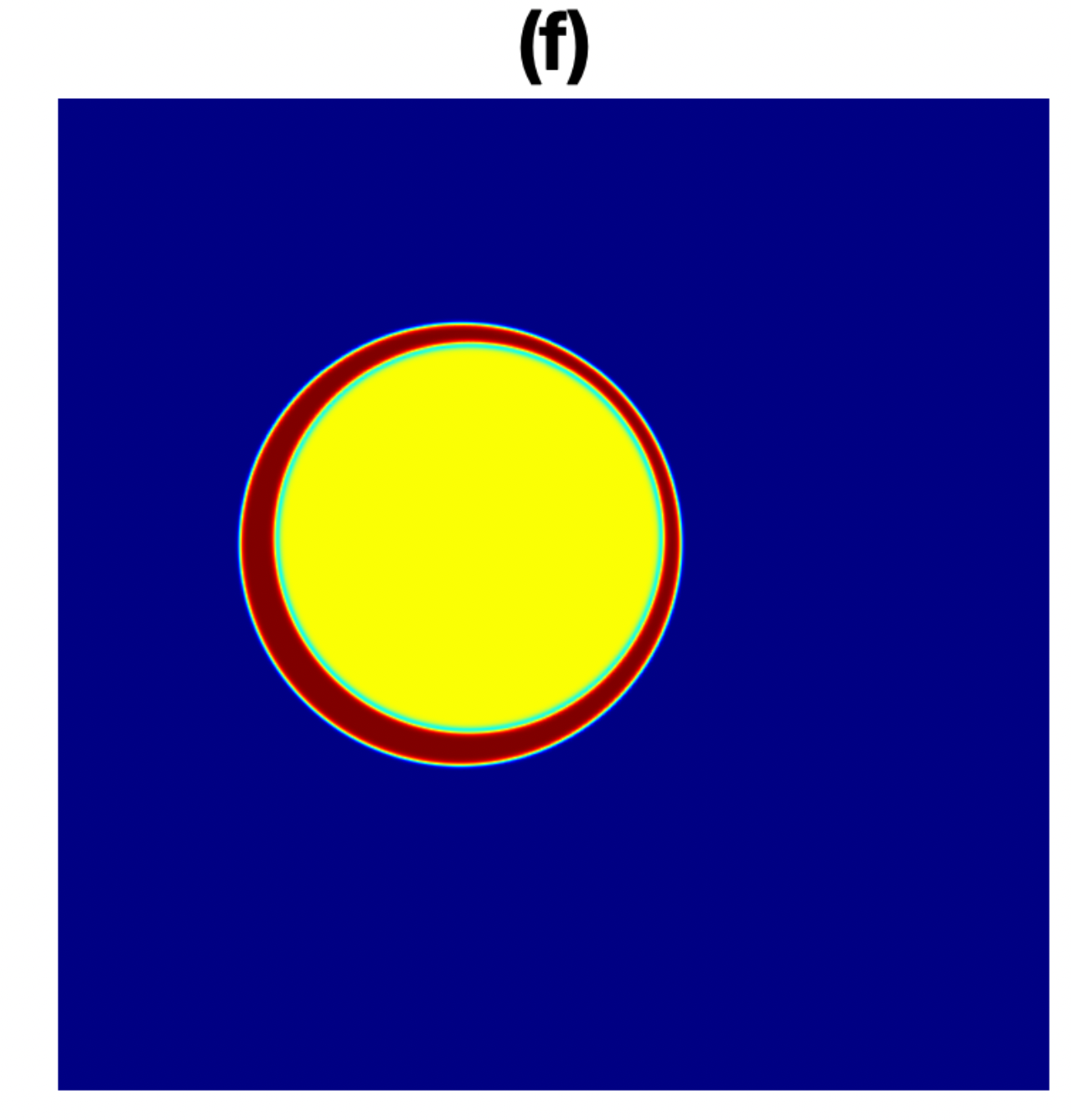}
\caption{Numerical simulations. Compared to Figure \ref{evolves}, here we reverse the values of $M_1$ and $M_2$, that is,   $M_1 = 0.04$, $M_2 = 0.12$. For all six simulations, $\sigma_{01}=\sigma_{12}=1$, and from (a) to (f), $\sigma_{02} = 1, 1.6, 1.8, 1.9, 2,$ and $3$, respectively. 
} 
\label{evolves2}
\end{figure}

In section 2 we consider minimizers of the local isoperimetric problem, minimizing the perimeter
among clusters
in $\R^2$ with given areas $M_1,M_2$, in the absence of the nonlocal term. 
We pose the problem in all of $\R^2$ both because it is a natural setting, and because it plays an important role in the global minimization of the nonlocal energy in the droplet regime, considered in section 4.
 The qualitative nature of minimizers depends strongly on the choice of $\sigma_{ij}$, and in particular on whether the triangle inequalities \eqref{triangle} hold {\it strictly} or not.  In case of strict triangle inequalities, isoperimetric sets with nonzero prescribed mass $M_1, M_2>0$ are {\it double bubbles} \cite{FABHZ,db1,db2,reichardt,lawlor,doubleAs,double}; see (a) to (d) in Figure \ref{evolves} and Figure \ref{evolves2} .  This is consistent with our result in our first paper, \cite{ablw}, in the unweighted case $\sigma_{ij}=  1 $, $i\neq j$, in which the isoperimetric sets are standard double bubbles with equal angles ${2\pi\over 3}$ at the triple junction points.  In the weighted case, the angles will be uniquely determined by the surface tensions via Young's law \cite{young,smith1,mullins1,www}.

The case of equality in one of \eqref{triangle} is also interesting, and yields a completely different geometry for minimizers.  If $\beta_1=0$, then we show that the optimal geometry is that of a {\it core shell,} $\mathcal{C}^{M_1}_{M_2}$, with an inner disk of area $M_2$ surrounded by an annular region of area $M_1$; see (e,f) in Figure \ref{evolves} and Figure \ref{evolves2}.  While this is a degenerate case, physically it is highly relevant as core shell constructions are often observed in nature.  Indeed, for an ABC triblock copolymer it seems more natural that minimizers form core shells than double bubbles, because of the linear structure of the blocks.  Since the bonding of the chains can only occur in linear order, one might expect that between regions of phase A and phase C there must be a phase B region, and that adjacency of phase A and C states should be energetically unfavorable.


This is indeed a degenerate setting, and the geometry of core shell minimizers is not uniquely determined by the perimeter alone, since the position of the interior disk in a core shell is arbitrary, from the point of view of the weighted perimeter alone.  The addition of the nonlocal term will resolve this degeneracy, as we will see in Proposition~\ref{concentric}.

For completeness, in numerical simulations, we also include the case $\beta_1 < 0$. In this case, the optimal geometry is also a core shell; see (f) in Figure \ref{evolves} and Figure \ref{evolves2}.

Finally, one may also consider the case $\beta_0=0$, in which case minimization prefers the separation of the cluster into disjoint single bubbles.  Again, this is a degenerate case, and after adding the nonlocal term, the relative position of the bubbles is not determined by the perimeter alone.  Indeed, we expect that the effect of the nonlocal repulsive term will be to push the two constituents far apart, as in the binary case.

\subsection*{Nonlocal effects}

It is well-known that the Green's function term in $\mathcal{E}(u)$ competes directly with the local isoperimetric term, in the sense that it is maximized by disks. In order to analyze its role in global minimizers we adopt the {\em droplet regime} scaling, introduced by Choksi-Peletier \cite{bi1}, a critical scaling in which both the isoperimetric and nonlocal terms in the energy act at the same energy scale.  This choice of material parameters (discussed in some detail in section 4) represents a dilute limit as the masses of phases $A$ and $B$ both tend to zero, but with correspondingly large interaction coefficients on the Green's function so that a rescaled characteristic function of each's phase domain behaves as Dirac delta measures in the limit.  
Introducing a small parameter $\eta>0$, which gives the length scale of a ``droplet'' of phase $A$ or $B$, we consider the rescaled energy 
\begin{equation}\label{Eeta1}  E_{\eta}^{} (v_{\eta}) 
  =
 \frac{\eta}{2} \sum_{i=0}^2  \beta_i \int_{\mathbb{T}^2} |\nabla v_{i,\eta} | +  \sum_{i,j = 1}^2  \  \frac{  \Gamma_{ij} }{2 |\log \eta| } \int_{\mathbb{T}^2} \int_{\mathbb{T}^2} G_{\mathbb{T}^2}(x - y)  v_{i, \eta}(x) v_{j, \eta}(y) dx dy,
\end{equation}
for $v_\eta = (v_{1,\eta},v_{2,\eta})= \eta^{-2}\chi_{\Om_\eta}\in [BV(\TT; \{0,\eta^{-2}\})]^2$ with prescribed masses, $\int_{\TT} v_{i,\eta} = M_i$, $i=1,2$. Here $\Gamma_{ij} = |\log \eta| \eta^3 \gamma_{ij}$; see Section \ref{nonlocalSec} for more details.

\begin{figure}[ht]
\centering
\includegraphics[scale=0.21]{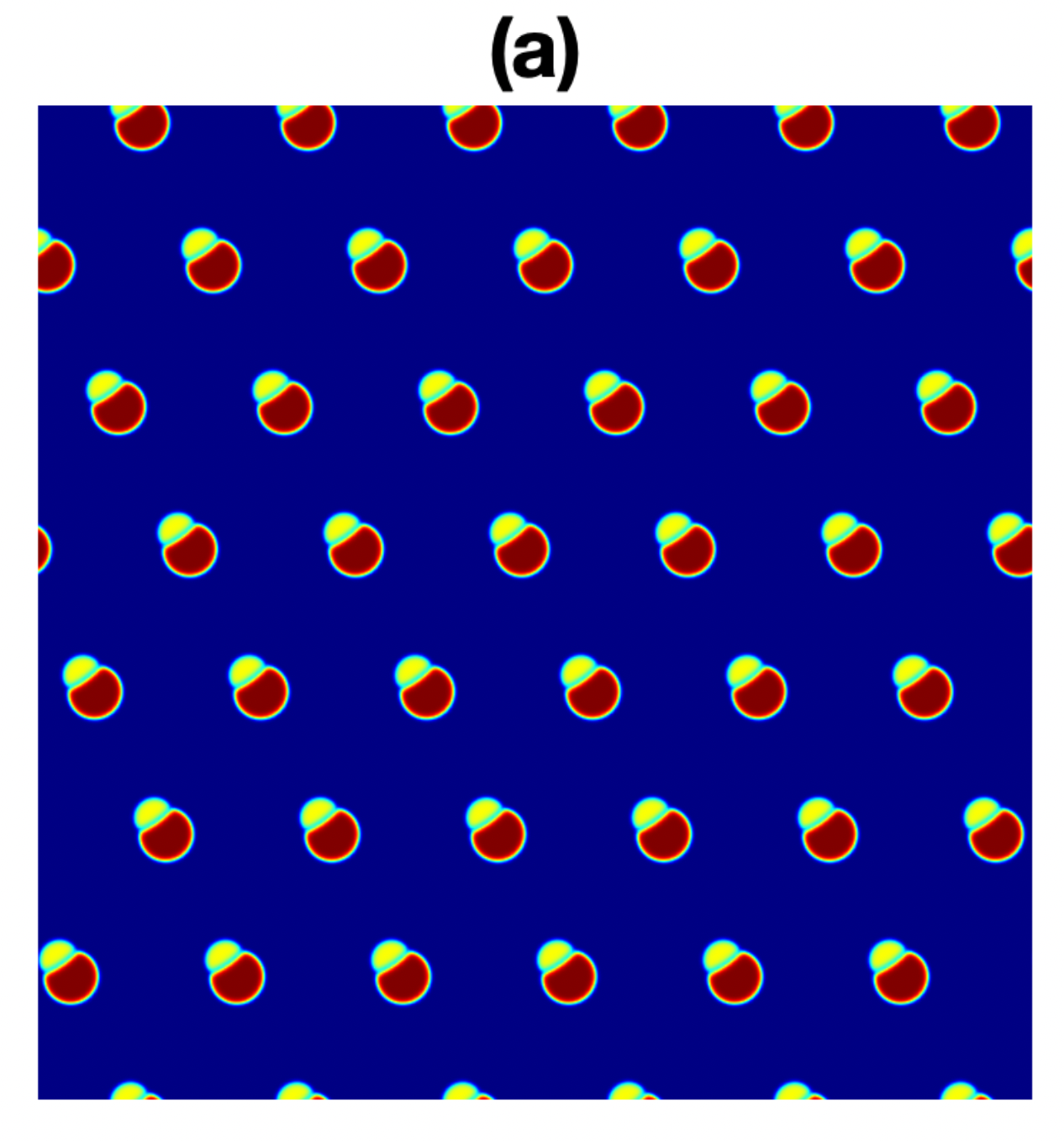}
\includegraphics[scale=0.21]{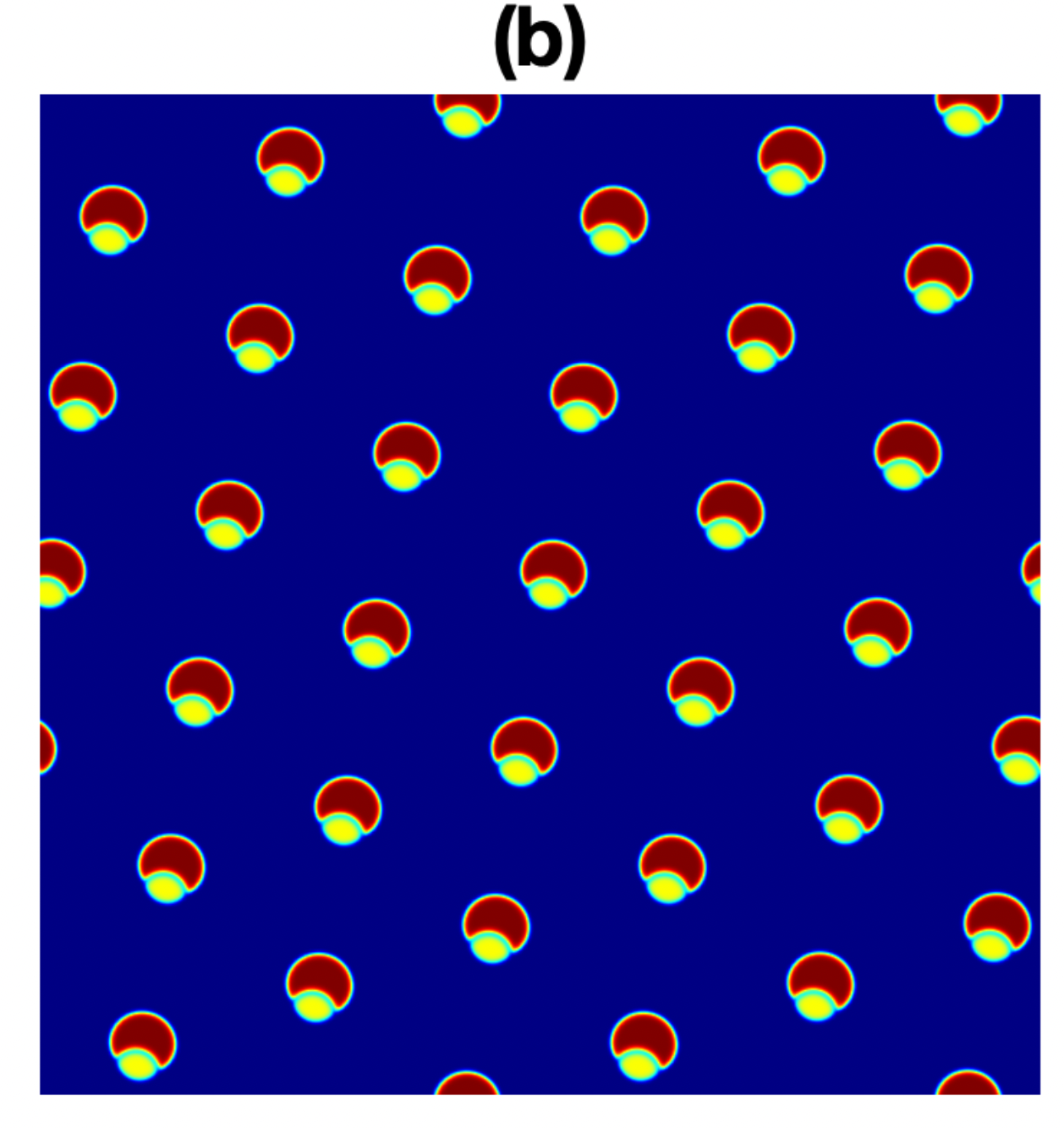}
\includegraphics[scale=0.21]{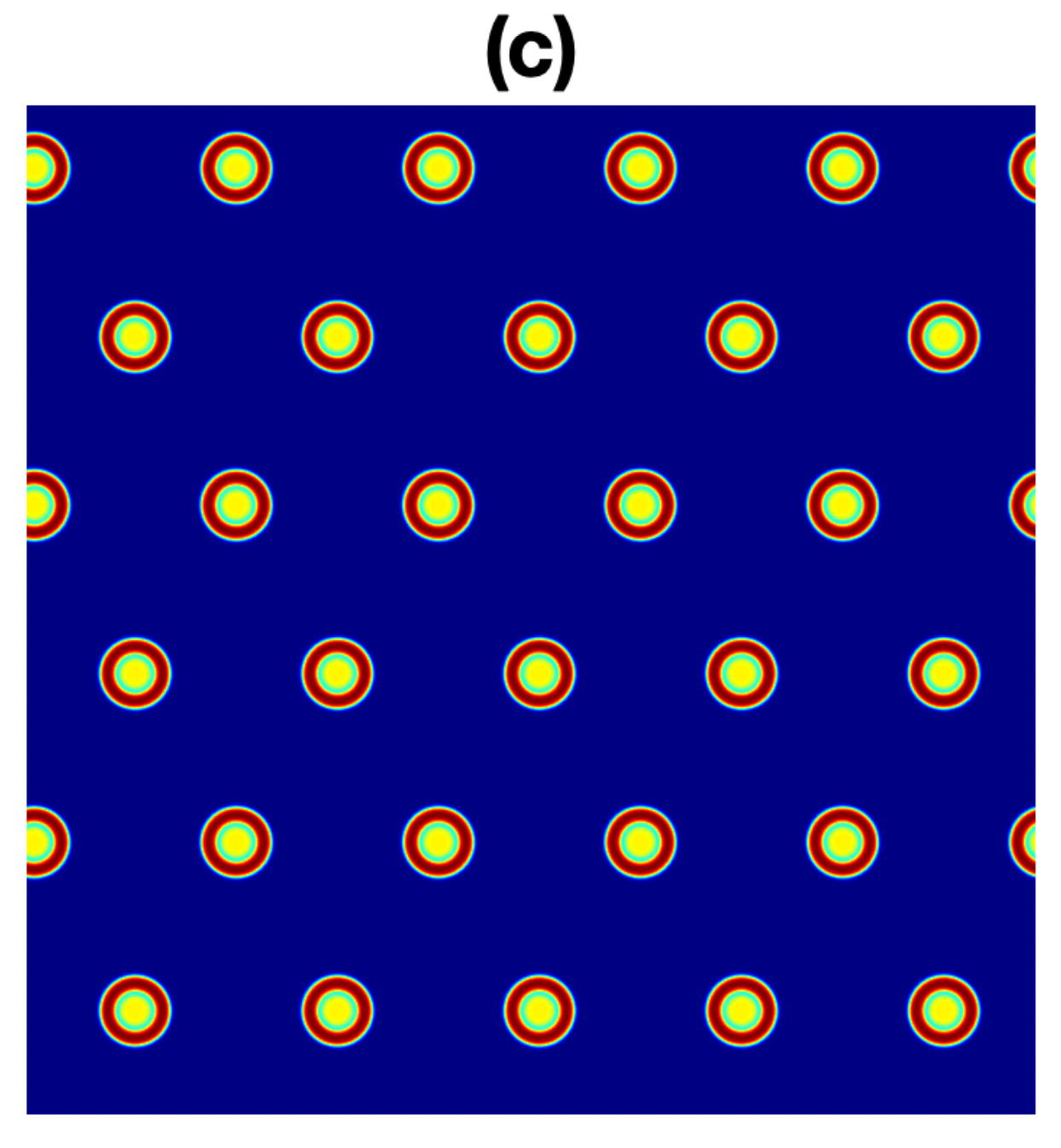}
\caption{ Numerical simulations:
three stationary states in nonlocal ternary systems under different values of $\sigma_{02}$.
Each stationary state evolves from random initial data. 
(a) Well-organized standard double bubbles when $\sigma_{02} = 1 $.
(b) Well-organized weighted double bubbles when $\sigma_{02} = 1.5 $.
(e) Well-organized core shells when $\sigma_{02} = 2 $.
For all three simulations, $\sigma_{01}=\sigma_{12}=1, M_1 = 0.10, M_2 = 0.05, \gamma_{11} = 16,000, \gamma_{12}=\gamma_{21}=0, \gamma_{22} = 54,000$. With changing $\sigma_{02}$, here we can see the transition from multiple double bubbles to multiple core shells. }
\label{evolves3}
\end{figure}

As in \cite{bi1,ablw}, finite energy configurations $E_\eta(v_\eta)\le C$ have a concentration structure (see Lemma~\ref{components}) by which $v_\eta$ splits into an at most countable collection of indecomposable clusters (in the sense of Maggi \cite[Chapter 29]{maggi}) with diameter $O(\eta)$.  For illustration, let's take this structure as an ansatz:  first, 
assume we have an at most countable collection of finite perimeter 2-clusters in $\R^2$, $\{A^k\}_{k\in\NN}$, where $A^k=(A^k_1,A^k_2)$, with $A^k_i\in\R^2$, $i=1,2$, and $\sum_{k=1}^\infty |A_i^k|=M_i$, i=1,2.  Then choose distinct points $\xi^k\in\TT$, and consider the configuration,
$$  \Om_\eta = \bigcup_{k\in\NN} (\eta A^k + \xi^k), \quad \text{and} \quad
   v_\eta = \eta^{-2}\chi_{\Om_\eta} .  $$
Substituting into $E_\eta$ yields:
\begin{align*}
E_\eta(v_\eta)
&=\sum_{k=1}^{\infty} \sum_{i=0}^2 \frac{\eta}{2} \beta_i \int_{\TT}|\nabla v_{i, \eta}^k |
  +{\Gamma_{ij} \over 2 |\log \eta|} \sum_{k,\ell=1}^{\infty} \sum_{i,j=1}^2
     \int_{\TT}\int_{\TT} v_{i, \eta}^k (x)\, G_{\mathbb{T}^2} (x-y)\, v_{j, \eta}^\ell(y)\, dx\, dy \\
&= \sum_{k=1}^{\infty}\left[ P_\sigma(A^k)  + \sum_{i,j=1}^2 {\Gamma_{ij}\over 4\pi} |A^k_{i}|\, |A^k_{j}| \right] + O(|\log\eta|^{-1})\\
&=\sum_{k=1}^{\infty}\GG(A^k) + O(|\log\eta|^{-1}),
\end{align*}  
with limiting energy of each component cluster $(A_1,A_2)$,
\begin{equation}\label{Gsigma}   \GG(A):= P_\sigma(A)  + \sum_{i,j=1}^2 {\Gamma_{ij}\over 4\pi} |A_{i}|\, |A_{j}|.
\end{equation}
The blow-up energy $\GG(A^k)$ is the weighted version of the energy of clusters studied in \cite{ablw}, and a ternary extension of the one analyzed in the binary case in \cite{bi1} in two dimensions.  The nonlocality of $E_\eta$ is expressed in the quadratic terms which depend on the mass:  the larger are $M_1, M_2$, the greater the need to split the phase domains into more and more droplets.  However, the nonlocal interaction does not affect the geometry of minimizers, which are studied in section 2 for the various choices of weights $\sigma_{ij}$.  
The deformation in the shape of double-bubbles, from the equal-angles case (with equal weights $\sigma_{ij}=  1 $) to cases of very different weight values, including well-organized weighted double bubbles and well-organized core shells, are illustrated in Figure~\ref{evolves3}.



As in the unweighted case (and the binary case,) we prove two  $\Gamma$-convergence results to describe the behavior of minimizers (or low energy states) of $E_\eta$.  The first limit 
 (Theorem~\ref{firstlimit}) describes the splitting of masses and the fine scale geometry of the indecomposable clusters as minimizers of $\GG$ in $\R^2$.  Let
 \begin{align} \label{ezero}
  e_0( m )&:=\min   \left\{   \GG( A ) \ | \  A =(A_1,A_2) \text{ 2-cluster, with $| A_i  |= m_i $, $i=1,2$}\right \},  \quad   \text{and} \\
  \label{ezerobar}
   \overline{e_0 }(M ) &:= \inf \left\{ \sum_{k=1}^{\infty} e_0 (m^k ) :  m^k = (m_1^k, m_2^k ),  \ m_i^k \geq 0,\ \sum_{k=1}^{\infty}  m_i^k = M_i, i = 1, 2 \right\}.
  \end{align}
Then, sequences of minimizers $v_\eta$ of $E_\eta$ with mass $M=(M_1,M_2)$ converge as measures 
$v_\eta\wto v_0=\sum_{k=1}^\infty m^k\delta_{x^k}$, where $x^k\in\TT$ and $ \{m^k\}_{k\in\NN}$ are determined by $\overline{e_0 }(M )$.
 
The second $\Gamma$-limit (Theorem~\ref{secondlimit}) exploits the remainder terms, of order $O(|\log\eta|^{-1})$ in the expansion of $E_\eta(v_\eta)$ above, to locate the centers $x^k$ of the droplets, and in the case of core shells, to determine the explicit geometry of the shells.  We recall from the discussion of the local isoperimetric problem above that when $\beta_1=0$ the perimeter term is degenerate, and the position of the interior disk is not determined by minimizing $P_\sigma$.  In Proposition~\ref{concentric} we show that the detailed geometry of core shells is determined at the $O(|\log\eta|^{-1})$ level, depending on the relative values of the interaction coefficients $\Gamma_{11}, \Gamma_{12}$.  When the repulsion between phases is weak, that is, $\Gamma_{12}$ is small, concentric core shells reduce the energy at second order,
while for stronger repulsion $\Gamma_{12}>\Gamma_{11}$ core shells' inner disks should be tangent to the exterior circle.  


\subsection*{Core shell assemblies}

Ren \& Wang have constructed critical points of $\mathcal{E}(u)$ representing dilute lattices of core shells \cite{stationary} and single bubbles \cite{disc}, and stationary solutions with double-bubble lattices were constructed by Ren \& Wei \cite{doubleAs}.  Numerical studies of periodic minimizers suggest 
that these assemblies do appear as minimizers of $\mathcal{E}(u)$ 
in many parameter regimes \cite{wrz}, and it is an interesting and challenging problem to verify this rigorously.  At least in the droplet regime limit described above, a first step is to study the limiting energy $\GG$, with given total mass $M=(M_1,M_2)$ and weights $\sigma$.  In our first paper \cite{ablw} we proved many properties of the minimizers of $\GG$ in the case of equal weights.  We showed that the number of droplets is finite, and the size of the constituent components is bounded in terms of the coefficients and masses.  We showed that, while a minimizer can exhibit both double-bubbles and single-bubbles, there can only be one species of single-bubble if there is coexistence of single- and double-bubbles.  

In section 3 we take up the same questions in case $\beta_1=0$, and the local isoperimetric problem favors core shell minimizers.  In Lemma~3.1 we show that there is a lower bound $m^-$ on the mass of any droplet constituent, whether single-bubble or component of a core shell, which depends only on the weights $\sigma_{ij}$, the coefficients $\Gamma_{ij}$, and the total masses $M_1,M_2$.  From this we may conclude (Corollary~\ref{finiteness}) that minimizing configurations of $\GG$ can have only finitely many nontrivial indecomposable components.  We note that the finiteness of components in the binary case studied by Choksi \& Peletier \cite{bi1} was proven using the concavity of the perimeter; as in our study of double-bubbles \cite{ablw}, the perimeter in the ternary case is not globally concave, and so more delicate arguments are required.

As was the case with double-bubbles, many open questions remain.  Numerical studies show regimes in which minimizers appear to have {\it only}  core shells, but we have no theorem which shows that this must be the case.  On the other hand, simulations also show that coexistence of core shells and single bubbles can occur when one chooses values of $M_i,\Gamma_{ij}$ which are very different from each other, so any result in this direction would have to take the ranges of values of the parameters into account. 

\subsection*{On clusters}

In this paper we use the framework of {\it clusters} of finite perimeter sets in $\R^2$, as set out in \cite[Part IV]{maggi}.  A 2-cluster in $\R^2$ is a disjoint pairing $A=(A_1,A_2)$ of finite perimeter sets, $|A_1\cap A_2|=0$, each of finite Lebesgue measure.  We write the mass as $m=(m_1,m_2)$, $m_i=|A_i|$, $i=1,2$.  It will be convenient for us to permit one of the chambers to be empty, so the case where one of $m_i=0$ is allowed.  The exterior domain $A_0=\R^2\setminus \overline{(A_1
\cup  A_2)}$ has infinite measure, but its perimeter is included in the total weighted perimeter of the cluster.  As $\widetilde A_0=\R^2\setminus A_0= \overline{( A_1 \cup  A_2)}$ has the same perimeter, it will often be convenient to replace $A_0$ by $\widetilde A_0$ in calculating the perimeter, that is:
$$ P_\sigma (A)= \frac{1}{2}  \left [\beta_1 P_{\R^2}(A_1) + \beta_2 P_{\R^2}(A_2)  +\beta_0 P_{\R^2}(\overline{(A_1  \cup A_2)}) \right ].  $$

The chambers of a cluster $A=(A_1,A_2)$ do not need to be connected, and indeed in studying the nonlocal problem we expect that they will split into disjoint components.  The proper measure theoretic definition is that of {\it indecomposability}:  a set $E$ is indecomposable if whenever $E=E_1\cup E_2$ with $|E_1\cap E_2|=0$ and $P_{\R^2}(E)=P_{\R^2}(E_1) + P_{\R^2}(E_2)$, then one of $|E_1|,|E_2|=0$.  (See \cite{ACMM01}.)

\subsection*{Numerical methods}

\begin{figure}[ht]
\centering
\includegraphics[scale=0.14]{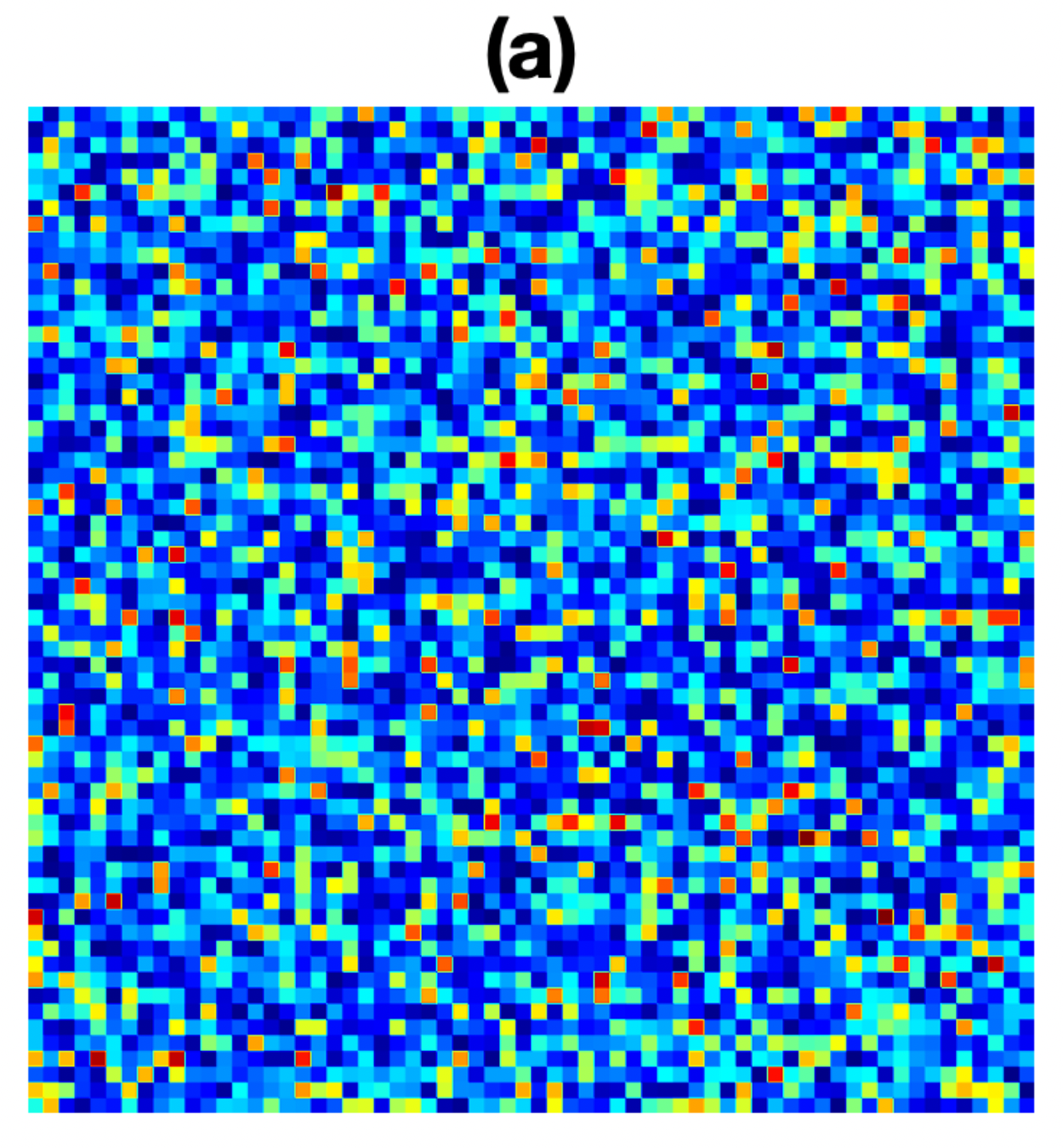}
\includegraphics[scale=0.14]{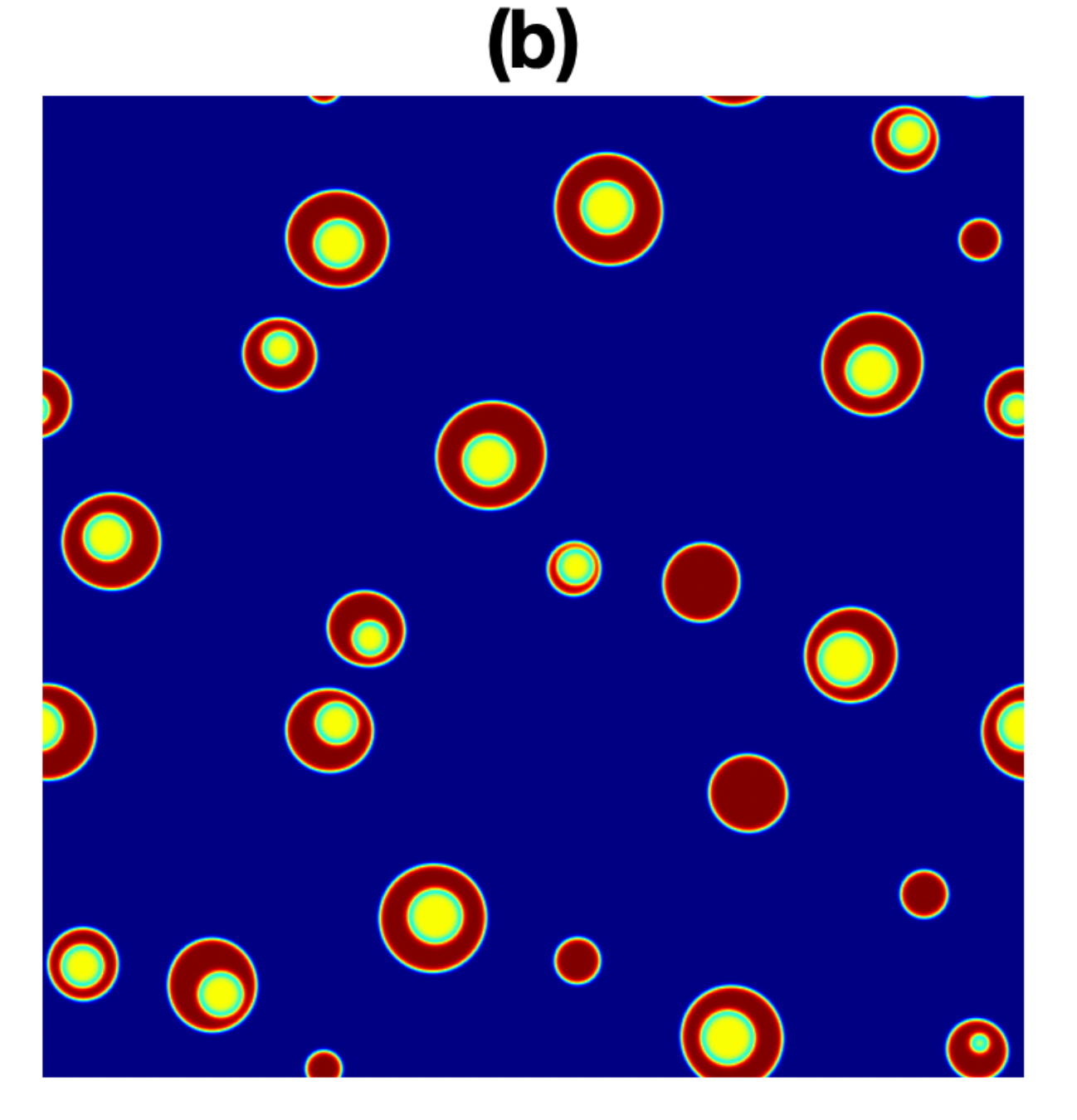}
\includegraphics[scale=0.14]{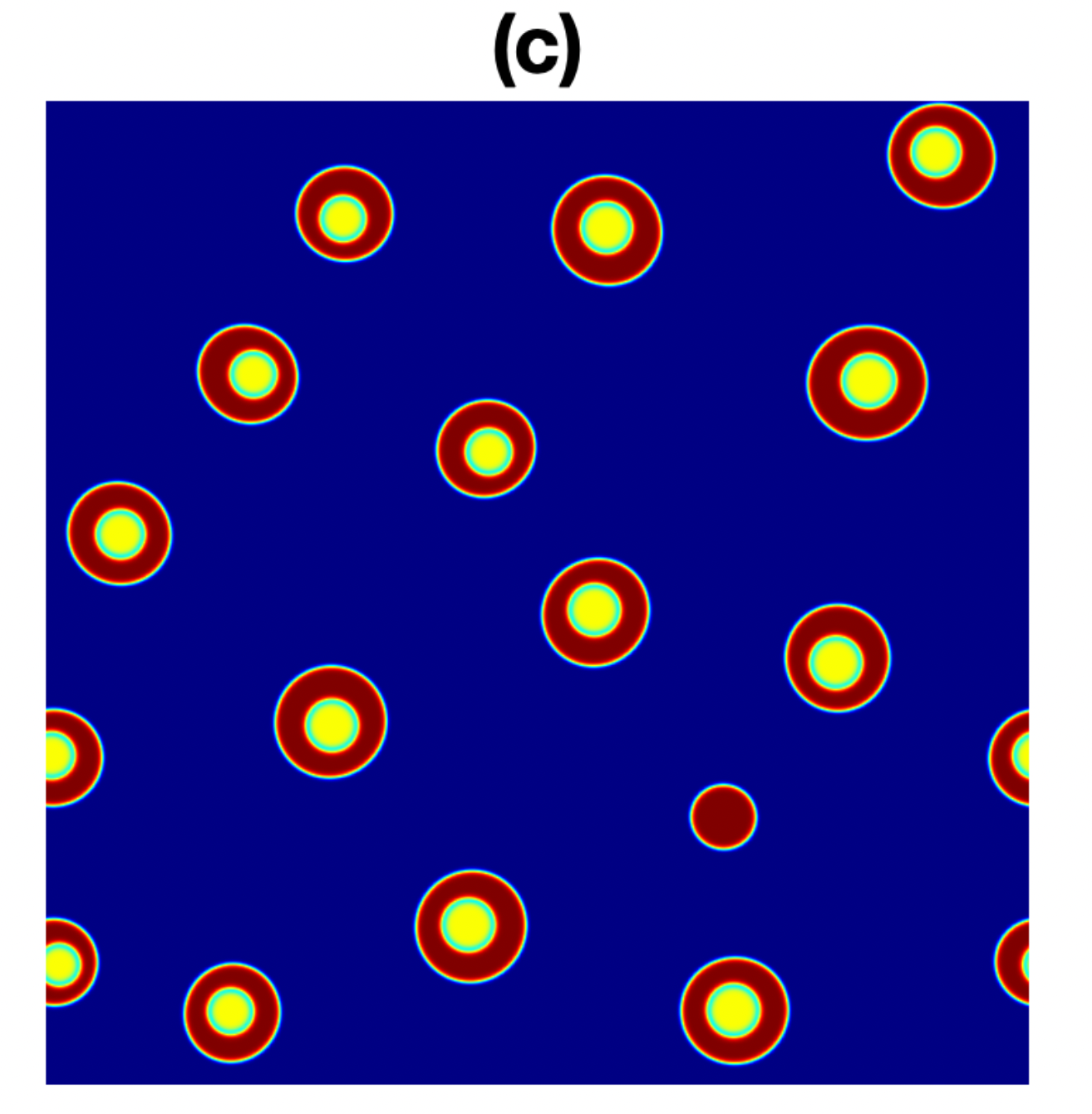}
\includegraphics[scale=0.14]{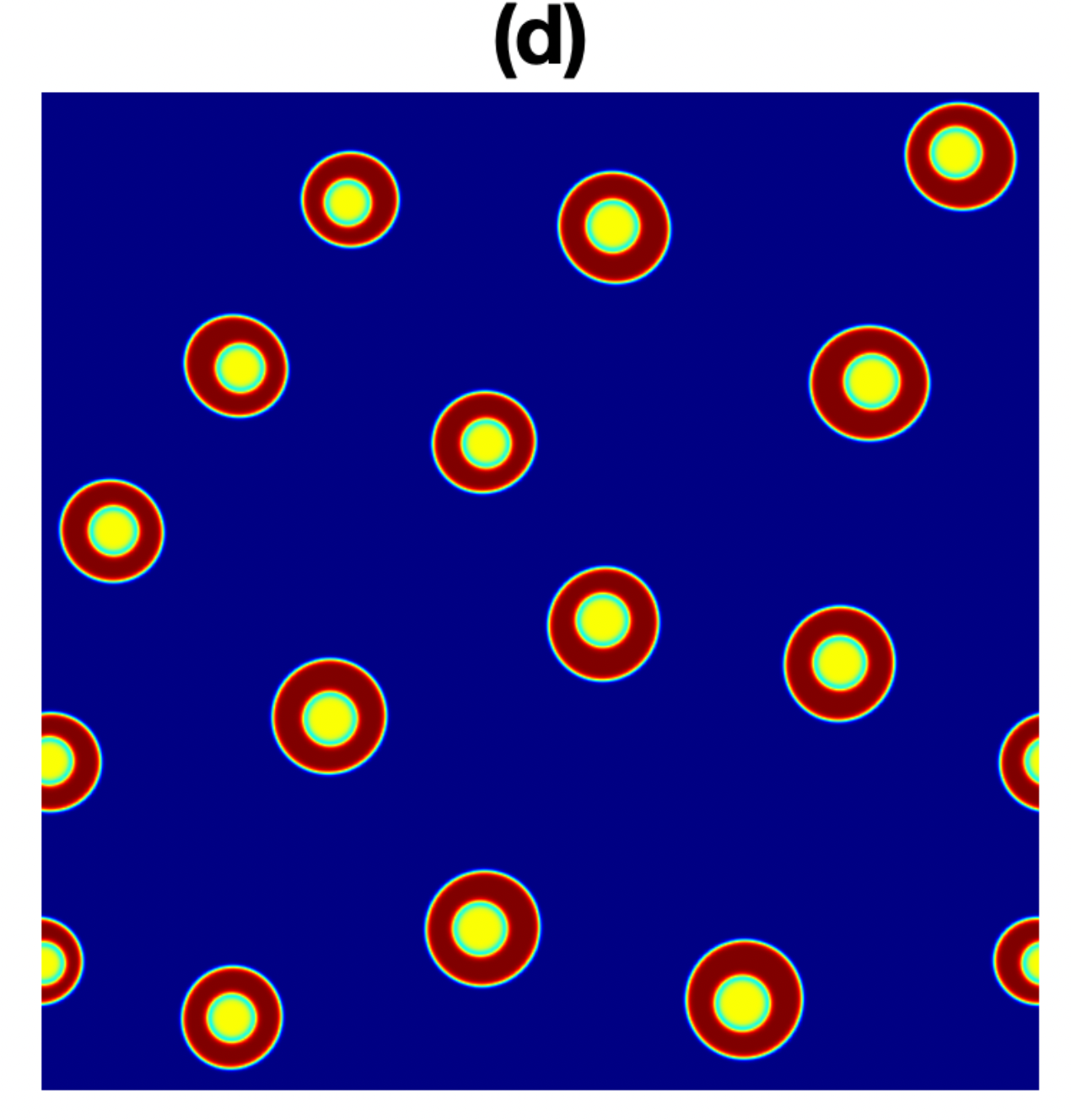}
\includegraphics[scale=0.14]{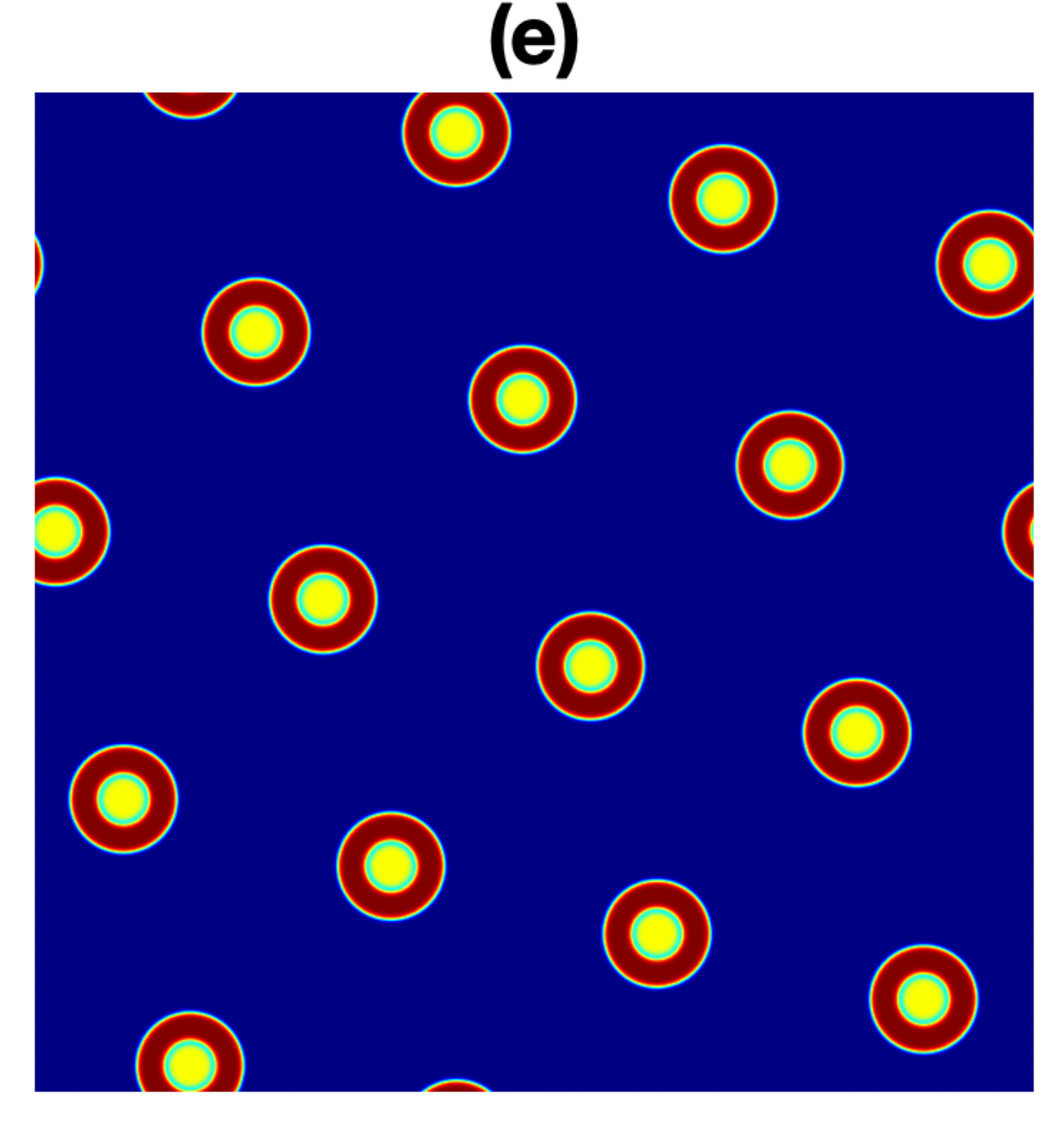}
\caption{ One sample numerical simulation. The system starts from random initial data and converges to a steady state of well-organized core shells. At the steady state, all core shells are concentric, of equal size, and distributed in a perfect hexagon pattern. Here $\sigma_{01}=\sigma_{12}=1$, $\sigma_{02} = 2$, $M_1 = 0.12$, $M_2 = 0.04$, $\gamma_{11} = 4,000$, $\gamma_{12}= \gamma_{21} = 0, \gamma_{22} = 20,000$. }
\label{evolves4}
\end{figure}

To minimize the free energy \eqref{energyuep}, we consider the $L^2$ gradient flow dynamics. Periodic boundary conditions are used here. To fulfill the mass constraints, we adopt a modified augmented Lagrange multiplier approach. The coupled nonlocal Allen-Cahn equations with mass constraints are first reformulated via a linear splitting scheme and then be solved efficiently by using the semi-implicit scheme to discretize the time variable and the spectral method to discretize the space variables.
The numerical simulations start from random initial configurations satisfying the mass constraints.

\subsection*{Acknowledgements}  The authors gratefully acknowledge the support of the Natural Science and Engineering Research Council (Canada) though the Discovery Grants program.



\section{The local isoperimetric problem}


In this section, we discuss among all the partitions of $\mathbb{R}^2$, the minimizers of the local part $P_{\sigma}$;
that is, minimizing
\begin{equation} \label{localE}
P_\sigma(A_1, A_2, A_0): =  \sum_{0 \leq i < j \leq 2} \sigma_{ij}   \mathcal{H}^1 (\partial A_i \cap \partial A_j),
\end{equation}
among
\begin{equation} \label{partitionOme}
\{ (A_1, A_2, A_0) :     | A_1 \cap A_2 | = 0;  | A_i| = m_i , i = 1,2;  A_0 = \R^2 \setminus  \overline{( A_1 \cup  A_2)}    \}.
\end{equation}
As noted above, since $A_0$ and $\widetilde A_0=\overline{A_1 \cup A_2} $ have the same boundary, the perimeter above may be expressed in terms of $A_1, A_2$ alone, and when convenient we may replace the exterior domain with the union $\overline{A_1\cup A_2}$.  In particular, we may re-express the total perimeter as
\begin{equation}\label{peri}
\PP_\sigma (A) = P_\sigma (A_1,A_2,\overline{A_1\cup A_2}). 
\end{equation}

In case the $\sigma_{ij}$ satisfy the triangle inequalites \eqref{triangle} we have the equivalent formulation of this isoperimetric problem in terms of BV characteristic functions, given in \eqref{bvperim}.  We are also interested in the case of $\sigma_{02}>\sigma_{01}+\sigma_{12}$, in which the representation \eqref{bvperim} is not coercive in the BV norm, and in this case a more careful treatment of the minimization problem is required.

The geometry of minimizers depends strongly on the surface tension values $\sigma_{ij}$, and we consider each case separately.

\subsection{{Pattern 1:  One Double bubble}}

When we assume that the triangle inequalities \eqref{triangle} hold with {\it strict} inequalities in each, minimizers of \eqref{localE} with $m_1\neq 0\neq m_2$ are {\bf double bubbles}.  This has been proven (in any dimension) by Lawlor  \cite{lawlor}.  This geometrical problem was already studied (in the context of grain boundaries) by  Mullins \& Smith  \cite{smith1,mullins1}. Minimizers consist of smooth circular arcs which meet at triple junctions.  The angle formed  at a triple junction must satisfy Young's Law (also known as a Herring Condition):  at each triple junction point, the normal vectors $n_{ij}$ to the arc separating phases $i$ and $j$ must satisfy the balancing condition,
$$   \sum_{i\neq j} \sigma_{ij}\, n_{ij} = 0.
$$
This is equivalent to
\begin{equation}\label{herring}
 {\sin\theta_1\over \sigma_{02}} = {\sin\theta_2\over\sigma_{01}} = {\sin\theta_0\over \sigma_{12}};
 \end{equation}
see Figure \ref{angles}.

\begin{figure}[ht]
\centering
\includegraphics[scale=0.25]{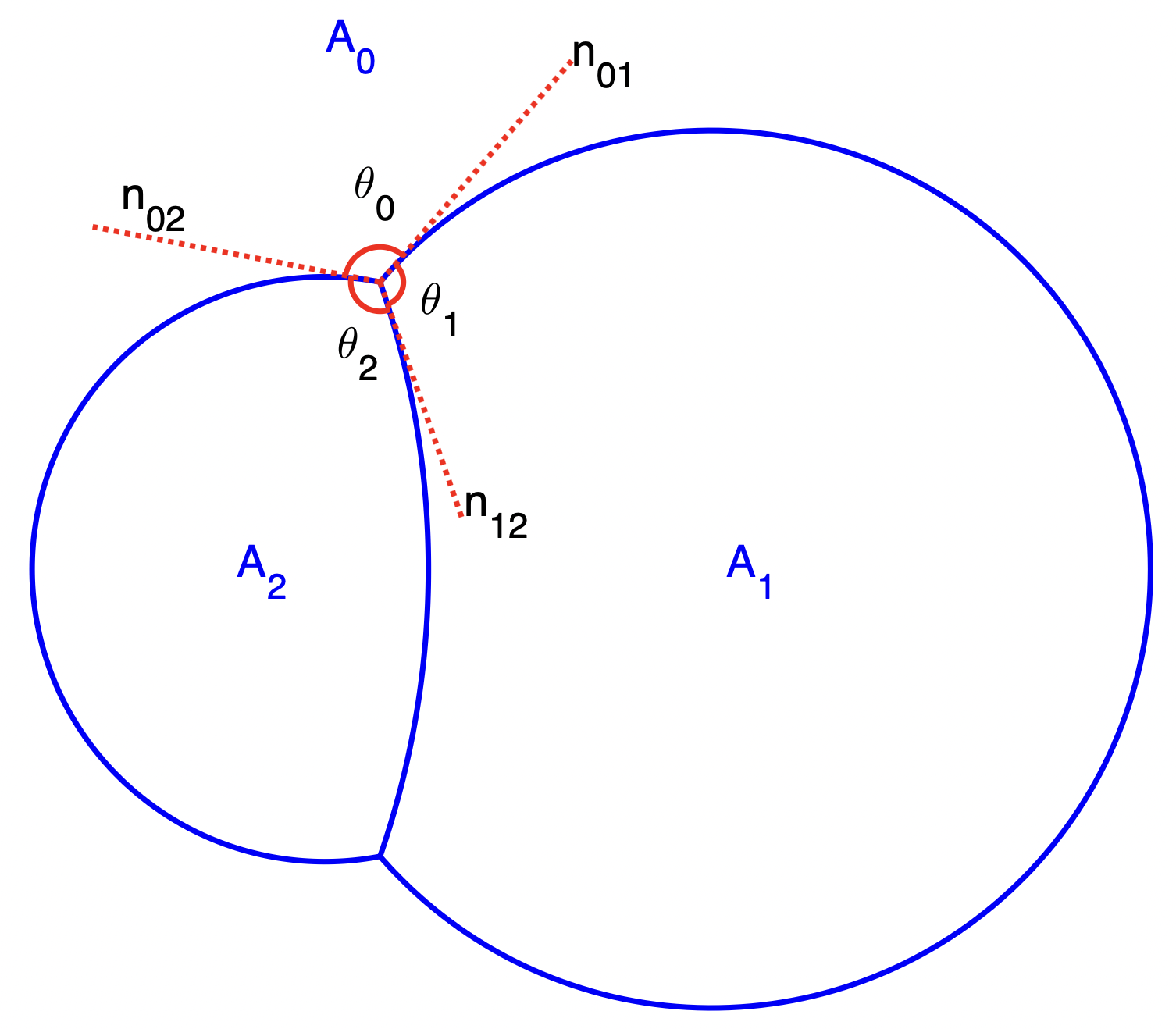}
\includegraphics[scale=0.25]{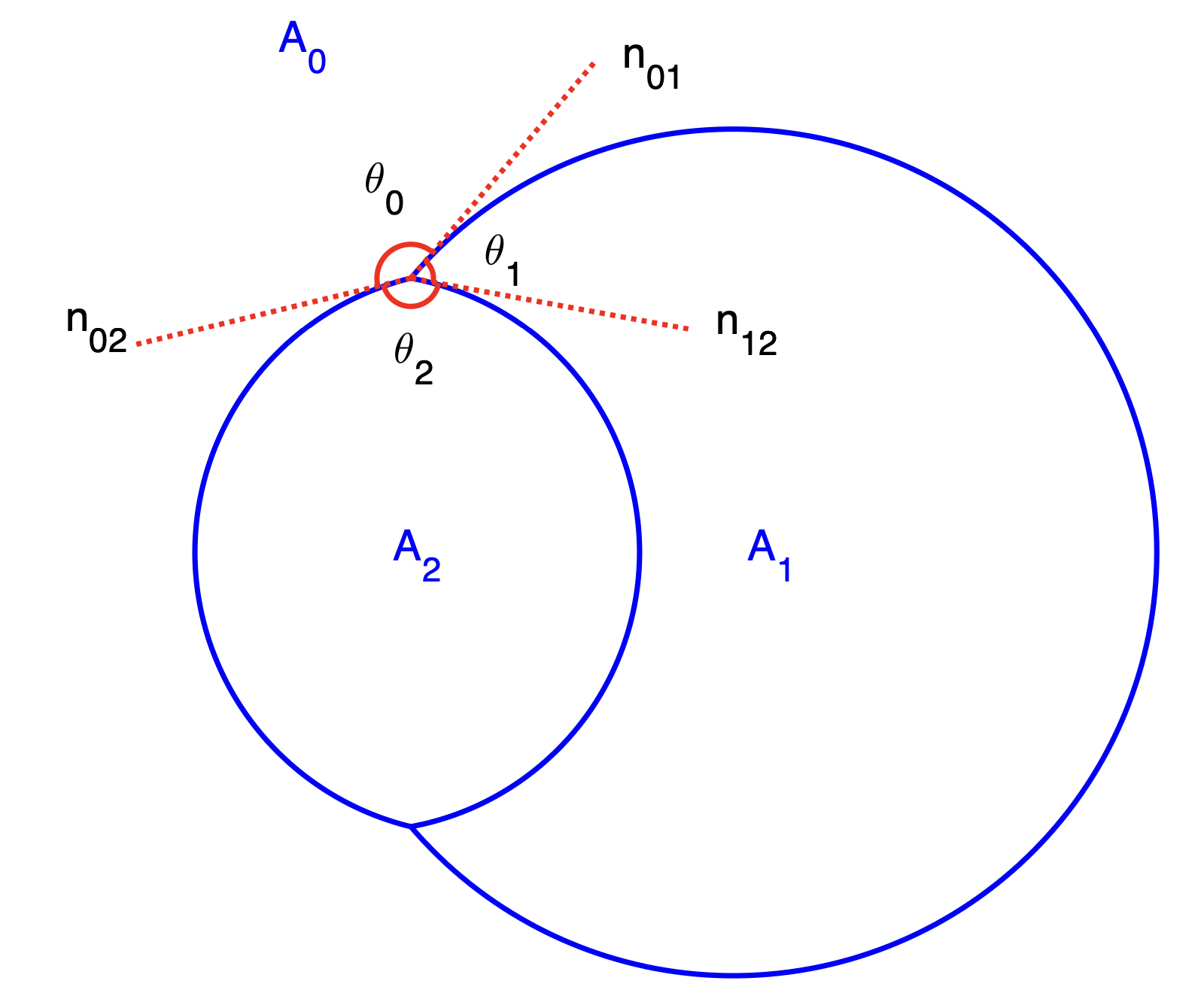}
\caption{ At triple junction points the angles $\theta_i$ between the normal vectors $n_{ki}$ and $n_{kj}$ are determined by the surface tensions $\sigma_{ij}$ via \eqref{herring}.  On the left, a symmetric double bubble, with equal values of $\sigma_{ij}$; on the right, $\sigma_{02}> \sigma_{12} = \sigma_{01}$ }. 
\label{angles}
\end{figure}

The special case of identical weights
 $\sigma_{01} = \sigma_{02} = \sigma_{12}=1$, has been studied in our previous paper \cite{ablw}.  In this case,
\begin{eqnarray}
 P_\sigma(A_{1},A_{2}, A_{0})  = \sum_{0 \leq i < j \leq 2}    \mathcal{H}^1 (\partial A_i \cap \partial A_j) = \frac{1}{2} \sum_{i=0}^2  P_{\mathbb{R}^2} (A_i),
\end{eqnarray}
which is a two component isoperimetric problem.  The standard double bubble satisfying the $120$ degree requirement at the two triple junction points (where three interfaces meet) is the unique solution to this isoperimetric problem \cite{FABHZ, db2, reichardt};
see (a) in Figure \ref{evolves} and Figure \ref{evolves2}. 

\subsection{{Pattern 2: One Core Shell}}

Core shell geometries can be expected in the degenerate case of the triangle inequality \eqref{triangle} where
\begin{equation}\label{nonstrict}
\sigma_{02}\ge\sigma_{01}+\sigma_{12}, 
\end{equation}
and that the other two inequalities in \eqref{triangle} hold strictly.  
With this hypothesis, the cost $\sigma_{02}$ of a transition from phase 2 to the background zero-phase is at least as large as a ``composite interface'' passing through phase 1, so intuitively we expect a shell structure to minimize perimeter.

\begin{defn}
	Given masses $m_1,m_2>0$, we define the class of generalized core shells, denoted $\mathcal{C}_{m_2}^{m_1}$, to consist of pairs $(A_1,A_2)$, $|A_i|=m_i$, $i=1,2$,  
	 with an inner disk $A_2$ of mass $m_2$, and an outer annulus $A_1$ of mass $m_1$. 
\end{defn}
That is, $A_2=B_{r_2}(p_2)$ and $A_1=B_{r_1}(p_1)\setminus B_{r_2}(p_2)$, with $r_2=\sqrt{M_2/\pi}$, $r_1=\sqrt{(M_1+M_2)/\pi}$, and centers $p_1,p_2\in\R^2$ chosen such that $A_2=B_{r_2}(p_2)\subset B_{r_1}(p_1)$.
Note that this definition does not require the two circles (the boundary of the disk and the outer boundary of the annulus) to be concentric; see (f) in Figure \ref{evolves} and \ref{evolves2}.
In addition, this definition permits a special case in which the two circles are tangent to each other; see (e) in Figure \ref{evolves} and \ref{evolves2}.  Indeed, the position of the inner circle $A_2$ does not affect the total perimeter of the configuration, and so the local isoperimetric problem cannot distinguish between these generalized core shells.


We expect the same geometry of minimizers in the more extreme case where $\sigma_{02}>\sigma_{01}+\sigma_{12}$.  In this situation we cannot rely on the equivalent formulation \eqref{energyu} in BV in order to ensure existence of a minimizer or lower semicontinuity, 
but we may still assert that core shell configurations $\mathcal{C}_{m_2}^{m_1}$ must have smaller perimeter than any other competitor.  Hence, we state our geometry result for both cases at once:

\begin{thm} \label{thm:coreshellexist}
	Let $\sigma_{ij}$ be given, such that 
	\[\sigma_{02}    \geq  \sigma_{01}  + \sigma_{12} , \] 
	and let $m_1,m_2>0$ be given.  Then the minimizer of $\PP_\sigma$ is a core shell, ie, of the class $\mathcal{C}_{m_2}^{m_1}$.
\end{thm}

\medskip

In the case of equality in \eqref{triangle}, 
\begin{equation}\label{equalitysigma} \sigma_{02}  = \sigma_{01}  + \sigma_{12}, 
\end{equation}
the proof is straightforward. So we start with the proof of this case.

\begin{proof}[Proof of Theorem \ref{thm:coreshellexist}, case of equality in \eqref{nonstrict}]
Applying \eqref{equalitysigma} to \eqref{localE}, for any admissible cluster $A_1, A_2, A_0$, 
\begin{eqnarray} 
P_\sigma(A_{1},A_{2}, A_{0}) &=&   \sigma_{01}  \left[  \mathcal{H}^1 (\partial A_0 \cap \partial A_1) +  \mathcal{H}^1 (\partial A_0 \cap \partial A_2) \right ] +
 \sigma_{12}  \left[  \mathcal{H}^1 (\partial A_1 \cap \partial A_2) +  \mathcal{H}^1 (\partial A_0 \cap \partial A_2) \right ]  \notag \\
&= &  \sigma_{01} P_{\R^2} (A_1 \cup A_2) +   \sigma_{12} P_{\R^2} ( A_2 ). \notag
\end{eqnarray} 
Since the area of $A_1 \cup A_2$ is fixed, that is, $| A_1 \cup A_2 | = m_1 + m_2$, by the isoperimetric inequality, $P_{\R^2} (A_1 \cup A_2)$ is minimized when $\partial  (A_1 \cup A_2) $ is a circle. Similarly, the area of $A_2$ is fixed, $P_{\R^2} ( A_2 )$ is minimized when $\partial  A_2 $ is a circle.  Since $A_2\subset A_1\cup A_2$, we obtain a core shell, of the form $\mathcal{C}_{m_2}^{m_1}$.  That is, given any admissible cluster $(A_1,A_2,A_0)$ with masses $m_1,m_2$, its weighted total perimeter is bounded below by that of a core shell in $\mathcal{C}_{m_2}^{m_1}$; {\it a posteriori} a minimizer exists, which is a core shell.  Moreover, the inequality will be strict in case either $A_1\cup A_2$ or $A_2$ are not disks, and so the class of minimizers is exactly $\mathcal{C}_{m_2}^{m_1}$.
\end{proof}

\medskip

%

When the inequality \eqref{nonstrict} is strict the situation is more delicate, as the weighted perimeter functional is no longer lower semicontinuous.  
We need some preliminary results.

\begin{lem}\label{not lower semi-continuous}
Given $\sigma_{ij}$, with $\sigma_{02}>  \sigma_{01}  + \sigma_{12}$,
then we can
construct explicit sequences
$u_{i,n}=\mathbf{1}_{A_{i,n}}$, $i=0,1,2$,  converging to some $u_i =\mathbf{1}_{A_{i}} $ in the strong $L^1$ topology, such that
\[\lim_{n\to+\infty} P_\sigma(A_{1,n},A_{2,n},A_{0,n} )
<P_\sigma(A_{1},A_{2},A_{0} ).\]
\end{lem}

\begin{proof}
Consider a sequence of configurations like in Figure \ref{hollow out and spread thin}, and we denote by $T_n$ (resp. $B_n$) the thin layer 
(resp. small ball) of type I material wrapping around the lobe of type II
constituent (resp. hollowed out from the lobe of type I material).
Then by construction, we can choose $|B_n|\to 0$ as $n\to+\infty$,
and the thickness of $T_n$ also goes to zero. 
For any $n$, the perimeter of
$A_{2,n}$ is completely insulated from $ A_{0,n}$, and
the contribution of the perimeter between type II and $T_n$
is 
\[+\sigma_{12} \H^1(\pd A_{2,n} \cap \pd T_n),\]
while that between $T_n$ and $ A_{0,n}$ is
\[+\sigma_{01} \H^1(\pd A_{0,n} \cap \pd T_n).\]
As the thickness of $T_n$ goes to zero, both
\[\H^1(\pd A_{0,n} \cap \pd T_n),\qquad
\H^1(\pd A_{2,n} \cap \pd T_n)\]
converge to $\H^1(\pd A_2 \cap \pd A_0) $. Thus
\begin{equation}
\lim_{n\to+\infty}\sigma_{12} \H^1(\pd A_{2,n} \cap \pd T_n) 
+\sigma_{01} \H^1(\pd A_{0,n} \cap \pd T_n)=
(\sigma_{01} +\sigma_{12}) \H^1(\pd A_2 \cap \pd A_0).
\label{sum in the insulated case}    
\end{equation}
But in the limit case,
the boundary of the lobe of type II constituent will
have no layer of type I constituent insulating it anymore, so it
will contribute 
\[+\sigma_{02} \H^1(\pd A_2 \cap \pd A_0) \]
to the perimeter term. This is greater than the sum
\eqref{sum in the insulated case}. As the other terms are continuous
when passing to the limit $n\to+\infty$, i.e.
\begin{align*}
    \lim_{n\to+\infty} &\sigma_{1j} \H^1(\pd( A_{1,n}\setminus B_n) \cap \pd A_{j,n} )
=\sigma_{1j} \H^1(\pd A_{1}\cap \pd A_{j} ),\qquad j=0,2\\
    \lim_{n\to+\infty}& \sigma_{01} \H^1(\pd (A_{1,n}\setminus B_n)  \cap \pd B_n  )
    =0,
\end{align*}
we infer 
\[ \lim_{n\to+\infty} P_\sigma(A_{1,n},A_{2,n},A_{0,n} )
<P_\sigma(A_{1},A_{2},A_{0} ), \]
as desired.
\end{proof}

\begin{cor}
As a consequence of Lemma \ref{not lower semi-continuous}, the full nonlocal energy $\mathcal{E}(u)$
is also not lower semicontinuous with respect to
the strong $L^1$ topology.
\end{cor}

\begin{proof}
It suffices to notice that the interaction term is continuous with respect
to the convergence from Lemma \ref{not lower semi-continuous}. 
\end{proof}

As a consequence of Lemma~\ref{not lower semi-continuous}, the existence of optimal configurations is significantly more
challenging when $\sigma_{02}   >  \sigma_{01}  + \sigma_{12}$, as it
prevents us from using most classical
arguments relying on
minimizing sequences. 
Thus, we have to rely on an ad hoc construction.


We require a geometrical lemma which proves an inner-cone condition for minimizers of $\PP_\sigma$.
We recall that, since we are working with sets $A_1,A_2,A_0$ of finite perimeter, both
$\partial A_i \cap \partial A_0$, $i=1,2$, are $\mathcal{H}^1$-rectifiable and hence Lipschitz regular.  The unit tangent vectors thus exist $\mathcal{H}^1$-a.e, but there may be corners on any of the components of the interfaces, and the following lemma restricts the sharpness of these angles.  Define the angle $\al_0$ as the the solution in $(0,\frac\pi{2})$ of 
    \[1-\sin\al - \sqrt{ \frac{\pi}{2}  \sin(2\al) } =0. \]

\begin{lem}\label{inner cone regularity of minima no triangular inequality}
	Let $m_1,m_2>0$, and 
	 $(A_1, A_2, A_0)$ in \eqref{partitionOme} be given.
If there exists
a point $x_0\in \sm:=\partial A_i \cap \partial A_j$ such that
the interior (to $A_i$) angle between the right and left tangent lines to $\sm$ at $x_0$
has amplitude $ \alpha < \alpha_0 $, then
there exists a perturbation of $(A_1, A_2, A_0)$ with 
the same respective masses but
lower (weighted) perimeter.
\end{lem}

\begin{proof}
Assume there exists
a point $x_0\in \sm:=\partial A_i \cap \partial A_j$ such that
the angle between the right and left tangent lines to $\sm$ at $x_0$
has amplitude $2\alpha>0$. Then we perturb the entire
configuration in the following way:
\begin{enumerate}
    \item first, we choose points $p_\vep, q_\vep \in \sm$ such
    that the path distance on $\sm$ between $p_\vep, q_\vep$ and
    $x_0$ is $\vep$. Due to the Lipschitz regularity, $\sm$
    is approximated in first order, near $x_0$, by its tangent line,
    therefore we have also
    \[\dt_\vep:= |x_0-p_\vep|=|x_0-q_\vep|=\vep +o(\vep). \]
    
    \item Second, we connect $p_\vep$ and $q_\vep$ with a straight
    line segment $\llbracket p_\vep,q_\vep\rrbracket$, and define
    \[ \~\sm_\vep := \sm \setminus \{\text{portion of } \sm 
    \text{ between } p_\vep \text{ and }q_\vep\} \cup 
    \llbracket p_\vep,q_\vep\rrbracket.\]
    Geometrically, it is clear that
    \[  |p_\vep-q_\vep| =2\dt_\vep \sin \al + o(\vep) , \]
    thus 
    \[  \H^1(\~\sm_\vep ) =\H^1(\sm) - 2\dt_\vep (1-\sin\al)  + o(\vep). \]
    
    \item The construction in the previous step, however,
    alters the total masses of each type of constituent. Indeed, it is clear that the mass
    of type $i$ constituent
    inside
    the region delimited by 
    \[ \{\text{portion of } \sm 
    \text{ between } p_\vep \text{ and }q_\vep\}
    \qquad \text{and} \qquad \llbracket p_\vep,q_\vep\rrbracket,\]
    whose area is
    \[A_\vep := \frac{1}{2} \vep^2 \sin(2\al) +o(\vep^2), \]
    has been simply removed, and been replaced with type $j$ constituent. To balance this issue, we remove a ball of area $A_\vep$ from
    type $j$ constituent, and replace it with 
    type $i$ constituent. Our final competitor is the set
    obtained in this way.

\end{enumerate}
By construction, such a competitor from the previous
    three steps will have the exact same total masses for each type
    of constituent as the original
    configuration. The perimeter between types $i$ and $j$ constituents
    has been decreased by $2\dt_\vep (1-\sin\al)  + o(\vep)  $ due to the construction
    in Step 2, and increased by $ 2\sqrt{\pi A_\vep}$ in Step 3.
    Since 
    \[ 2\dt_\vep (1-\sin\al) - 2\sqrt{\pi A_\vep} 
    =2\vep [ 1-\sin\al - \sqrt{\frac{\pi}{2}  \sin(2\al) }] +o(\vep), \]
    which, 
    since we assumed $\al<\al_0$, where $\al_0$ is the solution of 
    \[1-\sin\al - \sqrt{\frac{\pi}{2}  \sin(2\al) } =0, \]
    becomes positive for all sufficiently small $\vep$,  
    we get that such a construction produces a competitor with less
    total
    (weighted) perimeter.
The proof is thus complete.
\end{proof}

With this geometrical lemma we can then show that in the case of strict inequality in \eqref{nonstrict} that there should be no interfaces between phases 2 and 0.

\begin{lem}\label{lem:omega2}
Let $\sigma_{ij}$ be given, such that 
	\[\sigma_{02}    \geq  \sigma_{01}  + \sigma_{12} . \] 
	Let $m_1,m_2>0$ be given, and 
	let
	 $(A_1, A_2, A_0)$ in \eqref{partitionOme}
	be a minimizer of \eqref{localE}.
Then $\H^1(\pd A_0 \cap\pd A_2) = 0$, i.e. $A_2$ is separated from the background $ A_0$ by $ A_1$.
\end{lem}

\begin{proof}
          Case I: $\sigma_{02}  >  \sigma_{01}  + \sigma_{12} $.
	Assume the opposite, i.e. $\H^1(\pd A_0 \cap\pd A_2)>0$. We construct
	a competitor with lower perimeter in the following way:
	\begin{enumerate}
		\item first, remove a ball $B_\vep$ of area $\vep$ from $A_1$, and fill it
		with background constituent $A_0$. This step
		creates a new perimeter $\pd B_\vep$, between $A_1$ and the background $A_0$. Thus
		the perimeter is increased by 
		$$+2\sigma_{01} \sqrt{\pi\vep}.$$

\begin{figure}[ht]
\centering
\includegraphics[scale=0.4]{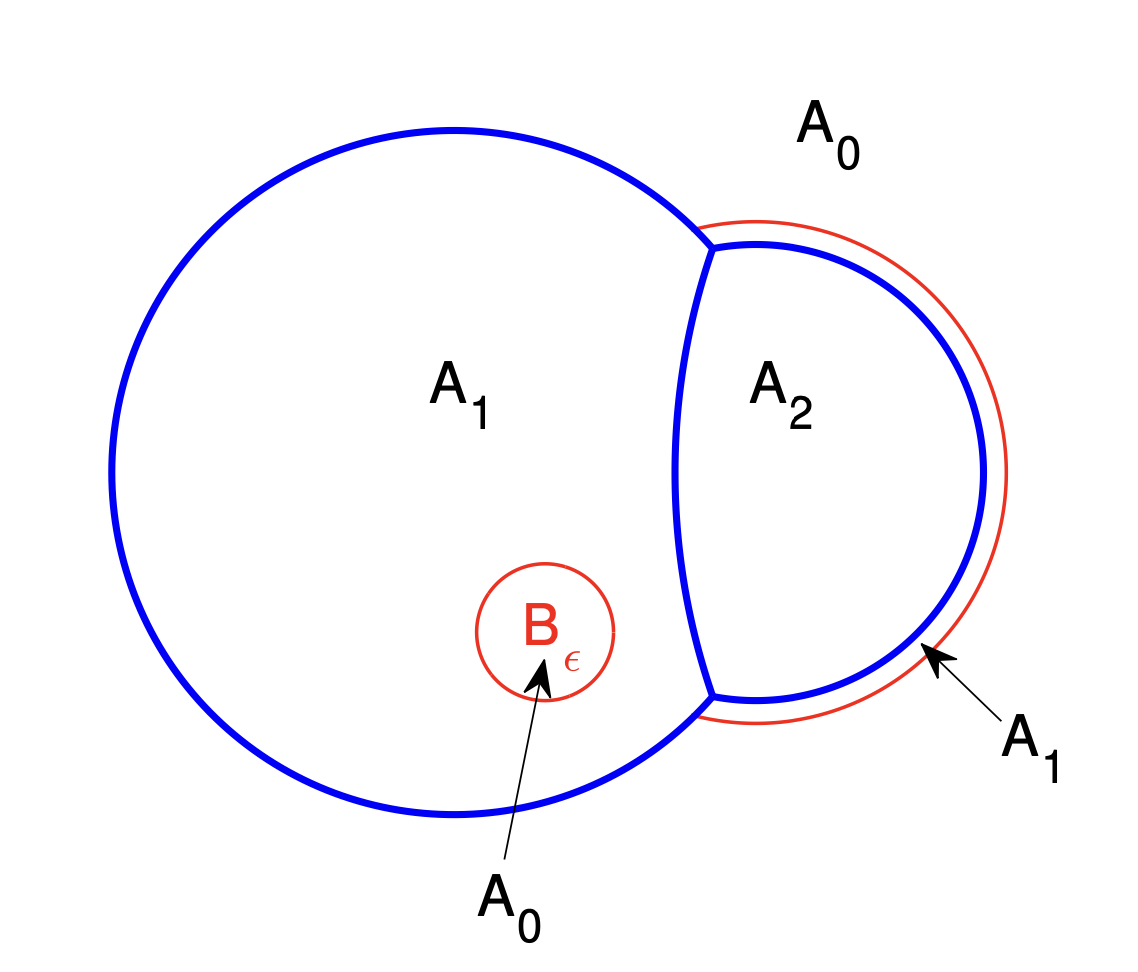}
\caption{Construction of a competitor with lower perimeter}
\label{hollow out and spread thin}
\end{figure}
		
		
		\item Next, we claim that we can add a thin layer around $\pd A_2 \cap\pd A_0$ of thickness
		$\vep/\H^1(\pd A_2 \cap\pd A_0)$, so the total mass of type I constituent is preserved.
		This adds a new perimeter, again between $A_1$ and the background $A_0$, of length
		$$+\sigma_{01}\H^1(\pd A_2 \cap\pd A_0) +f(\vep) $$ 
		for some function $f(\vep)\to 0$
		as $\vep\to 0$.
		However, this step also completely erases the former boundary $\pd A_2 \cap\pd A_0$,
		and transforms it into a boundary between $A_1$ and $ A_2$: thus the perimeter
		decreases by a term
		 \[ ( -\sigma_{02} + \sigma_{12})\H^1(\pd A_2 \cap\pd A_0)  .\]
	\end{enumerate}
Assuming that step 2 can be achieved, we complete the proof of Lemma~\ref{lem:omega2} and verify the details of the construction of the modified domain below.
The above construction induces a change in the perimeter of
\[+2\sigma_{01} \sqrt{\pi\vep} +f(\vep)
- (\sigma_{02}-\sigma_{01} - \sigma_{12})\H^1(\pd A_2 \cap\pd A_0) <0,\]
which contradicts the minimality of $(A_1,A_2, A_0)$.

It remains to verify the details of the construction in step 2 above, constructing the insulating layer as in 
Figure \ref{hollow out and spread thin}.


\medskip

{\em Step 1. Choosing the points.}
Take a curve $\Gamma$ on $  \pd A_2\cap \pd A_0$, 
and let
$\gamma:[0,1]\lra \Gamma$ be a constant speed parameterization. 
Furthermore, impose that the angle between 
the tangents
$\gamma'(0)$ and $\gamma'(s)$, $s\in [0,1]$, never exceeds $2\pi$.
Without loss of generality, assume the overall net turning from
$\gamma'(0)$ to $\gamma'(1)$ is in the counterclockwise sense, and we denote such turning
by
\[ A:=\angle \gamma'(0) \gamma'(1) \in [0,2\pi].\]
Let 
$$\Gamma_n:= \bigcup_{i=1}^n \llbracket  t_{i-1,n} ,t_{i,n}  \rrbracket,\qquad
t_{i,n}:= \gamma( s_{i,n} ),\quad s_{i,n}:=\frac{i}{n}, $$ 
be the piecewise linear curve through all the $t_{i,n}$. Note that this construction ensures 
$\H^1(\Gamma_n) \to \H^1(\Gamma)  $.
Now, for each $t_{i,n}$, 
$i=1,\cdots,n-1$,
denote by $\nu_{i,n}^\pm$ the left/right exterior (i.e. pointing towards the background) unit normal to $\Gamma_n$
at $t_{i,n}$, and let
\[ t_{i,n}^{\vep,\pm} := t_{i,n}+\vep \nu_{i,n}^\pm,\qquad i=1,\cdots,n-1. \]
For $i=0,n$, denote by $\nu_{i,n}$ the exterior unit normal to $\Gamma_n$.

\medskip

{\em Step 2. Constructing the layer.} Connect
$t_{0,n}$ to $t_{0,n}^\vep$, and $t_{n,n}$ to $t_{n,n}^\vep$, and then
 $t_{i-1,n}^{\vep,+}$ to  $t_{i,n}^{\vep,-}$, $i=1,\cdots,n$,
with line segments (with the convention $t_{0,n}^{\vep,+}=t_{0,n}^{\vep}$,
$t_{n,n}^{\vep,-}=t_{n,n}^{\vep}$). 

\begin{figure}[ht]
\centering
\includegraphics[scale=0.25]{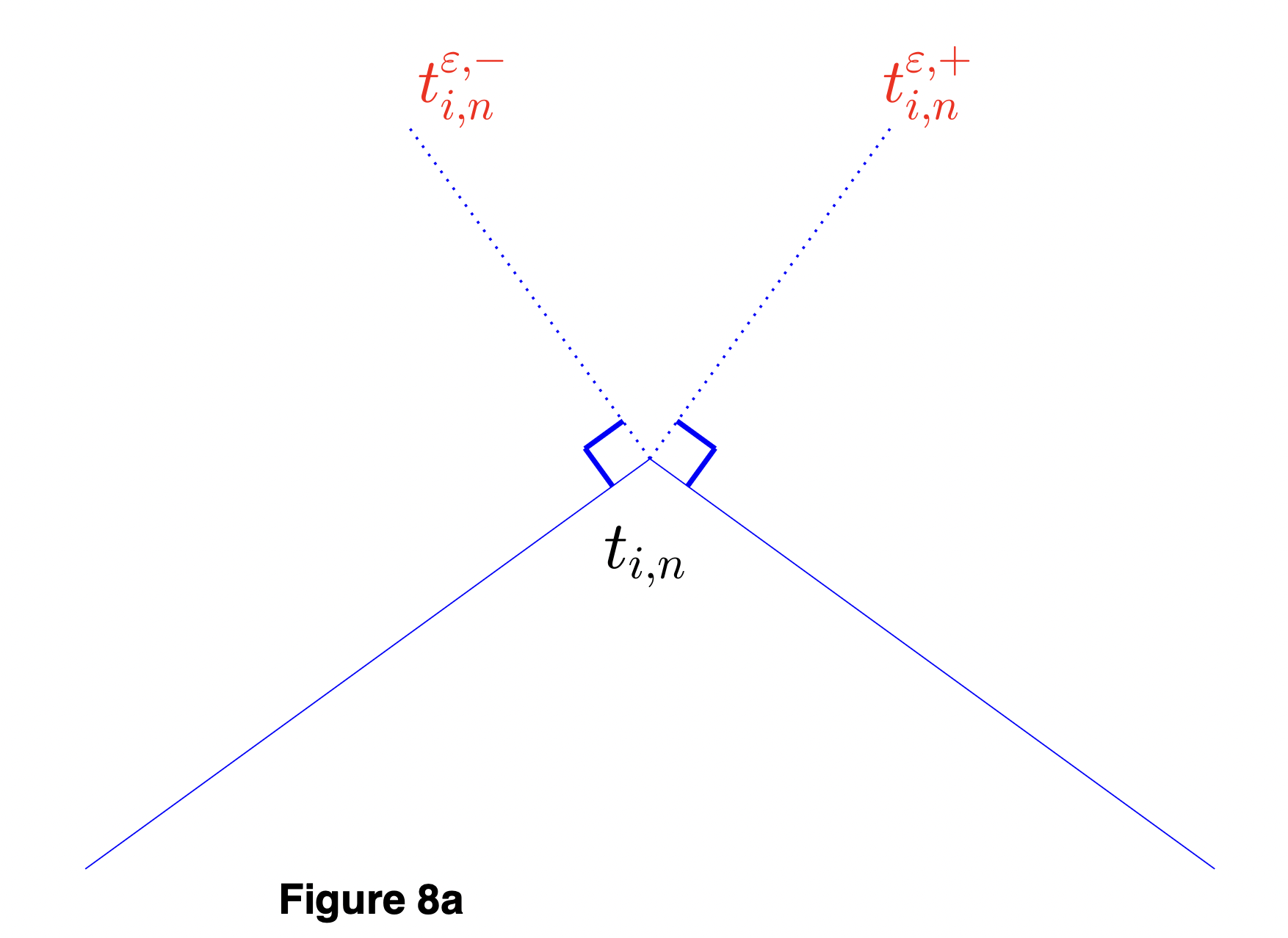}
\includegraphics[scale=0.25]{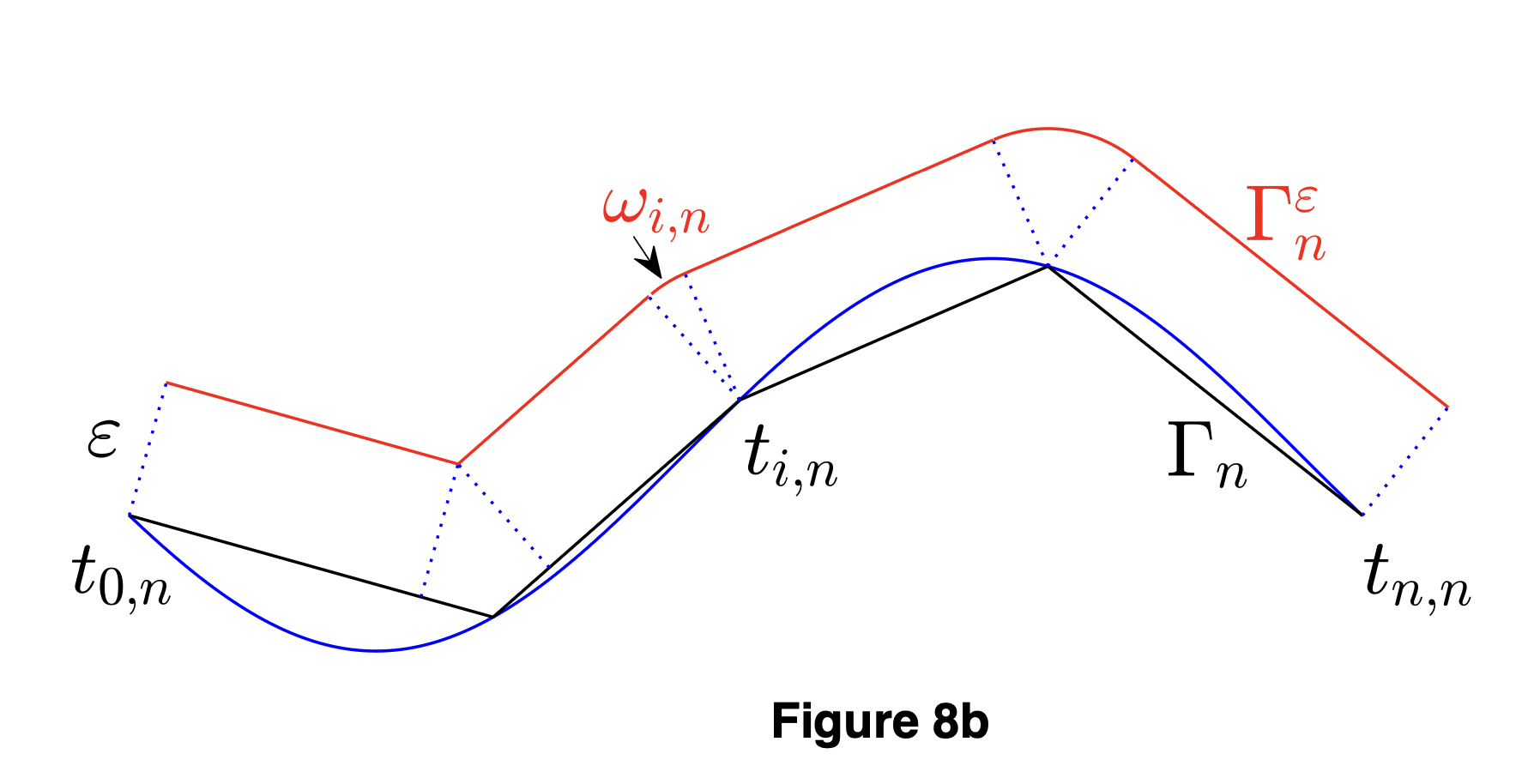}
\caption{A schematic representation of the construction.}
	\label{construction of the piecewise linear outer layer}
\end{figure}

The resulting set 
$$
\llbracket t_{0,n}^\vep, t_{0,n} \rrbracket
\cup 
\llbracket t_{n,n}^\vep, t_{n,n} \rrbracket\cup 
\bigcup_{i=1}^n \llbracket t_{i-1,n}^{\vep,+}, t_{i,n}^{\vep,-}\rrbracket $$ 
however,
might have two issues:
\begin{enumerate}
	\item first, if the angle in $t_{i,n}$ is convex, then
	the line segments $\llbracket t_{i-1,n} ^{\vep,+} , t_{i,n}^{\vep,-}\rrbracket$
and $\llbracket t_{i,n} ^{\vep,+} , t_{i+1,n}^{\vep,-}\rrbracket$ do not intersect.

\item Second, if the angle in $t_{i,n}$ is concave, then
the line segments $\llbracket t_{i-1,n}^{\vep,+}, t_{i,n}^{\vep,-}\rrbracket$ 
and $\llbracket t_{i,n}^{\vep,+}, t_{i+1,n}^{\vep,-}\rrbracket$ 
do cross each other.
\end{enumerate}
To overcome it, we do the following:
\begin{enumerate}
	\item if the angle in $t_{i,n}$ is convex, then we connect 
	the line segments $, t_{i,n}^{\vep,-}$
	and $ t_{i,n}^{\vep,+}$ with an arc of circle $\omega_{i,n}$ centered in $ t_{i,n}$.
	
	\item If the angle in $t_{i,n}$ is concave, then
	define $p_{i,n}^\vep:=  \llbracket t_{i-1,n}^{\vep,+}, t_{i,n}^{\vep,-}\rrbracket\cap \llbracket t_{i,n}^{\vep,+}, t_{i+1,n}^{\vep,-}\rrbracket$.  Then, replace 
these two segments
	with 
		$\llbracket t_{i-1,n}^{\vep,+},p_{i,n}^\vep\rrbracket\cup \llbracket p_{i,n}^\vep, t_{i+1,n}^{\vep,-}\rrbracket$
\end{enumerate}

Let $\Gamma_{n}^\vep$ be the resulting curve. By taking the limit
$n\to +\8$, we obtain a curve $\Gamma^\vep$ that
 plays the role of the outer boundary of the ``insulating layer'' from Figure \ref{hollow out and spread thin}.

\medskip

{\em Step 3. Estimating the length of }$\Gamma_{n}^\vep$. We want to estimate
from above
 the difference
$\H^1(\Gamma^\vep)-\H^1(\Gamma)$, by first estimating
$\H^1(\Gamma_n^\vep)-\H^1(\Gamma_n)$, and then pass to the limit $n\to+\8$.

Around a convex angle of amplitude $2\al$, there 
is 
\[  2\vep\bt,\qquad \bt:=\frac{\pi}{2}-\al, \]
{\em extra} length. See Figure \ref{convex angle case}.

\begin{figure}[ht]
\centering
\includegraphics[scale=0.35]{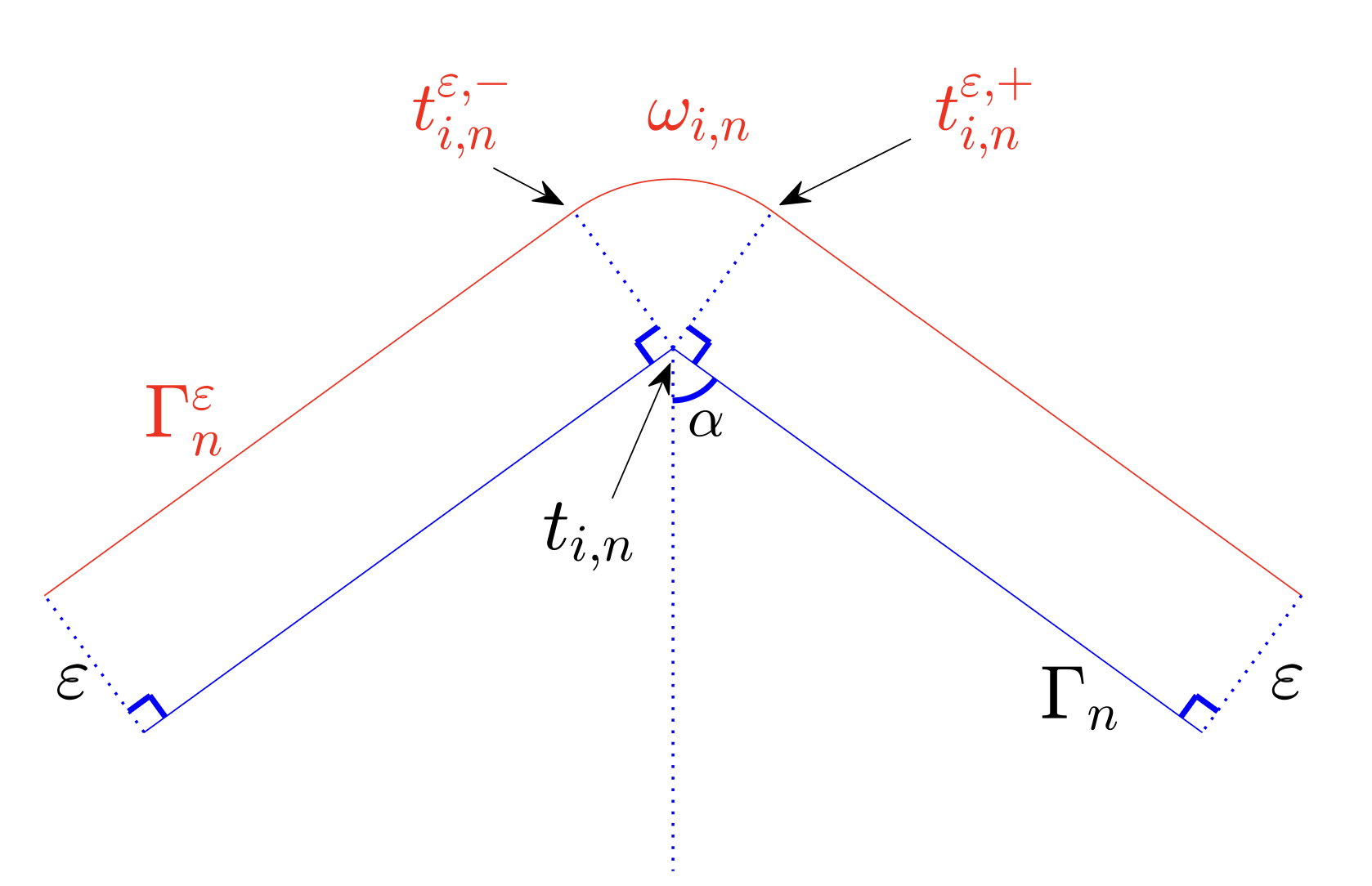}
\caption{Geometry around a convex angle}
	\label{convex angle case}
\end{figure}

Around a concave angle of amplitude $2\al$, there 
is 
\[  2\vep \tan\bt,\qquad \bt:=\frac{\pi}{2}-\al, \]
{\em less} length. See Figure \ref{concave angle case}.
\begin{figure}[ht]
\centering
\includegraphics[scale=0.35]{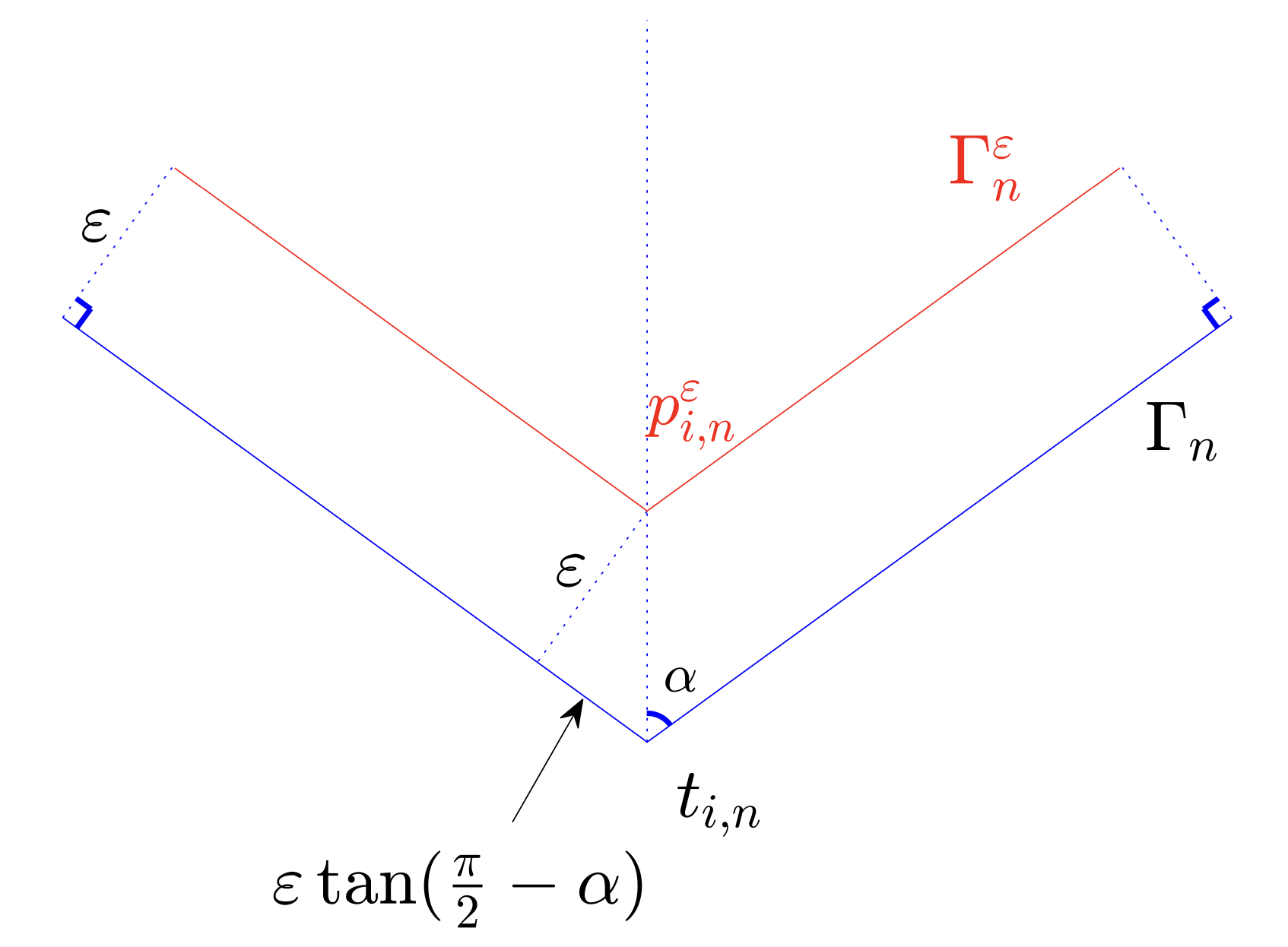}
\caption{Geometry around a concave angle}
	\label{concave angle case}
\end{figure}
Recalling that we have the extra pieces
$\llbracket t_{0,n}^\vep, t_{0,n} \rrbracket
$ and $
\llbracket t_{n,n}^\vep, t_{n,n} \rrbracket$, which might add $2\vep$ in length,
we have
\[ \H^1(\Gamma_n^\vep)-\H^1(\Gamma_n) \le
2\vep\Big[1+\sum_{i=1}^{N_+}  \bt_i^+ -\sum_{i=1}^{N_-} \tan \bt_i^-
\Big],\qquad
\bt_i^\pm :=\frac{\pi}{2}-\al_i^\pm,
\]
where the $\al_i^+$ (resp. $\al_i^-$) are the convex (resp. concave) angles.

Note that the tangent turns in opposite directions around a convex angle compared to a concave one, thus the total turning of the tangent is
\begin{align}
\label{total turning angle condition}
 A^+-A^-=A\in [0,2\pi],\qquad A^+:= \sum_{i=1}^{N_+} \bt_i^+ ,\quad
A^-:=  \sum_{i=1}^{N_-} \bt_i^- . 
\end{align}
Thus we need to bound
\[\sum_{i=1}^{N_+}  \bt_i^+ -\sum_{i=1}^{N_-} \tan \bt_i^-\]
from above, subject to \eqref{total turning angle condition}. Using the convexity of 
$\tan$, and the fact that $\tan\th\ge \th$,
\begin{align*}
\sum_{i=1}^{N_+}  \bt_i^+ -\sum_{i=1}^{N_-} \tan \bt_i^- &
\le A^+ - N^- \tan\frac{A^-}{N^-}
\le A^+ -A^-=A.
\end{align*}
Thus, combining all the above estimates,
\[ \H^1(\Gamma_n^\vep)-\H^1(\Gamma_n) \le 2(A+1)\vep. \]
The right hand side is now independent of $n$, which allows to take the $\liminf$ for $n\to+\8$:
our construction ensured $\H^1(\Gamma_n)\to \H^1(\Gamma)$,
while $\H^1(\Gamma^\vep)\le \liminf_{n\to+\8} \H^1(\Gamma_n^\vep)$,
hence
\[ \H^1(\Gamma^\vep)-\H^1(\Gamma) \le 2(A+1)\vep. \]

%

\end{proof}

We may now complete the proof of Theorem~\ref{thm:coreshellexist}, in the case of strict inequality in \eqref{nonstrict}.
\begin{proof}[Proof of Theorem~\ref{thm:coreshellexist}, remaining case]

With the conclusion of Lemma~\ref{lem:omega2}, the conclusion of Theorem~\ref{thm:coreshellexist} follows as in the case of equality.  Indeed,  since $\H^1(\pd A_0 \cap\pd A_2) = 0$, we again obtain the following identity for the total weighted perimeter,
\begin{eqnarray}
P_\sigma(A_{1},A_{2}, A_{0}) &=& 
\sigma_{01}    \mathcal{H}^1 (\partial  A_0 \cap \partial A_1) +
 \sigma_{12}   \mathcal{H}^1 (\partial A_1 \cap \partial A_2) 
\notag \\
&= &  \sigma_{01} P_{\R^2} (A_1 \cup A_2) +   \sigma_{12} P_{\R^2} ( A_2 ). \notag
\end{eqnarray}
Again, $A_2\subset A_1\cup A_2$, and each is optimized by choosing a disk of the appropriate area, and thus Theorem~\ref{thm:coreshellexist} is proven in both cases.

\end{proof}



It is interesting to think of the core shell as a limit case of the 
weighted
double bubbles as we increase the surface tension $\sigma_{02}$ to a point where equality is attained in \eqref{triangle}.  The numerical experiments in Figure \ref{evolves} and Figure \ref{evolves2}
illustrate this process. Beginning from the symmetric situation (all $\sigma_{ij}$ equal), we increase $\sigma_{02}$; so to reduce \eqref{localE}, $\H^1(\pd A_0 \cap\pd A_2)$ will decrease and we observe a weighted double bubble will be the solution; see (b) to (d) in Figure \ref{evolves} and Figure \ref{evolves2}. If we continue increasing  $\sigma_{02}$, when $\sigma_{02} = \sigma_{01} + \sigma_{12} $, a generalized core-shell with two circles tangential to each other is the minimizer; see (e) in Figure  \ref{evolves} and Figure \ref{evolves2}. 
As the location of the core shell is not determined by the geometry problem, we will see that it is the second order Gamma convergence and the interaction term $\Gamma_{12}$ that will determine the location.

\subsection{Pattern 3: Two single bubbles }

When $\sigma_{12}\ge \sigma_{01}+\sigma_{02}$, and strict inequality holds in the other two of \eqref{triangle}, we expect to have single bubble configurations.
This case is similar to the previous one, except now it is the interface between $A_1$ and $A_2$ which is effectively penalized, and minimizers should prefer to insert a layer of $A_0$ between these two components.  The consequence is that the optimal geometry separates the two minority phases into disjoint balls.

\begin{lem}\label{case3}
 Let $\sigma_{ij}$ be given, such that 
	\[\sigma_{12}    \geq  \sigma_{01}  + \sigma_{02} . \] 
	Let $m_1,m_2>0$ be given, and 
	let
	 $(A_1, A_2, A_0)$ in \eqref{partitionOme}
	be a minimizer of \eqref{localE}.
Then $\H^1(\pd A_1 \cap\pd A_2) = 0$, 
and the minimizer consists of disjoint balls $A_1$, $A_2$.
\end{lem}

\begin{proof}
We can follow the same steps as in the proof of Lemma~\ref{lem:omega2}  in case $\sigma_{12} =  \sigma_{01}  + \sigma_{02}$, but the situation is simpler when $\sigma_{12} >  \sigma_{01}  + \sigma_{02} $.
Suppose for a contradiction that there exists a minimizer of \eqref{localE} with  $\mathcal{H}^1(\partial A_1\cap\partial A_2)>0$.
Let $A_2^{\prime}$ be a translation of $A_2$, chosen such that $ \overline{A_1} \cap\overline{A_2^{\prime}}=\emptyset$, and $ A_0^{\prime}=\R^2\setminus\overline{( A_1\cup A_2^{\prime})}$.  Then, $|A_2^{\prime}|=|A_2|$, and $\mathcal{H}^1(\partial A_1\cap\partial A_2^{\prime})=0$, while the other components of the boundary have the same perimeter as before. 
Thus,
$$
P_\sigma(A_{1},A_{2}^{\prime},A_{0}^{\prime} )
< P_\sigma(A_{1},A_{2},A_{0} ),
$$
as $\sigma_{01}  + \sigma_{02} < \sigma_{12} $,
which contradicts the minimality of $(A_1,A_2, A_0)$.

Now that we have established $\H^1(\pd A_1 \cap\pd A_2) = 0$, the total perimeter splits into the weighted sum of the perimeters of $A_1$ and $A_2$, and each is minimized independently, resulting in a disjoint union of two balls of the given masses.
\end{proof}

As in the core shell case, the degeneracy of the weighted perimeter is felt through the nonuniqueness of minimizing configurations; the relative positions of the two bounded components of a minimizing cluster is arbitrary.

\section{The Geometry of Core Shell Configurations}

In this section we study the combined effect of the local (weighted) isoperimetric energy and the nonlocal interaction energy in the formation of core shell assemblies obtained by minimization of $E_\eta$ (see \eqref{Eeta1}) in the droplet regime limit.  Following our analysis of the isoperimetric problem in the previous section, this entails making the choice 
\begin{equation}\label{equality}   \sigma_{02}   =  \sigma_{01}  + \sigma_{12}  
\end{equation}
in \eqref{triangle}, in order that core shells are energetically preferred in $P_\sigma$.
As described in the Introduction, in Theorem~\ref{firstlimit} we will prove a $\Gamma$-convergence result for $E_\eta$ as $\eta\to 0$ to the limiting energy defined in \eqref{ezero}, \eqref{ezerobar}, for any $\sigma_{ij}$ satisfying the triangle inequalities \eqref{triangle}, including the case of equality above.  

\begin{figure}[ht]
\centering
\includegraphics[scale=0.2]{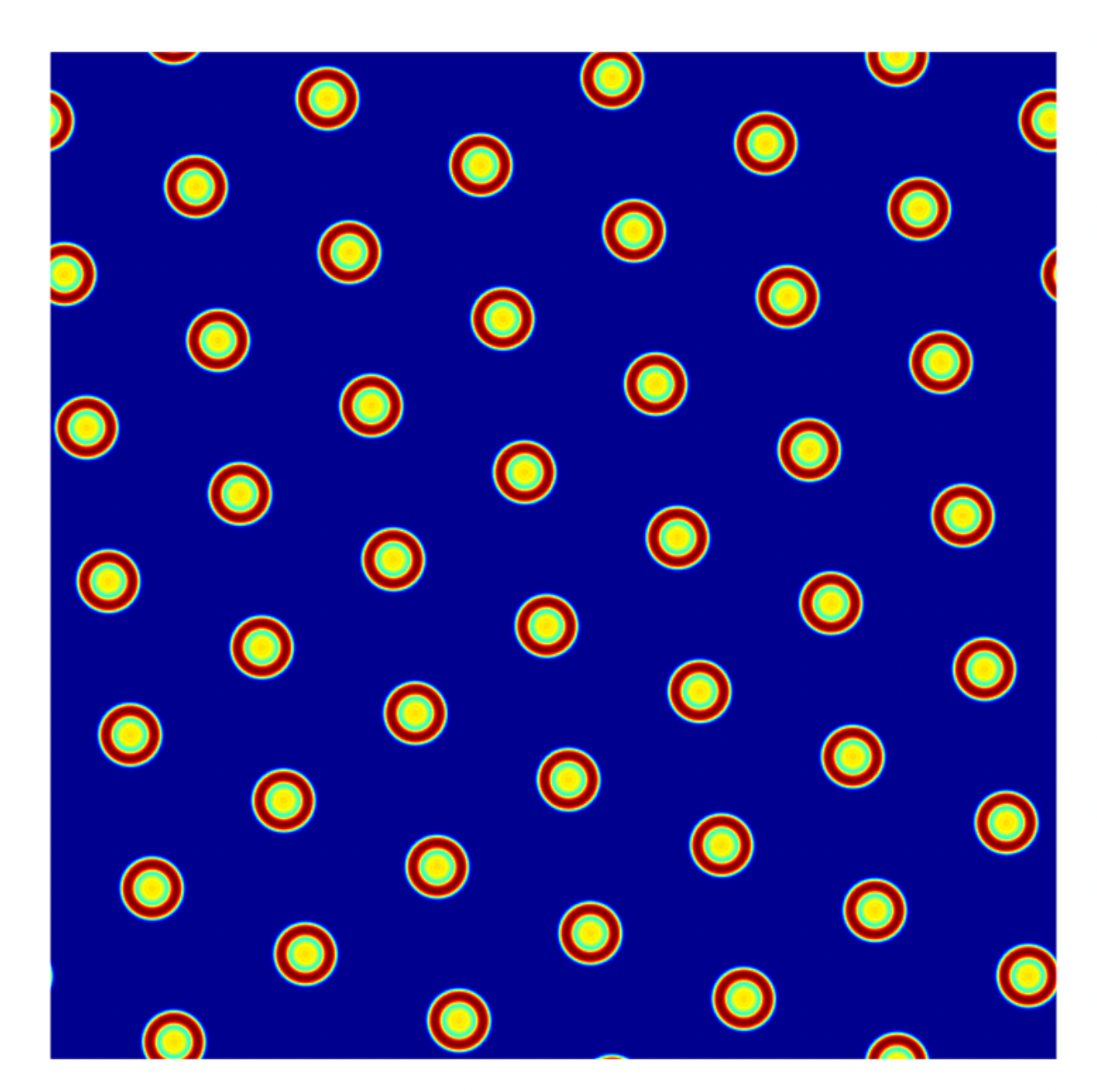}
\caption{ Another numerical simulation when $\sigma_{02}   =  \sigma_{01}  + \sigma_{12}$. These core shells are well-organized: concentric, of equal size, and distributed in a hexagon pattern.
Here $\sigma_{02} = 2$, $ \sigma_{01}=\sigma_{12}=1$, $ M_1 = 0.12$, $M_2 = 0.06$, $\gamma_{11} = 20,000$, $\gamma_{12}=\gamma_{21}=0$, $\gamma_{22} = 100,000$. }
\label{multiCoShe}
\end{figure}

By Theorem~\ref{thm:coreshellexist}, with the choice \eqref{equality} we have an explicit form for the local part of the energy,
\begin{eqnarray}
&& \inf  \left \{  \sum_{0 \leq i < j \leq 2} \sigma_{ij}   \mathcal{H}^1 (\partial A_i \cap \partial A_j)   : |A_i| =  m_i , i = 1,2;
A_0 = \mathbb{R}^2 \setminus \overline{(A_1 \cup A_2 )}  \right \}  \notag \\
&=& 2 \sigma_{01} \sqrt{\pi (m_1 + m_2 )} + 2 \sigma_{12} \sqrt{\pi m_2}
\end{eqnarray}
which represents the total perimeter of a generalized core-shell when $m_1, m_2 > 0$. Then, substituting in \eqref{Gsigma} and \eqref{ezero},
\begin{eqnarray} \label{e0m}
e_0 (m) &=& 2 \sigma_{01} \sqrt{\pi (m_1 + m_2 )} + 2 \sigma_{12} \sqrt{\pi m_2} +   \sum_{i,j=1}^2 \frac{\Gamma_{ij} m_i m_j}{4 \pi}.
\end{eqnarray}
In the case of $m_1 > 0, m_2 = 0$,
\begin{eqnarray} \label{e0ms1}
e_0 (m) =  e_0 (m_1, 0 ) = 2 \sigma_{01} \sqrt{\pi m_1}  +  \frac{\Gamma_{11} m_1^2 }{4 \pi}.
\end{eqnarray}
In the case of $m_1 = 0, m_2 > 0$,
\begin{eqnarray} \label{e0ms2}
e_0 (m) =  e_0 (0, m_2 ) = 2 \sigma_{01} \sqrt{\pi m_2}  + 2 \sigma_{12} \sqrt{\pi m_2} + \frac{\Gamma_{22} m_2^2 }{4 \pi} = 2 \sigma_{02} \sqrt{\pi m_2} + \frac{\Gamma_{22} m_2^2 }{4 \pi}.
\end{eqnarray}
In the above two cases, a core-shell is degenerated to a single bubble.

As in the previous studies of the droplet scaling for binary \cite{bi1} and ternary systems \cite{ablw}, the limiting minimization problem \eqref{ezerobar} is very subtle since the division of the total masses $M = (M_1, M_2)$ is determined by minimization itself.  In particular, although core shells are favored when both constituents $m_1,m_2>0$ are nontrivial, minimizers may exhibit a mixed state of core shells and single bubbles.  In addition, it is reasonable to expect that minimizers of $\overline{e_0}(M)$ will only have finitely many connected components, and that each constituent bubble should have a minimum size; these facts are known for the binary case, and in some parameter regimes, for unweighted ternary systems \cite{bi1,ablw}.
Such results are the goal of this section.

We recall \eqref{ezerobar}
\begin{eqnarray}
\overline{e_0}(M)=
\inf  \left \{ \sum_{k=1}^{\infty}  e_0 (m^k)  : m^k = (m_1^k, m_2^k), \ m_i^k \geq 0, \sum_{k=1}^{\infty} m_i^k = M_i, i = 1, 2 \right \}, \notag
\end{eqnarray}
where $M=(M_1,M_2)$.

First we show that having too small a mass is not energetically advantageous.

\begin{lemma}
	\label{lower bound on the mass of each core-shell}
		There exists a lower bound $m^-$
		depending on 
		\[  M_i, \; \sigma_{01}, \text{ and } \ggm_{ij}, \qquad i,j = 1, 2,  \]
		 such that no single bubble
	or core shell in a minimizing configuration can have total mass less than
	$$m^- :=\frac{  32 \pi^3 \sigma_{01}^2 }{(1+\sqrt{2})^2  (\ggm_{11}+2\ggm_{12}+\ggm_{22})^2  (M_1+M_2)^2}.$$
\end{lemma}

\begin{proof}
	Consider an arbitrary minimizing configuration. If there is only one bubble (single bubble or core shell)
	then the thesis is trivial. 
%
%
Pick two bubbles  (single bubble or core shell) whose
	masses are $m_i^k$, $i,k=1,2$, where the index $i$ denotes the constituent type, and
	$m_1^k$ (resp. $m_2^k$) is the ``outer'' (resp. ``inner'') shell.
We allow for
 $m_i^k= 0$ corresponding to the case of a single bubble.
	Their
	energy
	contribution is thus
	\begin{align}
	\label{energy contribution separate}
	\sum_{k=1}^2 2\sqrt{\pi} \Big[\sigma_{01}  \sqrt{m_1^k +m_2^k}
	+\sigma_{12}\sqrt{ m_2^k} \Big] + \sum_{i,j,k=1}^2 \frac{\ggm_{ij}}{4\pi} m_i^k m_j^k.
	\end{align}
	Combining them into a bubble, whose inner disk
	(resp. outer shell)
	has mass $m_2^1 +m_2^2$ (resp. $m_1^1 +m_1^2$), then the energy contribution
	becomes
	\begin{align}
	\label{energy contribution combined}
	2\sqrt{\pi} \Bigg[\sigma_{01}  \sqrt{\sum_{k=1}^2 (m_1^k +m_2^k) }+\sigma_{12}
	\sqrt{ \sum_{k=1}^2 m_2^k } \Bigg] + \sum_{i,j=1}^2 \frac{\ggm_{ij}}{4\pi} \Big(  \sum_{k=1}^2 m_i^k\Big)
	\Big(  \sum_{k=1}^2 m_j^k\Big).
	\end{align}
%
%
	Denote $m^k:=m_1^k+m_2^k$, $k=1,2$. By
	 subtracting
	\eqref{energy contribution separate} from
	 \eqref{energy contribution combined} we get
	 \begin{align*}
	2\sqrt{\pi} \Bigg[\sigma_{01} & \sqrt{\sum_{k=1}^2 (m_1^k +m_2^k) }+\sigma_{12}
\sqrt{ \sum_{k=1}^2 m_2^k } \Bigg] + \sum_{i,j=1}^2 \frac{\ggm_{ij}}{4\pi} \Big(  \sum_{k=1}^2 m_i^k\Big)
\Big(  \sum_{k=1}^2 m_j^k\Big)\\
&-\sum_{k=1}^2 2\sqrt{\pi} \Big[\sigma_{01}  \sqrt{m_1^k +m_2^k}
+\sigma_{12}\sqrt{ m_2^k} \Big] - \sum_{i,j,k=1}^2 \frac{\ggm_{ij}}{4\pi} m_i^k m_j^k \\
& = 
	2\sqrt{\pi} \sigma_{01} [\sqrt{m^1 +m^2 } -   \sqrt{m^1}-\sqrt{m^2} ]   + 2\sqrt{\pi} \sigma_{12} [\sqrt{m_2^1 +m_2^2 } -   \sqrt{m_2^1}-\sqrt{m_2^2} ] \\
	& \quad +\frac{1}{2\pi}
	\Big[  \ggm_{11} m_1^1 m_1^2  +\ggm_{12} (  m_1^1  m_2^2 + m_1^2m_2^1  
	) + \ggm_{22} m_2^1m_2^2 \Big].
	 \end{align*}
By the optimality of our initial configuration, we need the above term to be nonnegative, i.e.
	 \begin{align}
	 0 & \le 2\sqrt{\pi} \sigma_{01} [\sqrt{m^1 +m^2 } -   \sqrt{m^1}-\sqrt{m^2} ] + 2\sqrt{\pi} \sigma_{12} [\sqrt{m_2^1 +m_2^2 } -   \sqrt{m_2^1}-\sqrt{m_2^2} ] \notag \\
	& \quad +\frac{1}{2\pi}
	 \Big[  \ggm_{11} m_1^1 m_1^2  +\ggm_{12} (  m_1^1  m_2^2 + m_1^2m_2^1  
	 ) + \ggm_{22} m_2^1m_2^2 \Big].
	 \label{difference must me nonnegative}
	 \end{align}
Thus,
\[ 0  \leq  2\sqrt{\pi} \sigma_{01} [\sqrt{m^1 +m^2 } -   \sqrt{m^1}-\sqrt{m^2} ] + \frac{1}{2\pi}
	 \Big[  \ggm_{11} m_1^1 m_1^2  +\ggm_{12} (  m_1^1  m_2^2 + m_1^2m_2^1  
	 ) + \ggm_{22} m_2^1m_2^2 \Big] \]
Assume without loss of generality that $m^1\ge m^2$. 
Note that
	 \begin{align*}
  \sqrt{m^1}+\sqrt{m^2} -	 	 \sqrt{m^1 +m^2 }  & =
  \frac{ 2 \sqrt{m^1 m^2 }}{  \sqrt{m^1}+\sqrt{m^2} +	\sqrt{m^1 +m^2 } }
  \ge \frac{2}{2+\sqrt{2}  } \sqrt{m^2}, 
	 \end{align*}
	 and
	 \begin{align*}
	  \ggm_{11} m_1^1 m_1^2  +\ggm_{12} (  m_1^1  m_2^2 + m_1^2m_2^1  
	 ) + \ggm_{22} m_2^1m_2^2
	 &\le  \ggm_{11} m^1 m^2  +\ggm_{12} (  m^1  m^2 + m^2 m^1 
	 ) + \ggm_{22} m^1 m^2\\
	 &=(\ggm_{11}+2\ggm_{12}+\ggm_{22})  m^1 m^2\\
& \leq (\ggm_{11}+2\ggm_{12}+\ggm_{22})  (M_1 + M_2 ) m^2
	 \end{align*}
	 Thus \eqref{difference must me nonnegative} becomes
	 \begin{align*}
0&\le -\frac{  4 \sqrt{\pi} \sigma_{01} }{2+\sqrt{2}  } \sqrt{m^2} +
\frac{1}{2\pi} (\ggm_{11}+2\ggm_{12}+\ggm_{22}) ( M_1 + M_2 ) m^2\\
&\Lra m^2 \ge 
\frac{  32 \pi^3 \sigma_{01}^2 }{(1+\sqrt{2})^2  (\ggm_{11}+2\ggm_{12}+\ggm_{22})^2  (M_1+M_2)^2},
	 \end{align*}
	 concluding the proof since $m^1 \geq m^2 $.
\end{proof}

\begin{corollary}(Finiteness)\label{finiteness}
For any $M= (M_1, M_2), \ M_1, M_2 > 0$, a minimizing configuration for $\overline{e_0}(M)$ has finitely many nontrivial components.  That is, there exist $ N<\infty$ and pairs $m^1,\dots, m^N$, with $m^k=(m_1^k,m_2^k)\neq (0,0)$, for which $\overline{e_0}(M)=\sum_{k=1}^N e_0(m^k)$.
\end{corollary}


\begin{proof}
	Lemma \ref{lower bound on the mass of each core-shell} gives that each single bubble / core shell
	must have total mass at least $m^-$. Since the total combined mass of types I and II constituents
	is $M_1+M_2$, we have
	\[  \# \text{core shells}  +\# \text{single bubbles}
	\le\frac{M_1+M_2}{m^-} =:N . \]
\end{proof}

%
%

\begin{lemma}
	There exist computable lower bounds $m^-_i$, 
	depending on 
\[   M_i, \sigma_{01}, \sigma_{12}, \text{ and } \ggm_{ij},\qquad i,j=1,2, \]
	such that any bubble of type $i$ constituent, be it a single bubble,
	or lobe in a core shell, in a minimizing configuration must have mass at least
$m^-_i$.
\end{lemma}

\begin{proof}
 If type $i$ constituent, $i=1,2$, is entirely in one 
 bubble, then 
   the lower bound on $m_i^-$ is automatically true. 
%
The proof is slightly different between the cases where there are at least two core shells, and where there is only a single core shell with
the rest being single bubbles.

\medskip

{\em Case 1: at least two core shells.}
Denote by
$m_i^k$, $k=1,2$ the masses of their lobes. 
By Lemma \ref{lower bound on the mass of each core-shell} we have $m_1^k+m_2^k \ge m^-$, $k=1,2$, and assume that
$m_2^2$ is the smallest (for the other scenarios, the proof is similar). 
 We need to show that $m_2^2$ cannot be both too small. The energy contribution
of these two bubbles is 
\begin{align}
\label{energy contribution separate 2}
\sum_{k=1}^2 2\sqrt{\pi} \Big[\sigma_{01}  \sqrt{m_1^k +m_2^k}
+\sigma_{12}\sqrt{ m_2^k} \Big] + \sum_{i,j,k=1}^2 \frac{\ggm_{ij}}{4\pi} m_i^k m_j^k.
\end{align}
Now we do the following construction: move the entire bubble of type II constituent, with mass $m_2^2$, from the second bubble to the first, 
Then take the equivalent amount $m_2^2$ of type I constituent from the first bubble and move it to the second.
Therefore, this construction a core shell with inner disk (resp. outer
shell) of mass $m_2^1+m_2^2$ (resp. $m_1^1-m_2^2  $), and 
a single bubble of type I constituent of mass $m_1^2+m_2^2  $. Their energy contribution is thus
\begin{align}
2\sqrt{\pi}  \Big( \sigma_{01}\sqrt{m_1^1 +m_2^1} + \sigma_{12}
\sqrt{m_2^1+m_2^2} \Big) &+\frac{\ggm_{11}(m_1^1-m_2^2)^2
	+2\ggm_{12}(m_1^1-m_2^2)(m_2^1+m_2^2)+
	\ggm_{22} (m_2^1+m_2^2)^2  }{4\pi} \notag\\
&+2\sqrt{\pi} \sigma_{01} \sqrt{m_1^2+m_2^2  }
+\frac{\ggm_{11} (m_1^2+m_2^2)^2 }{4\pi}.
\label{energy contribution perturbed 2}
\end{align}
Subtracting \eqref{energy contribution perturbed 2} from \eqref{energy contribution separate 2}
gives
\begin{align*}
2\sqrt{\pi} \sigma_{12} &\Big[\sqrt{m_2^1}   +\sqrt{m_2^2}   -  \sqrt{m_2^1+m_2^2}    \Big] \\
&  -  \frac{1}{2\pi} \bigg[  \ggm_{11} ( { m_1^2 - m_1^1  + m_2^2 } ) 
+\ggm_{12} ( m_1^1 -  m_2^1 - m_1^2  - m_2^2     )  + \ggm_{22}  m_2^1 
\bigg] m_2^2,
\end{align*}
and due to the optimality of our initial configuration, such a difference must be negative. Thus
we need
\begin{align*}
0&\ge 2\sqrt{\pi} \sigma_{12} \Big[\sqrt{m_2^1}   +\sqrt{m_2^2}   -  \sqrt{m_2^1+m_2^2}    \Big] 
 -\frac{1}{2\pi} \bigg[  \ggm_{11} ( m_1^2 -m_1^1 +m_2^2  ) 
+\ggm_{12} ( m_1^1 -m_2^1 -m_1^2  - m_2^2  )  + \ggm_{22}  m_2^1 
\bigg] m_2^2\\
&\ge
2\sqrt{\pi} \sigma_{12} \frac{2 \sqrt{m_2^1m_2^2}  }{\sqrt{m_2^1}   +\sqrt{m_2^2}   + \sqrt{m_2^1+m_2^2}}
-\frac{1}{2\pi} \Big[ \ggm_{11}  ( m_1^2+m_2^2) +\ggm_{12}  m_1^1+  \Gamma_{22} m_2^1\Big] m_2^2,
\end{align*}
and using 
\[ m_2^1\ge m_2^2 \]
the previous line gives
\begin{align*}
0&\ge  2\sqrt{\pi} \sigma_{12} \frac{2 \sqrt{m_2^1m_2^2}  }{\sqrt{m_2^1}   +\sqrt{m_2^2}   + \sqrt{m_2^1+m_2^2}}
-\frac{1}{2\pi} \Big[ \ggm_{11} ( m_1^2+m_2^2) +\ggm_{12}  m_1^1 + \Gamma_{22} m_2^1 \Big] m_2^2\\
&\ge
   \frac{4 \sqrt{\pi}\sigma_{12}  }{ (2+\sqrt{2})  }\sqrt{m_2^2}
-\frac{1}{2\pi} \Big[ \ggm_{11}  (M_1 + M_2) +\ggm_{12} M_1 + \ggm_{22} M_2  \Big] m_2^2.
\end{align*}
Therefore, for such inequality to hold, we need
\[ \frac{1}{2\pi} \Big[\ggm_{11}  (M_1 + M_2) +\ggm_{12} M_1 + \ggm_{22} M_2 \Big] \sqrt{m_2^2}
\ge   \frac{4 \sqrt{\pi}\sigma_{12}  }{ (2+\sqrt{2})  },
 \]
and the proof of this case is complete.

\medskip
{\em Case 2: only one core shell.}
Using the same arguments from \cite{bi1}, we may conclude that
    all single bubbles of the same type constituent have the same mass.
 Several cases are possible.
\begin{enumerate}
    \item 
    If there are two
single bubbles of type $i$ constituent, both with mass $m^k$,
then combining them into one single bubble changes the energy by
\begin{align*}
    2\sqrt{\pi} \sigma_{0i} & [\sqrt{2m^k}-2\sqrt{m^k} ]
    -\frac{\ggm_{ii} }{2\pi} [2|m^k|^2-|2m^k|^2]
    =
    -2\sqrt{\pi}(2-\sqrt{2}) \sigma_{0i} \sqrt{m^k}
    +\frac{\ggm_{ii} }{\pi} |m^k|^2.
\end{align*}
Such change cannot be negative, as it would contradict the optimality
of the initial configuration, hence a necessary condition
is
\[2\pi\sqrt{\pi}(2-\sqrt{2}) \sigma_{0i} \ggm_{ii}^{-1}
    \le  |m^k|^{3/2} ,\]
thus prohibiting $m^k$ from being too small. 
Then, by noting that
the Euler-Lagrange equation contain terms that diverge when a mass
gets too small, we obtain a computable lower
on the masses.  

\item If there is only one other single bubble, 
then the entire configuration is made of a core shell and a single bubble.

	We first show that, in the core-shell, the type I constituent is 
	forming the outer annulus
	contacting the background, while the type II constituent is always concentrated in a ball. 
To this aim, we need to compare the energies of the following configurations:
\begin{enumerate}
	\item a single bubble (of mass $m_3$, and whatever constituent type), plus a core-shell where the type I constituent (of mass $m_1$) is 
	forming the outer annulus and the type II constituent (of mass $m_2$) is concentrated in a ball.
	
	\item The same single bubble (of mass $m_3$, and whatever constituent type), plus a core-shell where the type II constituent (of mass $m_2$) is 
	forming the outer annulus and the type I constituent (of mass $m_1$) is concentrated in a ball.
\end{enumerate}
The energy contribution of the single bubble is the same in both cases, as well as the interaction terms. Thus only the (weighted) perimeters of the core-shells are different:
in Case (a), this is equal to
\begin{align}
\label{energy case a}
 2\sqrt{\pi} ( \sigma_{12}\sqrt{m_2} + \sigma_{01}\sqrt{m_1+m_2}), 
\end{align}
while in Case (b), this is equal to	
\[ 2\sqrt{\pi} ( \sigma_{12}\sqrt{m_1} + \sigma_{02}\sqrt{m_1+m_2}), \]
which, since $\sigma_{02} = \sigma_{01} +\sigma_{12} $, is equal to
\begin{align}
\label{energy case b}
 2\sqrt{\pi} ( \sigma_{12}\sqrt{m_1} + \sigma_{01}\sqrt{m_1+m_2}  +\sigma_{12}\sqrt{m_1+m_2} ). 
\end{align}
It is clear that \eqref{energy case b} is always larger than \eqref{energy case a}, hence
Case (b) is energetically unfavorable.

\bigskip

Thus we are left with two cases to consider:
\begin{enumerate}
	\item a single bubble (of type I constituent and mass $M_1-\eta$), plus a core-shell where the type I constituent (of mass $\eta$) is 
	forming the outer annulus and the type II constituent (of mass $M_2$) is concentrated in a ball.
	The energy is thus
	\begin{align*}
	f_1(\eta) & = 2\sqrt{\pi} [  \sigma_{12} \sqrt{M_2} + \sigma_{01} \sqrt{M_2+\eta} 
	+\sigma_{01} \sqrt{M_1-\eta} 
	]  \\
	&\qquad+
	\frac{1}{4\pi} \Big[  \Gamma_{11}(  (M_1-\eta)^2+\eta^2 ) +2\Gamma_{12}M_2\eta
	+\Gamma_{22}M_2^2 \Big],
	\end{align*}
	and 
	\begin{align*}
	f_1'(\eta) & = \sigma_{01}\sqrt{\pi} \Big(  \frac{1}{ \sqrt{M_2+\eta} }  
-\frac{1}{ \sqrt{M_1-\eta} }
	\Big)   +\frac{1}{2\pi} \Big[  \Gamma_{11}(  2\eta-M_1 ) +\Gamma_{12}M_2
 \Big].
	\end{align*}
	Since $	f_1'(\eta)\to -\infty$ as $\eta \nearrow M_1$, we conclude that the single ball 
	cannot be too small. In order to prevent critical configurations arising for very small (yet positive)
	$\eta$, we notice that the (algebraic) equation
	\begin{align*}
0=f_1'(\eta) & = \sigma_{01}\sqrt{\pi} \Big(  \frac{1}{ \sqrt{M_2+\eta} }  
-\frac{1}{ \sqrt{M_1-\eta} }
\Big)   +\frac{1}{2\pi} \Big[  \Gamma_{11}(  2\eta-M_1 ) +\Gamma_{12}M_2
\Big]
\end{align*}
can be transformed in some 8th order polynomials by repeatedly taking the squares:
\begin{align}
0=f_1'(\eta) & = \sigma_{01}\sqrt{\pi} \Big(  \frac{1}{ \sqrt{M_2+\eta} }  
-\frac{1}{ \sqrt{M_1-\eta} }
\Big)   +\frac{1}{2\pi} \Big[  \Gamma_{11}(  2\eta-M_1 ) +\Gamma_{12}M_2
\Big]
\notag\\
&\Lra  
\frac{1}{ \sqrt{M_1-\eta} }-  \frac{1}{ \sqrt{M_2+\eta} }  =
   \frac{1}{2\pi\sigma_{01}\sqrt{\pi}} \Big[  \Gamma_{11}(  2\eta-M_1 ) +\Gamma_{12}M_2
\Big]
\notag\\
&\Lra  
\sqrt{M_2+\eta}-\sqrt{M_1-\eta}=
\frac{\sqrt{(M_1-\eta) (M_2+\eta) }}{2\pi\sigma_{01}\sqrt{\pi}} \Big[  \Gamma_{11}(  2\eta-M_1 ) +\Gamma_{12}M_2
\Big]
\notag\\
&\Lra  M_2+M_1 -2\sqrt{(M_1-\eta) (M_2+\eta) }
=
\frac{(M_1-\eta) (M_2+\eta) }{4\pi^3\sigma_{01}^2} \Big[  \Gamma_{11}(  2\eta-M_1 ) +\Gamma_{12}M_2
\Big]^2
\notag\\
&\Lra   4(M_1-\eta) (M_2+\eta) 
=\bigg[
\frac{(M_1-\eta) (M_2+\eta) }{4\pi^3\sigma_{01}^2} \Big[  \Gamma_{11}(  2\eta-M_1 ) +\Gamma_{12}M_2
\Big]^2-(M_2+M_1) \bigg]^2
\label{8th order algebraic equation}
\end{align}
As such, any solution of $f_1'(\eta)=0$ is also root of 
\eqref{8th order algebraic equation},
an 8th order polynomial,
where the coefficients depend only on $M_i$, $\Gamma_{ij}$, $\sigma_{ij}$. Therefore,
the smallest (positive) root $\eta_*$ depends only on $M_i$, $\Gamma_{ij}$, $\sigma_{ij}$,
and no other critical configuration with $0<\eta<\eta_*$ can exist.
	
		\item A single bubble (of type II constituent and mass $M_2-\zeta$), plus a  core-shell where the type I constituent (of mass $M_1$) is 
	forming the outer annulus and the type II constituent (of mass $\zeta$) is concentrated in a ball.
		The energy is thus
	\begin{align*}
	f_2(\eta) & = 2\sqrt{\pi} [  \sigma_{12} \sqrt{\eta} + \sigma_{01} \sqrt{M_1+\eta} 
	+\sigma_{02} \sqrt{M_2-\eta} 
	]  \\
	&\qquad+
	\frac{1}{4\pi} \Big[  \Gamma_{22}(  (M_2-\eta)^2+\eta^2 ) +2\Gamma_{12}M_1\eta
	+\Gamma_{11}M_1^2 \Big],
	\end{align*}
	and 
	\begin{align*}
	f_2'(\eta) & = \sqrt{\pi} \Big(  \frac{\sigma_{01}  }{ \sqrt{M_1+\eta} }  
	+	\frac{\sigma_{12}}{ \sqrt{\eta} }
	-\frac{\sigma_{02}}{ \sqrt{M_2-\eta} }
	\Big)   +\frac{1}{2\pi} \Big[  \Gamma_{22}(  2\eta-M_2 ) +\Gamma_{12}M_1
	\Big].
	\end{align*}
	Since
	\[ \lim_{\eta\to 0} f_2'(\eta) =+\8,\qquad 
	 \lim_{\eta\to M_2} f_2'(\eta) =-\8,\]
	 this means that there 
	 exists $\eta_0>0$, depending only on $M_i$, $\sigma_{ij}$, $\Gamma_{ij}$, such that
	 there
	 cannot be any stable configurations (and thus no optimal ones) if
	 $\eta<\eta_0$ or $\eta>M_2-\eta_0$. Consequently, this means that 
	 neither the single bubble, nor
	 any lobes of the core-shell, can have too small a mass.	 
\end{enumerate}

\end{enumerate}
Combining all the above cases concludes the proof.
\end{proof}


\begin{lemma} \label{one-type-single}
For any $M = (M_1, M_2), M_1, M_2 > 0$, when $\sigma_{02}   =  \sigma_{01}  + \sigma_{12}$ and $\Gamma_{12} = 0$, a minimizing configuration for $\overline{e_0}(M)$ can not have two different type of single bubbles.
\end{lemma}

\begin{proof}
Assume in a minimizing configuration for $\overline{e_0}(M)$, there are two single bubbles $(m_1^k, 0)$ and $(0, m_2^l)$. Then replacing these two single bubbles by a core shell $(m_1^k, m_2^l)$ will reduce the energy $ \sum_{k=1}^{\infty}  e_0 (m^k)$ since
\begin{eqnarray*}
 &&e_0 (m_1^k, 0 )  +  e_0 (0, m_2^l )  -   e_0 (m_1^k, m_2^l ) \\
&=& 2 \sigma_{01} \sqrt{\pi m_1^k}  +  \frac{\Gamma_{11} (m_1^k)^2 }{4 \pi} + 2 \sigma_{02} \sqrt{\pi m_2^l} + \frac{\Gamma_{22} (m_2^l)^2 }{4 \pi} \\
&& -  \left [ 2 \sigma_{01} \sqrt{\pi (m_1^k + m_2^l )} + 2 \sigma_{12} \sqrt{\pi m_2^l} +   \frac{\Gamma_{11} (m_1^k)^2 }{4 \pi} + \frac{\Gamma_{22} (m_2^l)^2 }{4 \pi}    \right ] \\
&=&   2 \sigma_{01} \sqrt{\pi } \left [ \sqrt{m_1^k} +  \sqrt{m_2^l} - \sqrt{ (m_1^k + m_2^l )} \right ] > 0 .
\end{eqnarray*}
This contradicts with the minimality.
\end{proof}

{\bf Existence of Core Shells}: Lemma \ref{one-type-single} implies that for any $M = (M_1, M_2), M_1, M_2 > 0$, when $\sigma_{02}   =  \sigma_{01}  + \sigma_{12}$ and $\Gamma_{12} = 0$, any minimizing configuration for $\overline{e_0}(M)$ must contain at least one generalized core shell.


\begin{remark}\rm 
In a binary system \cite{bi1} have shown that all bubbles are of equal size, by using the concavity of $e_{0}(m)$ with respect to the single mass parameter $m$. 
In simulations it appears that the same phenomenon should hold for double bubbles and core shells, that all components of a minimizer of the same type are congruent.  However, unlike the binary case, in ternary systems \eqref{e0m} a direct calculation of the Hessian shows it to be indefinite, and so the observed congruence of forms remains an open question.
\end{remark}


\section{The nonlocal problem in the droplet regime}
\label{nonlocalSec}

To see the effects of the nonlocal interaction we consider the Droplet rescaling, as in Choksi-Peletier (\cite{bi1}; see also \cite{ablw,ABCT2}.)  

\subsection{The Droplet Regime}

We take a periodic domain $\mathbb{T}^2$.
In the remainder of the paper, we assume that $\sigma_{ij}$ satisfies \eqref{triangle}, and so the perimeter term may be expressed in terms of total variation norms.
 Let  $u_i=\chi_{\Om_i}$, $i=1,2$.  As remarked upon earlier, while we must keep track of the perimeter of the exterior domain $\Om_0=( \overline{\Om_1\cup\Om_2 })^c$, the perimeter term is the same for its complement, $\widetilde\Om_0= \overline{\Om_1\cup\Om_2 }.$ Thus, it will be convenient to define
 $$   u_0 = \chi_{\Om_1\cup\Om_2},  $$
as it yields the same total variation,
$$  \int_{\mathbb{T}^2} |\nabla u_0|  = \int_{\mathbb{T}^2} |\nabla (1 -u_1- u_2)|
= \int_{\mathbb{T}^2} |\nabla \chi_{\Om_0}|.  $$
Thus, the local part of the functional may be expressed in terms of the two ``small'' sets $\Om_1,\Om_2$ and the constants $\beta_i\ge 0$, $i=0,1,2$ as:
\begin{eqnarray} P_\sigma(\Om_1,\Om_2)=
  \sum_{0 \leq i < j \leq 2} \sigma_{ij}   \mathcal{H}^1 (\partial \Omega_i \cap \partial \Omega_j)
= \frac{1}{2} \sum_{i=0}^2 \beta_i \int_{\mathbb{T}^2} |\nabla u_i | , \notag
\end{eqnarray}
with $\beta_i$ determined by \eqref{betasigma}, and so the energy becomes:
\begin{eqnarray} 
\mathcal{E} (u) =  \frac{1}{2} \sum_{i=0}^2 \beta_i \int_{\mathbb{T}^2} |\nabla u_i |+
  \sum_{i,j = 1}^2  \frac{ \gamma_{ij}}{2} \int_{\mathbb{T}^2} \int_{\mathbb{T}^2} G_{}(x-y)\; u_i (x) \; u_j (y) dx dy. \notag
\end{eqnarray}

\begin{remark}\label{equivrem}\rm
We note throughout that as long as the triangle inequalities \eqref{triangle} hold (even with $\beta_i=0$ in the Core Shell or Single Bubble cases), the weighted perimeter $P_\sigma(\Om_1,\Om_2)$ is equivalent to the standard norm for the cluster in $BV(\TT; \{0,1\})$.  That is, there exists a constant $C=C(\sigma_{ij})$ with 
$$    {1\over C} P_\sigma(\Om_1,\Om_2) \le \sum_{i=0}^2 \int_{\TT} |\nabla\chi_{\Om_i}|
\le C P_\sigma(\Om_1,\Om_2).
$$
In particular, the compactness and convergence proofs done for the unweighted case $\beta_i=1$, $\forall i=0,1,2$, in \cite{ablw} may be carried through to the weighted case with straightforward modifications.
\end{remark}

As in \cite{bi1, ablw},
we introduce a new parameter $\eta$ which is to represent the characteristic length scale of the droplet components.  Thus, areas scale as $\eta^2$, and so we choose mass constraints on $u=(u_1,u_2)$, 
\[   \int_{\mathbb{T}^2}  u_i = \eta^2 M_i \]
for some fixed $M_i$, $i=1, 2$.  We then rescale $u_i$ as
\begin{eqnarray}
v_{i, \eta}^{} = \frac{ u_i }{\eta^2}, \quad i=0,1,2, \quad\text{with}\quad { \int}_{\mathbb{T}^2}  v_{i, \eta}  =  M_i, \quad i=1,2,
\end{eqnarray}
and thus the appropriate space for droplet configurations is:
\begin{equation}\label{Xspace}
  X_\eta:=\left\{ v_{\eta}= (v_{1,\eta}, v_{2, \eta}) : \  \eta^2 v_{i,\eta}\in BV(\TT; \{0,1\}), i=1,2;  \ v_{1,\eta}v_{2, \eta}=0 \ a.e., \
\int _{\mathbb{T}^2}  v_{i, \eta}  =  M_i, \; i=1,2
\right\}, 
\end{equation}
We recall that we represent the exterior domain equivalently as $v_{0,\eta}=v_{1,\eta}+v_{2,\eta}$.

In terms of $v_\eta\in X_\eta$, the energy takes the form,
\begin{eqnarray} 
\mathcal{E} (u) \label{scaledEu}
&=& \eta \left (   \frac{\eta}{2} \sum_{i=0}^2 \beta_i \int_{\mathbb{T}^2} |\nabla v_{i,\eta} |  +   
\sum_{i,j = 1}^2  \frac{\gamma_{ij} \eta^3}{2} \int_{ \mathbb{T}^2} \int_{   \mathbb{T}^2 } v_{i, \eta}(x) v_{j, \eta}(y) dx dy \right )
\end{eqnarray}
Finally, we scale the interaction matrix $\gamma=[\gamma_{ij}]$  in such a way that both terms in \eqref{scaledEu} contribute at the same order in $\eta$.  This is accomplished by choosing
\begin{eqnarray}
\gamma_{ij} = \frac{1}{|\log \eta| \eta^3} \Gamma_{ij},   \nonumber
\end{eqnarray}
with fixed constants $\Gamma_{ij}\ge 0$.
Dividing by $\eta$, we thus obtain an $O(1)$ valued energy for  $v_\eta\in X_\eta$, defined by:
\begin{eqnarray} \label{Eeta}
  E_{\eta}^{} (v_{\eta}) := \frac{1}{\eta} \mathcal{E} ( u ) 
  =
 \frac{\eta}{2} \sum_{i=0}^2  \beta_i \int_{\mathbb{T}^2} |\nabla v_{i,\eta} | +  \sum_{i,j = 1}^2  \  \frac{  \Gamma_{ij} }{2 |\log \eta| } \int_{\mathbb{T}^2} \int_{\mathbb{T}^2} G_{\mathbb{T}^2}(x - y)  v_{i, \eta}(x) v_{j, \eta}(y) dx dy, \end{eqnarray}
and $E_\eta(v_\eta)=+\infty$ otherwise.  This is the droplet scaling of \cite{bi1},  which we adopt in the remainder of the paper.

\subsection{Expanding the energy in droplets}

The advantage of the droplet regime for nonlocal isoperimetric problems is that it separates length scales in the energy, so that both the local isoperimetric effects and the nonlocal repulsive interaction are both observed but at different scales in $\eta$. As a result of this balancing of the strengths of the competing terms, any finte energy configuration $v_\eta$ decomposes into an at most countable number of well-separated droplets, each of diameter $O(\eta)$.  

In the following we assume that $(v_\eta)_{\eta>0}$ are a family in $X_\eta$ with bounded energy, $\exists C>0$ with $E_\eta(v_\eta)\le C$, for all $\eta>0$.  Such configurations may be decomposed into indecomposable clusters (that is, connected in a measure-theoretic sense) of diameter of $O(\eta)$, for which the BV norm separates exactly.   
As in \cite{ablw} we may prove:
\begin{lemma}\label{components}
Assume \eqref{triangle} holds, and
let $v_\eta=\eta^{-2}\chi_{\Om_\eta}\in X_{\eta}$ with $\eta\int_{\TT} |\nabla v_{\eta} | \le C$.  Then, there exists an at most countable collection $v_{\eta}^k=\eta^{-2}\chi_{\Om_{\eta}^k}\in X_{\eta}$, with clusters $\Om_{\eta}^k=(\Om_{1,\eta}^k, \Om_{2,\eta}^k)$, such that 
\begin{enumerate}
\item[(a)]  $\Om_{i, \eta}^k\cap \Om_{j, \eta}^\ell=\emptyset$, for $k\neq\ell$ and $i,j=1,2$.
\item[(b)]  $v_\eta = \sum_{k=1}^\infty v_\eta^k$ in $X_{\eta}$; in particular,
$$  \int_{\TT} |\nabla v_{i, \eta}| = \sum_{k=1}^\infty \int_{\TT} |\nabla v_{i, \eta}^k|,
\quad i=0,1,2. $$
\item[(c)] There exists $C>0$ with $\diam(\Om_{\eta}^k) \le C\eta$ for all $k\in\NN$.
\end{enumerate}
\end{lemma}
\begin{proof}
Lemma~\ref{components} was proven in the case of equal weights in \cite{ablw}, and the same proof may be employed here.  However, it is also a consequence of the more general Decomposition Theorem \cite[Theorem 1]{ACMM01}, for finite perimeter sets in $\R^n$; we sketch the argument here for completeness.  Indeed, applying the Decomposition Theorem in \cite{ACMM01} to
$\Om_{0,\eta}=\bar \Om_{1,\eta}\cup \bar \Om_{2,\eta} $ 
for each $\eta>0$ we obtain an at most countable disjoint collection of finite perimeter sets,
$$  \Om_{0,\eta}=\bigcup_{k} \Om_{0,\eta}^k,
$$
for which each $\Om_{0,\eta}^k$ is indecomposable.  In two dimensions, the diameter of each indecomposable component is controlled by its perimeter, and so (c) holds for each $\Om^k_{0,\eta}$.  Using the $\{\Om^k_{0,\eta}\}$ to separate $\Om_{i,\eta}$ into essentially disjoint components $\Om_{i,\eta}^k=\Om_{i,\eta}\cap\Om^k_{0,\eta}\,$, $i=1,2$, the decomposition of the perimeters (b) then follows.  \end{proof}

\medskip

 In order to prove Gamma-convergence results and convergence of minimizers we need to expand the energy in terms of the indecomposable clusters $\Om_\eta^k=(\Om_{1,\eta}^k, \Om_{2,\eta}^k)$.  Since the diameter of the support $\Om_{\eta}^k$ is small compared to $\TT$, we may think of each $v_{i,\eta}^k$ as a function on $\R^2$.  This enables us to blow up each component at scale $\eta$ to determine the fine structure of the configuration $v_\eta$.  For each fixed $k\in\NN$ choose any point $\xi_\eta^k\in\Om_{0,\eta}^k\subset\TT$.  Then for  $i=1,2,$ we define a pair
$ z_{i}^k (x) : \mathbb{R}^2 \rightarrow \{0, 1\}${\,}, $i=1,2$, by
 \begin{eqnarray}
  z_{i}^k (x):=  \eta^2  v_{i, \eta}^k (\xi_\eta^k+\eta x)=\chi_{A_{i,\eta}^k}(x) , \quad i=1,2,
  \end{eqnarray}
 where $A_{i,\eta}^k = \eta^{-1}\left(\Om_{i,\eta}^k-\xi_\eta^k\right)\subset\R^2$, $i=1,2$, with $A^k_\eta=(A^k_{1,\eta}, A^k_{2, \eta})$ a finite perimeter 2-cluster in $\R^2$.
Then,
\begin{eqnarray}
 \int_{\mathbb{R}^2}  z_{i}^k = \int _{\TT}  v_{i,\eta}^k, \text{ and } \int_{\mathbb{R}^2} | \nabla z_{i}^k | =\eta \int _{\TT} | \nabla  v_{i,\eta}^k |.
\end{eqnarray}
 As above, we set $A^k_{0,\eta}=\bar A^k_{1,\eta}\cup \bar A^k_{2, \eta}\,$, which has the same perimeter as the exterior domain.
 
For a 2-cluster $A=(A_1,A_2)$, $A_i\subset\R^2$, $i=1,2$ of  finite perimeter, and $A_0=A_1\cup A_2$,  we define
$$  \GG (A) : = P_\sigma (A) +  \sum_{i,j=1}^2 \frac{\Gamma_{ij}}{4\pi} |A_i|\, |A_j| ,  $$
and for $m=(m_1,m_2)$,
\begin{align*}
  e_0( m )&:=\min   \left\{   \GG( A ) \ | \  A =(A_1,A_2) \text{ 2-cluster, with $| A_i  |= m_i $, $i=1,2$}\right \}
  \end{align*}
  By the results of Section 2, the minimum in $e_0(m)$ is attained for all $m=(m_1,m_2)$ nonzero, with the geometry of minimizers determined by the choice of weights $\sigma_{ij}$, a double bubble, core shell, or single bubble (which in particular occurs if either $m_i=0$, $i=1$ or $2$.)
We expect that the effect of the nonlocal interaction will be to enforce splitting of the mass into several droplets, but that energy minimization will determine the distribution of the droplet masses.  To this end, we define
  \begin{eqnarray}  \label{mine0}
 \overline{e_0 }(M ) := \inf \left\{ \sum_{k=1}^{\infty} e_0 (m^k ) :  m^k = (m_1^k, m_2^k ),  \ m_i^k \geq 0,\ \sum_{k=1}^{\infty}  m_i^k = M_i, i = 1, 2 \right\},
 \end{eqnarray}
where $M = (M_1, M_2)$.

Expanding $v_\eta$ into component clusters, and separating the on-diagonal terms in the double sum,
\begin{align} \nonumber
E_\eta(v_\eta)
&=\sum_{k=1}^{\infty} \sum_{i=0}^2 \frac{\eta}{2} \beta_i \int_{\TT}|\nabla v_{i, \eta}^k |
  +{\Gamma_{ij} \over 2 |\log \eta|} \sum_{k,\ell=1}^{\infty} \sum_{i,j=1}^2
     \int_{\TT}\int_{\TT} v_{i, \eta}^k (x)\, G_{\mathbb{T}^2} (x-y)\, v_{j, \eta}^\ell(y)\, dx\, dy \\   \nonumber
     & =  \sum_{k=1}^{\infty} \sum_{i=0}^2 \frac{\beta_i }{2}  \int_{\mathbb{R}^2}  |\nabla z_{i}^k| 
  +{\Gamma_{ij} \over 2|\log \eta|} \sum_{k,\ell=1}^{\infty} \sum_{i,j=1}^2
     \int_{ A_i^k}\int_{ A_j^\ell} G_{\mathbb{T}^2} (\xi^k_\eta+ \eta \tilde x - \xi^\ell_\eta - \eta \tilde y)\, d \tilde x\, d \tilde y \\    \nonumber
&= \sum_{k=1}^{\infty} \left[ P_\sigma(A^k_{\eta})  + \sum_{i,j=1}^2 {\Gamma_{ij}\over 4\pi} |A^k_{i,\eta}|\, |A^k_{j,\eta}| \right]+|\log\eta|^{-1} \Phi_\eta  \\
\label{firstlimit}
&= \sum_{k=1}^{\infty} \GG(A^k_\eta) + |\log\eta|^{-1} \Phi_\eta ,
\end{align}
where the remainder terms include a first piece depending on only the geometry of the indecomposable clusters, and a second piece containing off-diagonal terms which encodes the interactions between them,
\begin{align}\nonumber
   \Phi_\eta
 &= \sum_{k=1}^\infty  \sum_{i,j=1,2} {\Gamma_{ij}\over 2}
  \int_{A^k_{i,\eta}}\int_{A^k_{j,\eta}} \left( {1\over 2\pi}\log{1\over |x-y|} + R_{\TT} (\eta(x-y))\right) dx\, dy  \\
   \nonumber
     &\qquad
   + \sum_{k\neq\ell=1}^{\infty}  \sum_{i,j=1,2}  {\Gamma_{ij} \over 2}
     \int_{ A_{i,\eta}^k}\int_{ A_{j,\eta}^\ell} G_{\mathbb{T}^2} (\xi_\eta^k+ \eta \tilde x - \xi_\eta^\ell - \eta \tilde y)\, d \tilde x\, d \tilde y
     \\
      \label{interaction}
     &= : \Phi^1(\{A^k_\eta\}_{k\in\NN}) + \Phi^2_\eta(\{A^k_\eta\}_{k\in\NN},\{\xi^k_\eta\}_{k\in\NN})
\end{align}

\subsection{The first order limit}

The decomposition \eqref{firstlimit} induced by Lemma~\ref{components} suggests a first order Gamma limit for the functional $E_\eta$, in which the mass splits via $\overline e_0(M)$, encoding the geometric information from the weighted isoperimetric energy $\GG$.

We define a class of measures with countable support on $\TT$,
$$  Y:=\left\{ v_0=\sum_{k=1}^\infty (m_1^k,m_2^k)\, \delta_{x^k} \ | \  m_i^k\ge 0, \ x^k\in\TT \ \text{distinct points}\right\},
$$
and a functional on $Y$,
\begin{eqnarray}  \label{E0first}
E_0^{} (v_0 ) : = 
\left\{  
\begin{array}{rcl}
\sum_{k=1}^{\infty} \overline{e_0} (m^k ),  & &  \text{ if }  v_0\in Y, \\  
 \infty, & & \text{ otherwise. }   
 \end{array}  
\right.   
\end{eqnarray}  

As in \cite{ablw}, we have a first Gamma-convergence result:
 \begin{theorem}[First $\Gamma$-Limit] \label{firstGammalimit}
We have 
\begin{eqnarray}
E_{\eta}^{}  \overset{\Gamma}{\longrightarrow} E_0^{}  \  \text{ as } \ \eta \rightarrow 0. \nonumber
\end{eqnarray}
That is,
\begin{enumerate}
\item[(a)]  Let $ v_{\eta}\in X_\eta $ be a sequence with $\sup_{\eta>0} E_{\eta}^{}  (v_{\eta} )<\infty$.  Then there exists a subsequence $\eta\to 0$ and $v_0\in Y$ such that 
$ v_{\eta} \rightharpoonup v_{0} $ (in the weak topology of the space of measures),  and 
\begin{eqnarray}
\underset{\eta\rightarrow0}{\lim\inf} E_{\eta}^{}  (v_{\eta} ) \geq E_0^{} (v_{0}). \nonumber 
\end{eqnarray}
\item[(b)]  Let $v_0\in Y$ with $E_0^{} (v_{0}) < \infty$. Then there exists a sequence $ v_{\eta} \rightharpoonup v_{0} $ weakly as measures, such that 
\begin{eqnarray}
\underset{\eta\rightarrow0}{\lim\sup} E_{\eta}^{}  (v_{\eta} ) \leq E_0^{} (v_{0}). \nonumber 
\end{eqnarray}
\end{enumerate}
\end{theorem}

As the proof is very similar to that of \cite[Theorem 3.2]{ablw}, we provide only a sketch here.  
For the lower semicontinuity result (a), we take any sequence $(v_\eta)_{\eta>0}$ with uniformly bounded energy $E_\eta(v_\eta)\le C$, and 
 apply   \cite[Lemma 3.5]{ablw}.  Note that  the proof of Lemma~3.5 of \cite{ablw}, and of all the convergence results in that paper, were given for the case of equal weights $\beta_i=1$, $i=0,1,2$, but the proofs are identical for the weighted case, as long as the triangle inequalities \eqref{triangle} hold.  It states that, up to a subsequence $\eta\to 0$ (not relabelled,) there exist at most countably many 2-clusters $\{A^k\}$ in $\RR$ and corresponding points $\{x_{\eta}^k \}$ in $\TT$, such that the  domain 
 $\Om_\eta=\bigcup_{k\in\NN}\Om_\eta^k$ (with indecomposable clusters $\Om_\eta^k$,) satisfies:
\begin{equation}\label{CL12}  \left|  A^k \ \triangle  \,\left( \eta^{-1}\left[ \Om_{\eta}^k -x_{\eta}^k \right]\right) \right| = \left| A^k \ \triangle \, A^k_\eta\right|
   \xrightarrow{\eta_{} \rightarrow 0} 0, \qquad\forall k.
   \end{equation}
Moreover, the masses of the clusters are preserved in the limit,
\begin{equation}
   \label{CL13}
   M_i=\lim_{\eta_{} \to 0} \sum_{k=1}^\infty  \eta^{-2} |\Om_{i, \eta}^k| = \sum_{k=1}^\infty |A_i^k |, \quad i=1,2, 
   \end{equation}
and the total energy is bounded below by $\overline{e_0}(M)$. 
\begin{equation}
   \label{CL14}  \liminf_{\eta_{} \to 0} E_{\eta_{} }( v_{\eta_{}  }) \ge
       \sum_{k=1}^\infty \EE(A^k) \ge \overline{e_0}(M).
\end{equation}

At this point, note that \eqref{CL14} is the lower bound of the first Gamma convergence result, Theorem~\ref{firstGammalimit}.  The remainder of the Gamma convergence then  follows the proof of Theorem~3.2 in \cite{ablw}.  

\medskip

The geometrical structure can be made more precise for minimizers:  the limiting clusters after blow-up must minimize the weighted perimeter functional $\GG$ with the given mass distribution imposed by $\overline{e_0}(M)$.  Indeed, it is enough to assume that $E_\eta(v_\eta)\to \overline{e_0}(M)$ to obtain this conclusion:
\begin{proposition}\label{MinThm}
Let  $v_\eta^{\ast}=\eta^{-2}\chi_{\Om_\eta}\in X_\eta$ such that $\displaystyle\lim_{\eta\to 0}E_\eta=\overline{e_0}(M)$.  Then, there exists a subsequence $\eta \to 0$ (still denoted by $\eta$)  such that:
\begin{enumerate}
\item[(i)]  there exist connected clusters $(A^k)_{k\in\NN}$ in $\RR$ and points $x_{\eta}^k \in \TT$, $k\in\NN$, for which
\begin{equation} \label{MT1}
          \eta^{-2}\left| \Om_\eta \ \triangle \ \bigcup_{k=1}^\infty \left(\eta A^k + x_{\eta}^k \right)
                          \right| \xrightarrow{\eta \rightarrow 0} 0,
\end{equation}  
and
$$  M_i= \sum_{k=1}^\infty |A_i^k|;  $$
\item[(ii)]  each $A^k$ is a minimizer of $\GG$:  
\begin{equation}\label{MT2}
   \GG(A^k)= e_0(m^k), \qquad  m^k=(m_1^k, m_2^k)=|A^k|;
\end{equation}
Moreover,
\begin{equation}\label{MT3} 
   \overline{e_0}(M) = 
   \lim_{\eta\to 0} E_\eta(v_\eta)=\sum_{k=1}^\infty \GG(A^k)  =     
         \sum_{k=1}^\infty e_0(m^k).
   \end{equation}   
\item[(iii)]  If $k\neq\ell$, $d(x^k_\eta, x^\ell_\eta)\gg\eta$; more precisely, 
$$  \lim_{\eta\to 0} {|\log d(x^k_\eta, x^\ell_\eta)|\over |\log\eta|} =0, \quad \forall k\neq \ell.
$$
\end{enumerate}
\end{proposition}
\begin{remark}\label{finiteremark}\rm
In Corollary~\ref{finiteness} we prove that in the core shell case $\beta_1=0$, minimizers of $\overline{e_0}(M)$ may only have finitely many nontrivial clusters, $A^1,\dots, A^K$, for some $K\in\NN$.  This was also the case in the unweighted case,  and in the cases of single bubbles ($\beta_0=0$) as proven in \cite[Theorem 1.1]{ablw}.
\end{remark} 

The proof of Proposition~\ref{MinThm} follows exactly those of Lemmas~3.5 and 3.6 in \cite{ablw} (recalling from Remark~\ref{equivrem} that the weighted perimeter controls the relevant BV norms in the convergence steps.)

\subsection{Second order limit}

To obtain the Second Gamma Limit, Theorem~\ref{secondGammalimit}, and the more detailed description of limiting configurations, we must analyze the remainder term in \eqref{interaction}.  As the analysis in the case of strict triangle inequalities \eqref{triangle} (leading to minimizers with double bubbles) or in case $\beta_0=0$ (yielding only single bubbles) is essentially identical to the unweighted case of our previous paper \cite{ablw}, we will concentrate on the case $\beta_1=0$, in which core shells are favored.  

We subtract the first $\Gamma$-limit from $E_\eta$, which (from the decomposition \eqref{firstlimit}, \eqref{interaction}) we expect to have scale $|\log \eta|^{-1}$.   
For $v_\eta\in X_\eta$, let
\begin{eqnarray}\label{Feta}
F_{\eta}^{} (v_{\eta} )  : = |\log \eta | \left [ E_{\eta}^{} (v_{\eta} ) -   \overline{e_0 } ( M ) \right ].
\end{eqnarray}
For the second $\Gamma$-limit we consider $v_\eta\in X_\eta$ for which $F_{\eta}(v_\eta)$ is bounded.  For these $v_\eta$, Proposition~\ref{MinThm} applies, and we may thus assume that (along subsequences) the detailed concentration structure described by \eqref{CL12}, \eqref{CL13}, \eqref{CL14} is observed.  Given Lemma~\ref{lem:omega2}, each limiting cluster $A^k=(A^k_1,A^k_2)$ is either a core shell $\mathcal{C}^{m^k_1}_{m^k_2}$ (if $m_1^k,m_2^k>0$,) or a single bubble (in case one of $m_i^k=0$.)  We recall that minimization of $\GG$ does not determine the centers of core shells; this will be addressed below.

\medskip

Next we define the second order limit function, starting with its domain, which is suggested from the first order limit Theorem~\ref{firstGammalimit} and the structure described in Proposition~\ref{MinThm}.
For $K \in \mathbb{N}$, $m_1^k  \geq 0$, $m_2^k \geq 0$ and $m_1^k + m_2^k > 0$, the sequence $K \otimes  (m_1^k, m_2^k) $ is defined by
\begin{eqnarray*}
 ( K \otimes  (m_1^k, m_2^k) )^k : = \left \{
 \begin{aligned}
 &(m_1^k, m_2^k),&  1 \leq k \leq K, \\
& (0, 0),&  K+1 \leq k < \infty.
 \end{aligned}
 \right.
\end{eqnarray*}
Let ${\cal{M}}_{M}$ be the set of optimal sequences made of all clusters for the problem \eqref{mine0}:
\begin{eqnarray}
{\cal{M}}_{M} := \Bigg \{ K \otimes  (m_1^k, m_2^k) : K \otimes  (m_1^k, m_2^k)  \text{ minimizes  \eqref{mine0} for } M_i ,  i = 1, 2,  \nonumber \\
 \text{ and  } \  \overline{e_0} (m^k) = e_0^{} (m^k), \  m^k = (m_1^k, m_2^k )  \Bigg \}. \nonumber
\end{eqnarray}
Let $Y_M$ denote the space of all measures $v_0=\sum_{k=1}^K m^k \delta_{x^k}$ with $\{x^1,\dots,x^K\}$ distinct points in $\TT$ and $K\otimes m^k\in \mathcal{M}_M$.  That is, $v_0\in Y_M$ represents the limit (in the sense of distributions) of a sequence of energy minimizers $v_\eta$ of mass $M$, which obey the conclusions of Proposition~\ref{MinThm}.  We recall that the finiteness of these components is a consequence of Corollary~\ref{finiteness}.  

\medskip

We now describe the terms appearing in the second $\Gamma$-limit, which arise from passage to the limit in $\Phi_\eta$, defined in \eqref{interaction}.  First, the ``self-interaction'' terms making up $\Phi^1(\{A^k_\eta\}_{k\in\NN})$.   For $m=(m_1,m_2)$, $m_1,m_2\ge 0$ and not both zero, define
\begin{multline}\label{fzerodef}
f_0(m)= \inf\biggl\{\sum_{i,j=1,2} \frac{\Gamma_{ij} }{2}  \left [ 
\frac{1}{2\pi}  \int_{A_i}\int_{A_j}  \log { 1 \over |x-y|} dx\, dy
+ m_i\, m_j\, R_{\TT}(0)\right]  \ : \  \\
 \  \ A=(A_1,A_2) \ \text{minimizes $\GG$ with $|A_\ell|=m_\ell$, $\ell=1,2$}
\biggr\}
\end{multline}
The minimization here is only pertinent in case $m_1,m_2>0$, that is, where $A=\mathcal{C}^{m_1}_{m_2}$ is a core shell; for single bubbles the minimizers of $\GG$ with one of $m_i=0$ are unique up to rigid motion, and the set above is a singleton.  It is in this term that energy minimization of the nonlocal energy resolves the degeneracy in the local isoperimetric problem, in favor of {\it centered} core shells:
\begin{proposition}\label{concentric} Let $\beta_1=0$, and $m=(m_1,m_2)$ with $m_1,m_2>0$.
\begin{enumerate}
\item[(a)]
If $\Gamma_{11}>\Gamma_{12}$, then the minimum in $f_0(m)$ is attained by a concentric core shell $A=\mathcal{C}^{m_1}_{m_2}$.
\item[(b)]  If $\Gamma_{11}<\Gamma_{12}$, then the minimum in $f_0(m)$ is attained by a core shell $A=\mathcal{C}^{m_1}_{m_2}$ whose inner boundary circle is tangent to the outside circle.
 \end{enumerate}
\end{proposition}

\begin{proof}  Given $m_1,m_2>0$, from Lemma~\ref{lem:omega2} we know that all minimizers of $\GG$ with $m_1,m_2>0$ are core shells, that is they are composed of two (round) balls $B_1 = B_{r_1} (0)$, $B_2=B_{r_2}(p)$, with $r_1>r_2$, and the outer domain $A_1=B_1\setminus B_2$, with inner disk $A_2=B_2$. The distance between $p$ (the center of $B_2$) and $0$ must be small enough that $A_2\subset A_1$.  We must choose $p$ to minimize the quantity
$$   \sum_{i,j =1}^2   \frac{\Gamma_{ij} }{2}  \left[  I_{ij} + m_i m_j R_{\TT}(0)\right],    \quad\text{where}\quad  I_{ij} =\frac{1}{2\pi}  \int_{A_i} \int_{A_j}  \log { 1 \over |x-y|} dx\, dy .
$$
Notice that $I_{22}$ is independent of the location of $p$, as are of course the terms $m_im_jR_{\TT}(0)$.

For any $x \in \mathbb{R}^2 \setminus B_2$, by the Mean Value Theorem for harmonic functions,
\begin{eqnarray*}
 \int_{B_2}  \log { 1 \over |x-y|} dy = |B_2|  \log { 1 \over |x-p|}.
\end{eqnarray*}
Thus, 
\begin{align*} 
I_{12} &=\frac{m_2}{2\pi}  \int_{B_1 \setminus B_2}   \log { 1 \over |x-p|} dx = \frac{m_2}{2\pi}  \left [   \int_{B_1 }   \log { 1 \over |x-p|} dx - \int_{B_2 }   \log { 1 \over |x-p|} dx \right ]  \\
&=\frac{m_2}{2\pi}  \int_{B_1 }   \log { 1 \over |x-p|} dx - I_{22}.
 \end{align*}

We claim that $I_{12}=I_{12}(p)$ is maximized when $p=0$ and minimized when the circles are tangent.  By symmetry we may assume $p=(t,0)$, with $t\ge 0$ lying in an interval $t\in [0,\alpha]$ for which $A_2\subset A_1$.  Differentiating,
\begin{equation}\label{deriv}
   {d\over dt}I_{12}(t,0) = {1\over 2\pi}\int_{B_1} { x_1 - t\over |x-(t,0)|^2} dx.  
\end{equation}
Let $\omega_t=\{ x\in B_{{\color{red} 1} }\ | \ x_1\ge t\}$, and $\tilde\omega_t$ its reflection in the line $\{x_1=t\}$.  As the integrand in \eqref{deriv} is odd around the line $\{x_1=t\}$, and 
$B_1\setminus(\omega_t\cup\tilde\omega_t)\subset \{x_1<t\}$, we have $\forall t\in (0,\alpha]$,
$$  {d\over dt}I_{12}(t,0) = {1\over 2\pi}\int_{B_1\setminus(\omega_t\cup\tilde\omega_t)} { x_1 - t\over |x-(t,0)|^2} dx <0.
$$
Thus, $I_{12}(t,0)$ is even and strictly monotone decreasing in $t\in [0,\alpha]$, so it attains its maximum value at $t=0$ and minimum at the extreme value $t=\alpha$, where the circles are tangent.  This concludes the proof of the claim.
 

To complete the argument, we note that,
\begin{eqnarray} 
I_{11} &= & \frac{1}{2\pi}  \int_{A_1} \int_{A_1}  \log { 1 \over |x-y|} dx\, dy = \frac{1}{2\pi}  \int_{A_1} \int_{B_1}  \log { 1 \over |x-y|} dx\, dy - I_{12} \notag \\
&=&   - I_{12}  + \frac{1}{2\pi} \left [  \int_{B_1} \int_{B_1}  \log { 1 \over |x-y|} dx\, dy -  \int_{B_2} \int_{B_1}  \log { 1 \over |x-y|} dx\, dy \right ]   \notag \\
&=&   - 2 I_{12}  + \frac{1}{2\pi}  \int_{B_1} \int_{B_1}  \log { 1 \over |x-y|} dx\, dy  - I_{22}. \notag
 \end{eqnarray}
 As $I_{22}$ and $  \frac{1}{2\pi} \int_{B_1} \int_{B_1}  \log { 1 \over |x-y|} dx\, dy $ are independent of the location of $p$, we have
\begin{eqnarray}
 \sum_{i,j =1}^2    \frac{\Gamma_{ij} }{2}   I_{ij} &=& \frac{\Gamma_{11} }{2} (-2 I_{12}(p)) + \Gamma_{12} I_{12}(p) + \text{ terms independent of } p  \notag \\
&=& -  (\Gamma_{11} - \Gamma_{12}) I_{12}(p) + \text{ terms independent of } p.
\end{eqnarray}
The stated conclusions then follow, depending on the sign of the coefficient $(\Gamma_{11} - \Gamma_{12})$.
\end{proof}

We note that in the case $\Gamma_{11}<  \Gamma_{12} $ it may be difficult to observe tangential core shells, since that choice of strong repulsion between the two phases may result in core shells being split into single bubbles.


\medskip

With this satisfactory resolution of the question of the specific geometry of optimal core shells, we may now define the second $\Gamma$-limit functional, by including the formal limits of the two remainder terms from \eqref{interaction}.
For $v_0\in Y_M$, we define
\begin{equation}\label{secondlimit}
 F_0(v_0)= \sum_{k=1}^{K}   f_0(m^k)  
+   \sum_{k,\ell=1\atop k\neq\ell}^K \sum_{i,j =1}^2    \frac{\Gamma_{ij} }{2} m_i^k m_j^\ell G_{\mathbb{T}^2}(x_i^k - x_j^\ell ) ,
\end{equation}
and $F_0(v_0)=+\infty$ otherwise.  The first term in \eqref{secondlimit} is thus already completely determined by the first $\Gamma$-limit and the fine structure of core shells as determined by Proposition~\ref{concentric}, and only the spatial distribution of droplets is at play in $F_0$.

%

\begin{theorem}[Second $\Gamma$-limit]  \label{secondGammalimit}
We have 
\begin{eqnarray}
F_{\eta}^{}   \overset{\Gamma}{\longrightarrow} F_0^{}     \  \text{ as } \  \eta \rightarrow 0. \nonumber
\end{eqnarray}
That is,
\begin{enumerate}
\item[(a)]  Let $ v_{\eta}\in X_\eta $ be a sequence with $\sup_{\eta>0} F_{\eta}^{}  (v_{\eta} )<\infty$.  Then there exists a subsequence $\eta\to 0$ and $v_0\in Y_M$ such that 
$ v_{\eta} \rightharpoonup v_{0} $ (in the weak topology of the space of measures),  and 
\begin{eqnarray}
\underset{\eta\rightarrow0}{\lim\inf} F_{\eta}^{}  (v_{\eta} ) \geq F_0^{} (v_{0}). \nonumber 
\end{eqnarray}
\item[(b)]  Let $v_0\in Y_M$ with $F_0^{} (v_{0}) < \infty$. Then there exists a sequence $ v_{\eta} \rightharpoonup v_{0} $ weakly as measures, such that 
\begin{eqnarray}
\underset{\eta\rightarrow0}{\lim\sup} F_{\eta}^{}  (v_{\eta} ) \leq F_0^{} (v_{0}). \nonumber 
\end{eqnarray}
\end{enumerate}
\end{theorem}

The tools required for the proof of the second limit are essentially included in our paper \cite{ablw}, in the unweighted case.  We provide a slightly more detailed sketch here for the case $\beta_1=0$, as the paper \cite{ablw} concentrated on the first $\Gamma$-limit and the structure of minimizers and did not give details for the second $\Gamma$-limit.

\begin{proof}[Sketch of the proof of Theorem \ref{secondGammalimit}]
We begin with compactness and the lower limit.  Take a sequence $\eta_n\to 0$ and assume $v_{\eta_n}\in X_\eta$ is a sequence with bounded $F_{\eta_n}(v_{\eta_n})$.  (Recall $F_\eta$ is defined in \eqref{Feta}.)  For convenience, we abuse notation and denote this sequence (and all forthcoming subsequences) simply as $v_\eta$, and $\eta\to 0$ along this sequence (and eventually, certain subsequences.)  By the boundedness of $F_\eta(v_\eta)$ we immediately conclude that
$$   \lim_{\eta\to 0} E_\eta(v_\eta) = \overline{e_0}(M),  $$
and hence Lemma~\ref{components} and Proposition~\ref{MinThm} (Lemma~3.5 in \cite{ablw}) apply to the sequence $v_\eta$.  As a consequence, $v_\eta=\eta^{-2}\chi_{\Om_\eta}$ splits into at most countably many indecomposable clusters, $\Om_\eta=\bigcup_k \Om_\eta^k$, whose energy $E_\eta(v_\eta)$ breaks down via \eqref{firstlimit}.  Moreover, by Corollary~\ref{finiteness} in the core shell case $\beta_1=0$ minimizers of $\overline{e_0}(M)$ can only have a finite number $K\in\NN$ of nontrivial components. Hence,
\begin{align*}
E_\eta(v_\eta) - \overline{e_0}(M) 
  &=  
\sum_{k=1}^{K} \GG(A^k_\eta) - \overline{e_0}(M) 
+ |\log\eta|^{-1} \Phi_\eta(A_{\eta}^{{\color{red}k}},\xi^k_\eta) \\
&\ge \sum_{k=1}^K e_0(|A^k_\eta|) - \overline{e_0}(M) 
+ |\log\eta|^{-1} \Phi_\eta(A^k_\eta,\xi^k_\eta) \\
&\ge |\log\eta|^{-1} \Phi_\eta(A^k_\eta,\xi^k_\eta),
\end{align*}
since $\overline{e_0}(M)\le \sum_{k=1}^K e_0(m^k)$ whenever $\sum_{k=1}^K m^k=M$.  That is,  $\Phi_\eta(A^k_\eta,\xi^k_\eta)\le F_\eta(v_\eta)\le C$ are uniformly bounded as $\eta\to 0$.

Next, we pass to the limit in $\Phi^1(\{A^k_\eta\}_{k\in\NN})$ in \eqref{interaction}.  By \eqref{MT1}, for each $k$,  the clusters $|A^k_\eta \triangle A^k|\to 0$.  In addition, by Lemma~\ref{components} (c), the diameters diam$\,{A^k_\eta}\le C$ are uniformly bounded in $k,\eta$.  As $\log{1\over |x-y|}$ is locally integrable, we may then pass to the limit using dominated convergence,
\begin{equation}\label{Flim1}
\lim_{\eta\to 0} \Phi^1(\{A^k_\eta\}_{k\in\NN}) =
\sum_{i,j=1,2} \frac{\Gamma_{ij} }{2}  \left [ 
\frac{1}{2\pi}  \int_{A_i}\int_{A_j}  \log { 1 \over |x-y|} dx\, dy
+ m_i\, m_j\, R_{\TT}(0)\right]  \ge \sum_{k=1}^K  f_0(m^k).
\end{equation}

Since $\Phi^1$ converges, we conclude that the second interaction term $\Phi^2_\eta(\{A^k_\eta\}_{k\in\NN},\{\xi^k_\eta\}_{k\in\NN})$ is also bounded above.  To pass to the limit in this term, we first note that if for some pair $k\neq\ell$, $d(\xi^{k}_\eta,\xi^\ell_\eta)\ge \delta_{k,\ell}>0$ for all $\eta>0$, then by a similar argument as above we may pass to the limit in this $k,\ell$ term and obtain the corresponding term in $F_0$ (see \eqref{secondlimit}.)  We claim that this must be the case for all pairs $k\neq \ell$.  Indeed, we assume for some $k\neq\ell$ (and a subsequence $\eta\to 0$ that $r_\eta^{k,\ell}:=d(\xi^{k }_\eta,\xi^\ell_\eta)\to 0$, then recall that by Proposition~\ref{MinThm}~(iii) we must nevertheless have 
 for $i,j=1,2$,
\begin{align*}
\int_{ A_{i,\eta}^k}\int_{ A_{j,\eta}^\ell} G_{\mathbb{T}^2} (\xi_\eta^k+ \eta \tilde x - \xi_\eta^\ell - \eta \tilde y)\, d \tilde x\, d \tilde y 
&= \int_{ A_{i,\eta}^k}\int_{ A_{j,\eta}^\ell} \left[
   {1\over 2\pi}\log{1\over |\xi^k_\eta-\xi^\ell_\eta|} + 
R_{\mathbb{T}^2} (\xi_\eta^k+ \eta \tilde x - \xi_\eta^\ell - \eta \tilde y)\right] d \tilde x\, d \tilde y   \\
& = |A_{i,\eta}^k||A_{j,\eta}^\ell|\left( {1\over 2\pi}|\log r^{k,\ell}_\eta| + R_{\mathbb{T}^2}(0) + o(1)\right).
\end{align*}
Since $|A^k|>0$ for $k=1,\dots,K$, there exists $i,j\in\{1,2\}$ for which $|A^k_i||A^\ell_j|>0$, and that term will be unbounded above.  This contradicts the upper bound on $\Phi^2_\eta(\{A^k_\eta\}_{k\in\NN},\{\xi^k_\eta\}_{k\in\NN})$, and hence we conclude that the points $\{\xi^k_\eta\}_{k=1,\dots,K}$ must remain distinct as $\eta\to 0$.  Passing to a further subsequence if necessary, we may assume $\xi^k_\eta\to x^k\in \TT$, with distinct $x^k\neq x^\ell$, $k\neq\ell$, and 
$$  \liminf_{\eta\to 0} F_\eta(v_\eta) \ge \liminf_{\eta\to 0} \Phi_\eta(A^k_\eta,\xi^k_\eta)
   \ge F_0(v_0),  $$
   where $v_0=\sum_{k=1}^K m^k\delta_{x^k}\in Y_M$. 
 Note that (repeating an argument from above) by \eqref{MT1} and the uniform boundedness of the diameters of $A^k_\eta$, we may also conclude that $v_\eta\wra v_0$ in the weak topology on the space of measures on $\TT$.  
   This concludes the proof of  the compactness and lower limit of the second $\Gamma$-limit.
   
   The construction of a recovery sequence for part 2 is straightforward:  given points $x^k$ and masses $m^k=(m^k_1,m^k_2)$, it suffices to assemble a superposition of $\eta$-rescaled minimizers of $e_0(m^k)$ at each $x^k$.  When the minimizing cluster $A^k$ is a core shell, the central disk is positioned
   centered when $\Gamma_{12}$ is small according to Proposition~\ref{concentric}.
\end{proof}


\begin{thebibliography}{1}


\bibitem{ablw}  S. Alama, L. Bronsard, X. Lu and C. Wang. Periodic minimizers of a ternary non-local isoperimetric problem. \textit{Indiana Univ. Math. J.}, {70, 2557-2601,2021.}

%
\bibitem{ABCT1} S. Alama, L. Bronsard, Rustum Choksi, I. Topaloglu. Droplet breakup in the liquid drop model with background potential. \textit{Commun. Contemp. Math.}, 21(3), 1850022, 2019.
\bibitem{ABCT2} S. Alama, L. Bronsard, Rustum Choksi, I. Topaloglu. Droplet phase in a nonlocal isoperimetric problem under confinement. \textit{Comm. Pure Appl. Anal.}, 19, 175-202, 2020.
%
\bibitem{almgren} F.J. Almgren. Existence and regularity almost everywhere of solutions to elliptic variational problems with constraints. \textit{Mem. Amer. Math. Soc.}, 4(165), 1976.
\bibitem{ACMM01} Ambrosio, Luigi ;  Caselles, Vicent ;  Masnou, Simon ;  Morel, Jean-Michel . Connected components of sets of finite perimeter and applications to image processing.
 J. Eur. Math. Soc. (JEMS)  3  (2001),  no. 1, 39--92.
\bibitem{baldo} S. Baldo. Minimal Interface Criterion for Phase Transitions in Mixtures of Cahn-Hilliard Fluids. \textit{Ann. Inst. H. Poincar\'e Anal. Non Lin\'eaire}, 7(2): 67-90, 1990.
\bibitem{braides} A. Braides. Gamma Convergence for Beginners. Oxford Lecture Series in Mathematics and Its Applications, 22, 2002.
%
%
\bibitem{bi1} R. Choksi and M.A. Peletier, Small Volume Fraction Limit of the Diblock Copolymer Problem: I. Sharp Interface Functional. \textit{SIAM J. Math. Anal.}, 42(3):1334-1370, 2010. 
\bibitem{blendCR} R. Choksi and X. Ren. Diblock copolymer-homopolymer blends: derivation of a density functional theory. \textit{Physica D}. 203(1-2):100-119, 2005.
%
%
\bibitem{degiorgi} E. De Giorgi. Sulla Convergenza di Alcune Successioni D'integrali del Tipo Dell'area. \textit{Rend. Mat.}, 6(8): 277-294, 1975.
%
%
\bibitem{FABHZ} J. Foisy, M. Alfaro, J. Brock, N. Hodges, and J. Zimba. The standard double soap bubble in $\mathbb{R}^2$ uniquely minimizes perimeter. \textit{Pacific J. Math.}, 159(1):47-59, 1993.
%
%
%
\bibitem{db1} J. Hass and R. Schlafly. Double bubbles minimize. \textit{Ann. Math.}, 151(2):459-515, 2000.
\bibitem{db2} M. Hutchings, F. Morgan, M. Ritor\'e, and A. Ros. Proof of the double bubble conjecture. \textit{Ann. Math.}, 155(2):459-489, 2002.
%
%
%
%
%
%
%
%

\bibitem{Kraft} Kraft, D., Measure-Theoretic Properties of Level Sets of Distance Functions. J. Geom. Anal. 26 (2016), 2777–2796.

\bibitem{lawlor} Lawlor, Gary R. Double bubbles for immiscible fluids in 
$\R^n$. J. Geom. Anal. 24 (2014), no. 1, 190–204.

\bibitem{maggi} F. Maggi. Sets of finite perimeter and geometric variational problems: An introduction to geometric measure theory. New York: Cambridge University Press. 2012.
%
\bibitem{microphase} H. Nakazawa and T. Ohta. Microphase separation of ABC-type triblock copolymers. \textit{Macromolecules}, 26(20):5503-5511, 1993.
%
%
%
\bibitem{reichardt} B. Reichardt. Proof of the double bubble conjecture in $\mathbb{R}^n$. \textit{J. Geom. Anal.}, 18(1):172-191, 2008.
\bibitem{stationary}  X. Ren and C. Wang. A stationary core-shell assembly in a ternary inhibitory system. \textit{Discrete Contin. Dyn. Syst.}, 37(2):983-1012, 2017. 
\bibitem{disc} X. Ren and C. Wang. Stationary disk assemblies in a ternary system with long range interaction. \textit{Commun. Contemp. Math.}, 1850046, 2018.
\bibitem{miniRW} X. Ren and J. Wei. On the multiplicity of solutions of two nonlocal variational problems. \textit{SIAM J. Math. Anal.}, 31(4):909-924, 2000.
\bibitem{lameRW} X. Ren and J.Wei. Triblock copolymer theory: Ordered ABC lamellar phase. \textit{J. Nonlinear Sci.}, 13(2):175-208, 2003.
\bibitem{doubleAs}  X.Ren and J.Wei. A double bubble assembly as a new phase of a ternary inhibitory system. {\em Arch. Rat. Mech. Anal.}, 215(3):967-1034, 2015.
\bibitem{double} X. Ren and J. Wei. A double bubble in a ternary system with inhibitory long range interaction. \textit{Arch. Rat. Mech. Anal.}, 46(4):2798-2852, 2014.
\bibitem{smith1} C. Smith. Grains, phases and interfaces: an interpretation of Microstructure. \textit{Trans. AIME}, 175,15-51, 1948.
\bibitem{sternberg} P. Sternberg,  Vector-valued local minimizers of nonconvex variational problems.
Current directions in nonlinear partial differential equations (Provo, UT, 1987).
 Rocky Mountain J. Math.  21  (1991),  no. 2, 799--807.
\bibitem{taylor} J. Taylor. The structure of singularities in soap-bubble like and soap-film-like minimal surfaces.\textit{Ann. Math.}, 103(3):489-539, 1976.
%
\bibitem{wrz} C. Wang, X. Ren and Y. Zhao. Bubble assemblies in ternary systems with long range interaction. \textit{Comm. Math. Sci.}, 17(8): 2309-2324, 2019.
\bibitem{www} D. Wang, X. Wang and Y. Wang. The Dynamics of three-phase triple junction and contact points. \textit{SIAM J. Appl. Math.}, 77(5), 1805-1826, 2017.
\bibitem{xd} Z. Xu and Q. Du. On the ternary Ohta-Kawasaki free energy and its one dimensional global minimizers. 
 J. Nonlinear Sci.  32  (2022),  no. 5, Paper No. 61.
\bibitem{ynn} H. Yabu, S. Nagano and Y. Nagao. Core-shell cylinder (CSC) nanotemplates comprising mussel-inspired catechol-containing triblock copolymers for silver nanoparticle arrays and ion conductive channels. \textit{RSC Adv.},  8, 10627-10632, 2018.
\bibitem{mullins1} C. Yang, A. Rollett and W. Mullins. Measuring Relative Grain Boundary Energies and Mobilities in an Aluminum Foil from Triple Junction Geometry. \textit{Scripta Met.}, 44, 2735-40, 2001.
\bibitem{young} T. Young. An essay on the cohesion of fluids. \textit{Philos. Trans. R. Soc. London}, 95, 65–87, 1805.



\end{thebibliography}
\end{document}